\documentclass[11pt]{amsart}
\usepackage[left=2.5cm,right=2.5cm,top=2.5cm,bottom=2.5cm]{geometry}

\usepackage{amssymb,amsfonts}
\usepackage[all,arc]{xy}
\usepackage{enumerate}
\usepackage{mathrsfs}
\usepackage{amsmath}
\usepackage{amsmath,bm}
\usepackage{mathtools}
\usepackage{amsthm}
\usepackage{slashed}
\usepackage{amsmath} 
\usepackage{dsfont}
\usepackage{tikz}

\usetikzlibrary{babel}
\usepackage{graphicx}  
\usepackage{hyperref}  

\newtheorem{thm}{Theorem}[section]
\newtheorem{cor}[thm]{Corollary}
\newtheorem{prop}[thm]{Proposition}
\newtheorem{lem}[thm]{Lemma}

\theoremstyle{definition}
\newtheorem{defn}[thm]{Definition}

\newtheorem{exmp}[thm]{Example}

\newtheorem{notn}[thm]{Notation}

\newtheorem{obs}[thm]{Observation}

\theoremstyle{remark}
\newtheorem{rem}[thm]{Remark}

\def\MLine#1{\par\hspace*{-\leftmargin}\parbox{\textwidth}{\[#1\]}}

\usepackage{todonotes}

\def\Sr{\mathcal{S}_{\bullet}}
\def\Sri{\tilde{\mathcal{S}_{\bullet}}}

\def\Z{\mathbb{Z}}
\def\E{\mathbb{E}}
\def\P{\mathbb{P}}
\def\coc{c\text{-}occ}
\def\Asi{A_{\sigma,i}}
\def\leqsi{\preccurlyeq_{\sigma,i}}

\newcounter{indice}

\newcommand{\permutation}[1]{
\setcounter{indice}{0};
\foreach \i in {#1}
\addtocounter{indice}{1};

\addtocounter{indice}{1}
\draw [help lines] (1,1) grid (\theindice,\theindice);

\setcounter{indice}{1};

\foreach \i in { #1 } {
\draw (\theindice+.5,\i+.5) [fill] circle (.2);
\addtocounter{indice}{1};
}
\addtocounter{indice}{-1};

}

\title[Local convergence for permutations]{Local convergence for permutations and local limits for uniform $\rho$-avoiding permutations with $|\rho|=3$}

\author{Jacopo Borga}

\address{Institut für Mathematik, Universität Zürich, Winterthurerstrasse 190, CH-8057, Zürich}
\email{jacopo.borga@math.uzh.ch}



\begin{document}

\begin{abstract}
We set up a new notion of local convergence for permutations and we prove a characterization in terms of proportions of \emph{consecutive} pattern occurrences. We also characterize random limiting objects for this new topology introducing a notion of ``shift-invariant" property (corresponding to the notion of unimodularity for random graphs). 

We then study two models in the framework of random pattern-avoiding permutations. We compute the local limits of uniform $\rho$-avoiding permutations, for $|\rho|=3,$ when the size of the permutations tends to infinity. The core part of the argument is the description of the asymptotics of the number of consecutive occurrences of any given pattern. For this result we use bijections between $\rho$-avoiding permutations and rooted ordered trees, local limit results for Galton--Watson trees, the Second moment method and singularity analysis. 
\end{abstract}

\maketitle

\tableofcontents

\section{Introduction}

The goal of this article is to introduce a local topology for permutations and to study some initial interesting applications. 

\subsection{Scaling and local limits for discrete structures}
It is well-known that, after a suitable rescaling, a simple random walk on $\mathbb{Z}$ converges to the 1-dimensional Brownian motion: probably this is the most famous scaling limit result in probability theory. In the last thirty years such scaling limits have been intensively studied for different discrete structures. We mention some mathematical areas where this notion has been investigated, without aiming at giving a complete list.

In the framework of random trees and planar maps, the systematic study of scaling limits has been initiated by Aldous with the pioneering series of articles about the well-known \emph{Continuum random tree} (\cite{aldous1991continuum_1,aldous1991continuum_2,aldous1993continuum_3}). After that, many new results have been proven, in particular about the Brownian map, which is the scaling limit of random planar maps uniformly distributed over the class of all rooted $q$-angulations with $n$ faces, for $q=3$ and $q\geq4$ even integer (see for instance \cite{le2013uniqueness,marckert2006limit,miermont2013brownian}).

In statistical mechanics, the study of scaling limits of discrete models is also a very active topic. For instance, the scaling limit of the discrete Gaussian free field on lattices, \emph{i.e.,} the continuum GFF, has been proven to be the scaling limit of many others random structures (see for instance \cite{kenyon2001dominos,sheffield2007gaussian}). 

More closely related to this work, in \cite{hoppen2013limits} a notion of scaling limits for permutations has been recently introduced, called \emph{permutons}. They are probability measures on the unit square with uniform marginals, and they represent the scaling limit of the diagram of permutations as the size grows to infinity. This new notion of convergence has been studied (sometimes phrased in other terms) in several works. We mention a few of them.
\begin{itemize}
	\item The study of the scaling limits of a uniform random permutation avoiding a pattern of length three was initiated by Miner and Pak \cite{miner2014shape} and  Madras and Pehlivan \cite{madras2016structure}. Next, with a series of two articles, Hoffman, Rizzolo and Slivken (\cite{hoffman2017pattern,hoffman2017pattern_2}) strengthened these early results and understood many interesting phenomena that had previously gone unexplained. In particular they explored the connection of these uniform pattern-avoiding permutations to Brownian excursions. All these articles do not use the ``permuton language".
	\item The first concrete and explicit example of convergence in the ``permuton language" has been studied by Kenyon, Kral,  Radin and Winkler \cite{kenyon2015permutations}. They studied scaling limits of random permutations in which a finite number of pattern densities have been fixed.
	\item Bassino, Bouvel, F\'eray, Gerin and Pierrot \cite{bassino2018brownian} showed that a sequence of uniform random separable permutations of size $n$ converges to the \emph{Brownian separable permuton}. In this work they introduced this new limiting object that has been later investigated by Maazoun \cite{maazoun2017brownian}.
	\item In a second work, Bassino, Bouvel, F\'eray, Gerin, Maazoun and Pierrot \cite{bassino2017universal} showed that the Brownian permuton has a universality property: they consider uniform random permutations in proper substitution-closed classes and study their limiting behavior in the sense of permutons, showing that the limit is an elementary one-parameter deformation of the Brownian separable permuton.
	\item With a statistical mechanical approach , Starr \cite{starr2009thermodynamic} investigates the permuton limit of Mallows permutations. This is a non-uniform model where the probability of every permutation is proportional to $q^{\text{number of inversions}}$. Treating the question as a mean-field problem, he is able to calculate the distribution of the permuton limit (after a suitable rescaling of the parameter $q$).
	\item Rahman, Virag and Vizer \cite{rahman2016geometry} generalized the study of scaling limits for permutations to permutation valued processes, \emph{i.e.,} sequences of sequences of permutations.
	\item Permutons limits have been characterized in terms of convergence of frequencies of pattern occurrences (see \cite[Theorem 2.5]{bassino2017universal}). We will come back later to this characterization.
\end{itemize}

In parallel to scaling limits, a second notion of limits of discrete structures, on which this paper focuses, has been introduced: the local limits. Informally, scaling limits look at the convergence of the objects from a global
point of view (after a rescaling of the distances between points of the objects), while local limits look at discrete objects in a neighborhood of a distinguished point (without rescaling distances). As done for scaling limits, we allude to some frameworks where this notion has been studied, again without aiming at giving a complete overview.

Local limit results around the root of random trees were first implicitly proved by Otter \cite{otter1949multiplicative} and then explicitly by Kesten \cite{kesten1986subdiffusive} and Aldous and Pitman \cite{aldous1998tree}. Janson \cite{janson2012simply} gives a unified treatment of the local limits, as the number of vertices tends to infinity, of simply generated random trees. Recently, Stufler \cite{stufler2016local} studied also the local limits for large Galton--Watson trees around a uniformly chosen vertex building on previous results of Aldous \cite{aldous1991asymptotic} and Holmgren and Janson \cite{holmgren2017fringe}. 

Although implicit in many earlier works, the notion of local convergence around a random vertex (called \emph{weak local convergence}) for random graphs has been formally introduced by Benjamini and Schramm \cite{benjamini2001recurrence} and Aldous and Steele \cite{aldous2004objective}. One would expect that a limit object for this topology should ``look the same" when regarded from any of its vertices (because of the uniform choice of the root). This property is made precise by the notion of unimodular random rooted graph and the weak local limit of any sequence of random graphs is indeed proven to be unimodular (see \cite{benjamini2001recurrence}).  

Also in the framework of random planar maps the local convergence has been studied. For example, a well-known result is that the local limit for random infinite triangulations/quadrangulations is the uniform infinite planar triangulation/quadrangulation (UIPT/UIPQ) (see for instance \cite{angel2003uniform,krikun2005local,stephenson2018local}). 

In the framework of permutations, to the best of our knowledge, a local limit approach has not been investigated so far. The goal of the current paper is to fill this gap and show several interesting aspect of the local convergence for permutations. More precisely,
\begin{itemize}
	\item we prove a characterization of the local convergence in terms of proportions of \emph{consecutive} pattern occurrences (see Theorems \ref{thm_charact} and \ref{thm7}). For the sake of comparison, recall that permutons limits have been characterized in terms of convergence of frequencies of (non-consecutive) pattern occurrences.
	\item We characterize random limiting objects for this new topology (see Theorem \ref{shiftinvthm}), introducing a notion of ``shift-invariant" property (corresponding to the notion of unimodularity for random graphs).
	\item We exhibit some concrete applications of the developed theory in the setting of pattern-avoiding permutations (see Theorems \ref{thm_1} and \ref{thm_2}). 
\end{itemize}

\begin{rem}
	We wish to mention that Hoffman, Rizzolo and Slivken \cite{hoffman2016fixed} used local limits results for random trees in the study of some specific properties of permutations, like the number of fixed points (\emph{i.e.,} indexes $i$ such that $\sigma_i=i$) of pattern-avoiding permutations. Moreover, they do not study general local limits results for permutations.
\end{rem}

\begin{rem}
	We also mention that recently Pinsky \cite{pinsky2018infinite} studied limits of random permutations avoiding patterns of size three considering a different topology. His topology captures the local limit of the permutation diagram around a {\em corner}. The two topologies (that from \cite{pinsky2018infinite} and that from the present article) are not comparable. We point out that there are two main advantages to our definition of local limits: first, convergence for our local topology is pleasantly equivalent to the convergence of consecutive pattern proportions; second, many natural models have an interesting limiting object for our topology (while converging to $+\infty$ in the topology of \cite{pinsky2018infinite}).
\end{rem}

The reader not familiar with the terminology of permutation patterns (such as pattern occurrences, pattern avoidance, \emph{etc.}) can find the necessary background in Section \ref{notation}.

\subsection{Overview of our results}
Our work can be divided into two main parts.  

\subsubsection{Local convergence for permutations}
In the first part (Section \ref{loc_conv}) we develop the theory of local convergence for permutations. We will often view a permutation $\sigma$ as a diagram, \emph{i.e.,} the set of points of the Cartesian plane at coordinates $(j,\sigma_j)$ (see Fig.~\ref{rest} for an example). In the context of local convergence we need to look at permutations with a distinguished entry, called \emph{root}. Denoting with $\mathcal{S}^n$ the set of permutations of size $n$, we say that a pair $(\sigma,i)$ is a \emph{finite rooted permutation} of size $n$ if $\sigma\in\mathcal{S}^n$ and $i\in \{1,\dots,n\}$. 

To a rooted permutation $(\sigma,i)$, we associate a total order $\preccurlyeq_{\sigma,i}$ on a finite interval of integers containing 0 denoted by $A_{\sigma,i}.$
The total order $(A_{\sigma,i},\preccurlyeq_{\sigma,i})$
is simply obtained from the diagram of $\sigma$ (as shown in the left-hand side of Fig.~\ref{rest}). We shift the indices of the $x$-axis in such a way that the column containing the root of the permutation has index zero (the new indices are reported under the columns of the diagram). Then we set $j\preccurlyeq_{\sigma,i}k$ if the point in column $j$ is lower than the point in column $k.$

\begin{figure}[htbp]
	\begin{center}
		\includegraphics[scale=.67]{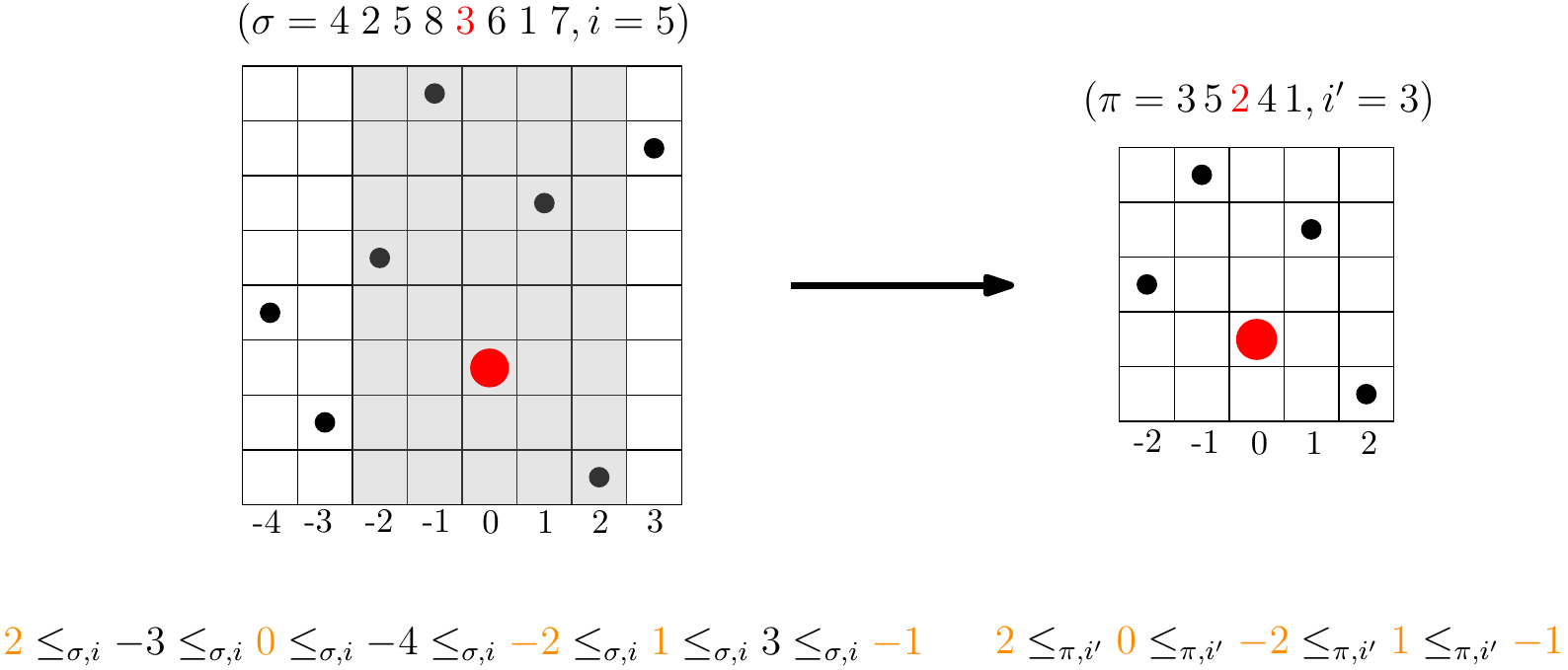}\\
		\caption{Two rooted permutations and the associated total orders. \label{rest}}
	\end{center}
\end{figure}

Since this correspondence defines a bijection, we identify every rooted permutation $(\sigma,i)$ with the total order $(A_{\sigma,i},\preccurlyeq_{\sigma,i})$.
Thanks to this identification, it is natural to call \emph{infinite rooted permutation} a pair $(A,\preccurlyeq)$ where $A$ is an infinite interval of integers containing 0 and $\preccurlyeq$ is a total order on $A$.

In order to define a notion of local convergence, we introduce a notion of $h$-restriction  around
the root. It can be thought of as the diagram of the pattern induced by a ``vertical strip" of width $2h+1$ around the root of the permutation (see Fig.~\ref{rest}) or equivalently of the restriction of the order $(A,\preccurlyeq)$ to $A\cap[-h,h].$ Note that this second equivalent characterization trivially applies also to infinite rooted permutations.

Then we say that a sequence $(A_{n},\preccurlyeq_{n})_{n\in\Z_{>0}}$ of rooted permutations is \emph{locally convergent} to a rooted permutation $(A,\preccurlyeq)$ if, for all $h\in\Z_{>0},$ the $h$-restrictions of the sequence $(A_{n},\preccurlyeq_{n})_{n\in\Z_{>0}}$ converge to the $h$-restriction of $(A,\preccurlyeq)$ (see Definition \ref{local_conv_def} for a formal statement).

We extend this notion of local convergence for rooted permutations to (unrooted) permutations, rooting them at a uniformly chosen index of $\sigma$. With this procedure, a fixed permutation $\sigma$ naturally identifies a random variable $(\sigma,\bm{i})$ that takes values in the
set of finite rooted permutations (here and throughout the paper we denote random quantities using \textbf{bold} characters).
In this way, the following notion of weak local convergence is natural: given a sequence of permutations $(\sigma^n)_{n\in\Z_{>0}}$, we say that $(\sigma^n)_{n\in\Z_{>0}}$ \emph{Benjamini--Schramm converges} to a random possibly infinite rooted permutation $(\bm{A},\bm{\preccurlyeq}),$ if the sequence $(\sigma^n,\bm{i}_n)_{n\in\Z_{>0}},$ where $\bm{i}_n$ is a uniform index in $[1,|\sigma^n|]$, converges in distribution to $(\bm{A},\bm{\preccurlyeq})$ with respect to the above defined local topology. We point out that our choice for the terminology comes from the analogous notion for graphs (see \cite{benjamini2001recurrence}).

We prove the following characterization in terms of proportions of consecutive pattern occurrences (for a formal and complete statement see Theorem \ref{bscharact}). For any pattern $\pi$ of size $k$, and any permutation $\sigma$ of size $n$, we denote by
$$\widetilde{\coc}(\pi,\sigma)=\frac{\coc(\pi,\sigma)}{n}=\frac{\text{number of consecutive occurrences of }\pi\text{ in }\sigma}{n},$$
the proportion of consecutive occurrences of $\pi$ in $\sigma.$ Moreover, we denote with $\mathcal{S}$ the set of permutations of finite size.

\begin{thm}
	\label{thm_charact}
	For any $n\in\Z_{>0}$, let $\sigma^n$ be a permutation of size $n$. Then the Benjamini--Schramm convergence
	for the sequence $(\sigma^n)_{n\in\Z_{>0}}$ is equivalent to the existence of non-negative real numbers
	$(\Delta_{\pi})_{\pi\in\mathcal{S}}$ such that, for all patterns $\pi\in\mathcal{S}$,
	$$\widetilde{\coc}(\pi,\sigma^n)\to\Delta_{\pi}.$$
\end{thm} 

\begin{rem}
	We could have introduced local convergence for permutations also from another point of view, closer to the one developed for graphons (see \cite{lovasz2012large}): we could say that a sequence of permutations is locally convergent if for every pattern $\pi\in\mathcal{S}$ the sequences $\widetilde{\coc}(\pi,\sigma^n)$ converge and then characterize the limiting objects. Theorem \ref{thm_charact} shows that the two approaches are equivalent.  
\end{rem}

What makes Theorem \ref{thm_charact} particularly interesting is the intense study of consecutive patterns. This is a very active field both from a combinatorial and a probabilistic point of view, with applications in computer science, biology, and physics. Classical statistics like the number of descents, runs and peaks contained in a permutation can be viewed as particular examples of consecutive patterns in permutations, but more general approaches have been developed, for example concerning the asymptotic behavior of the number of
permutations avoiding a consecutive pattern. For a complete overview, we refer to the wonderful survey on the topic by Elizalde \cite{elizalde2016survey}.

Once the Benjamini--Schramm convergence has been introduced for sequences of deterministic permutations, we then extend this notion  to sequences of \emph{random} permutations $(\bm{\sigma}^n)_{n\in\Z_{>0}}.$ The presence of two sources of randomness, one for the choice of the permutation and one for the (uniform) choice of the root, leads to two non-equivalent possible definitions: the \emph{annealed} and the \emph{quenched}  version of the Benjamini--Schramm convergence (see Definitions \ref{weakweakconv} and \ref{strongconv}). Intuitively, in the second definition, the random permutation is frozen, whereas  in the first one, the random permutation and the choice of the root are treated on the same level.

We obtain the following two additional characterizations (for formal and complete statements see Theorems \ref{weakbsequivalence} and \ref{strongbsconditions}).
\begin{thm}
	\label{thm7}
	For any $n\in\Z_{>0}$, let $\bm{\sigma}^n$ be a random permutation of size $n$. Then
	\begin{enumerate}[(a)]
		\item the annealed version of the Benjamini--Schramm convergence of $(\bm{\sigma}^n)_{n\in\Z_{>0}}$ is equivalent to the existence
		of non-negative real numbers $(\Delta_{\pi})_{\pi\in\mathcal{S}}$ such that 
		$$\E[\widetilde{\coc}(\pi,\bm{\sigma}^n)]\to\Delta_{\pi},\quad\text{for all patterns}\quad\pi\in\mathcal{S}.$$
		\item The quenched version of the Benjamini--Schramm convergence of $(\bm{\sigma}^n)_{n\in\Z_{>0}}$ is equivalent
		to the existence of non-negative real random variables $(\bm{\Lambda}_\pi)_{\pi\in\mathcal{S}}$ such that
		$$\big(\widetilde{\coc}(\pi,\bm{\sigma}^n)\big)_{\pi\in\mathcal{S}}\stackrel{(d)}{\to}(\bm{\Lambda}_{\pi})_{\pi\in\mathcal{S}},$$
		w.r.t. the product topology (where $\stackrel{(d)}{\to}$ indicates the convergence in distribution). 
	\end{enumerate}
	Obviously, the quenched version implies the annealed version.
\end{thm}

\begin{rem}
	The idea of defining and studying the quenched Benjamini--Schramm convergence is motivated by the theorem above. Indeed, one of the main goal for introducing the local topology was to obtain the characterization stated in point (b). We point out that the quenched Benjamini--Schramm convergence has been considered also for random trees (see for instance \cite{devroye2014protected}). The terminology ``annealed/quenched" is classical in statistical mechanics in the study of random systems on random environments.
\end{rem}

At this point we want to highlight a remarkable difference with the case of permuton convergence and its characterization mentioned before.
Indeed,  for the proportion of (non-consecutive) occurrences of patterns, the convergences in (a) and (b) are equivalent as shown in \cite[Theorem 2.5]{bassino2017universal}.
On the other hand, in Example \ref{notequivalent} we show that the annealed and quenched  version of the Benjamini--Schramm convergence are not equivalent.

Finally, we are also able to characterize random limiting objects for the annealed version of the Benjamini--Schramm convergence introducing a ``shift-invariant" property. For an order $(\Z,\preccurlyeq),$ its \emph{shift} $(\Z,\preccurlyeq')$ is defined by $i+1\preccurlyeq' j+1$ if and only if $i\preccurlyeq j.$ A random infinite rooted permutation, or equivalently a random total order, is said to be shift-invariant if it has the same distribution than its shift (for precise statements see Definition \ref{shiftinvariant}, Proposition \ref{easyimplication} and Theorem \ref{conversecostruction}).

\begin{thm}
	\label{shiftinvthm}
	A random infinite rooted permutation is ``shift-invariant" if and only if it is the local limit of a sequence of random permutations in the annealed Benjamini--Schramm sense. 
\end{thm}

In some sense, our ``shift-invariant" property corresponds to the notion of unimodularity for random graphs. As said before, it is well-known that the Benjamini--Schramm limit of a sequence of random graphs is unimodular. However, it is an open problem to determine if every unimodular random graph is the local limit of a sequence of finite random graphs rooted uniformly at random.

\subsubsection{Local limits for uniform $\rho$-avoiding permutations with $|\rho|=3$} In the second part of the article (Sections \ref{231} and \ref{321}), we demonstrate the relevance of this new notion of convergence. 
As first models for studying local convergence, we consider some classes of pattern avoiding permutations. This choice is motivated by the intense study of permutation classes in the last thirty/forty years. For a great survey we refer to Vatter \cite{vatter2014permutation}. In general, they have been deeply studied from a combinatorial point of view, finding the enumeration of specific classes. The most basic result in the field is that the number of $\rho$-avoiding permutations of size $n$, for $|\rho|=3,$ is given by the $n$-th Catalan number.

More recently a new probabilistic approach to the study of permutation classes has been investigated. An example, as explained in the previous section, is the study of scaling limits of uniform random permutations in a fixed pattern-avoiding class. Another example is the general problem of studying the limiting distribution of the number of occurrences (after a suitable rescaling) of a fixed pattern $\pi$ in a uniform random  permutation belonging to a fixed class when the size tends to infinity (see for instance Janson \cite{janson2018patterns,janson2017patterns,janson2017patterns321} where the author studied this problem in the model of uniform permutations avoiding a fixed family of patterns of size three).

In this work we focus on the classes of $\rho$-avoiding permutations, for $|\rho|=3,$ since several bijections (see for instance \cite{claesson2008classification}) with other combinatorial objects (such as rooted ordered trees and Dyck paths) are known and we expect some connection between the local limit convergence for trees and the local limit convergence for permutations. 

Our two main results of this second part are the two following theorems.
\begin{thm}
	\label{thm_1}
	For any $n\in\Z_{>0}$, let $\bm{\sigma}^n$ be a uniform random $231$-avoiding permutation of size $n$. Then
	we have the following convergence in probability,
	\begin{equation}
	\label{primaeq}
	\widetilde{\coc}(\pi,\bm{\sigma}^n)\stackrel{P}{\to}\frac{2^{\emph{LRMax}(\pi)+\emph{RLMax}(\pi)}}{2^{2|\pi|}},\quad\text{for all}\quad\pi\in\emph{Av}(231),
	\end{equation}
	where $\emph{LRMax}(\pi)$ (resp. $\emph{RLMax}(\pi)$) denotes the number of left-to-right maxima (resp. right-to-left maxima) in $\pi$.
\end{thm}

\begin{thm}
	\label{thm_2}
	For any $n\in\Z_{>0}$, let $\bm{\sigma}^n$ be a uniform random $321$-avoiding permutation of size $n$. Then
	we have the following convergence in probability for all $\pi\in\emph{Av}(321)$,
	\begin{equation}
	\label{blablahvyfdgwtrg5}
	\widetilde{\coc}(\pi,\bm{\sigma}^n)\stackrel{P}{\to}\begin{cases}
	\frac{|\pi|+1}{2^{|\pi|}} &\quad\text{if }\pi=12\dots|\pi|,\\
	\frac{1}{2^{|\pi|}} &\quad\text{if }\coc(21,\pi^{-1})=1, \\
	0 &\quad\text{otherwise.} \\ 
	\end{cases}
	\end{equation}
\end{thm}

\begin{obs}
	By symmetry, this covers all cases of $\rho$-avoiding permutations with $|\rho|=3.$ Indeed, every permutation of size three is in the orbit of either $231$ or $321$ by applying reverse (symmetry of the diagram w.r.t.\ the vertical axis) and complementation (symmetry of the diagram w.r.t.\ the horizontal axis). Beware that inverse (symmetry of the diagram w.r.t.\ the principal diagonal) cannot be used since it does not preserve consecutive pattern occurrences.
\end{obs}

Since the limits of
the random sequences $\big(\widetilde{\coc}(\pi,\bm{\sigma}^n)\big)_{n\in\Z_{>0}}$ are deterministic, for all patterns $\pi\in\mathcal{S}$, the convergence in probability in  (\ref{primaeq}) and (\ref{blablahvyfdgwtrg5}) is equivalent to the convergence in distribution of the vectors $\big(\widetilde{\coc}(\pi,\bm{\sigma}^n)\big)_{\pi\in\mathcal{S}}$ for the product topology. Therefore Theorems \ref{thm_1} and \ref{thm_2} trivially imply the characterization $(b)$ in Theorem~\ref{thm7} and so prove that the sequences $(\bm{\sigma}^n)_{n\in\Z_{>0}}$ converge for the quenched version, \emph{i.e.,} the stronger version, of the Benjamini--Schramm convergence. 

Here are some other interesting aspects of Theorems \ref{thm_1} and \ref{thm_2}.
\begin{itemize}
	\item A first important fact is the concentration phenomenon, namely the fact that the limits of the random sequences $\big(\widetilde{\coc}(\pi,\bm{\sigma}^n)\big)_{n\in\Z_{>0}}$ are deterministic, for all pattern $\pi\in\mathcal{S}$. Indeed Janson \cite[Remark 1.1]{janson2018patterns} notices that, in some classes, we have concentration for the (non-consecutive) pattern occurrences around their mean, in others not. It would be interesting to understand in a more general setting when this concentration phenomenon does or does not occur (we will come back to this fact in the next section).
	\item The second important fact is the different behavior of the two models of Theorem \ref{thm_1} and Theorem \ref{thm_2}: the first limiting density
	has full support on the space of 231-avoiding permutations, whereas the second gives positive measure only to 321-avoiding permutations whose inverse have at most one descent. 
	Indeed, despite having the same enumeration sequence, $\text{Av}(231)$ and $\text{Av}(321)$ are often considered in the community as behaving really differently. Our results give new evidence of this belief.
	\item Theorem \ref{thm7} ensures just the existence of the limiting random total orders on $\Z$ for the Benjamini--Schramm convergence, without providing an explicit construction. Nevertheless, in both cases (see Sections \ref{explcon} and \ref{explcon2}), we are able to provide an explicit construction of them.
\end{itemize}

Although the two theorems have very similar statements and in both models we use bijections between $\rho$-avoiding permutations and rooted ordered trees, the two proofs involve different techniques.
\begin{itemize}
	\item For the proof of Theorem \ref{thm_1} we use the Second moment method. We study the  asymptotic behavior of the first and the second moments of $\widetilde{\coc}(\pi,\bm{\sigma}^n)$ applying a technique introduced by Janson \cite{janson2017patterns,janson2003wiener}. Instead of studying uniform trees with $n$ vertices, we focus on specific families of binary Galton--Watson trees (which have some nice independence properties). Then we recover results for the first family of trees using singularity analysis for generating functions.
	\item For the proof of Theorem \ref{thm_2} we use a probabilistic approach for the study of local limits for Galton--Watson trees pointed at a uniform vertex. The bijection between trees and 321-avoiding permutations that we used, strongly depends on the position of the leaves. We therefore study the contour functions of some specific Galton--Watson trees in order to extract information about the positions of the leaves in the neighborhood of a uniform vertex. 
\end{itemize}
 
\subsection{Future projects and open problems}
We believe that this new notion of local convergence for permutations is worth being further investigated in several ways. We list our ideas for future projects.
\begin{itemize}
	\item Motivated by the work of Janson \cite{janson2018patterns} we would like to extend our results to uniform permutations avoiding multiple patterns. In the case of patterns of size three, this generalization does not contain any difficulty (due to the high level of independence for the values of a uniform permutation in these classes) and the explicit analysis of the problem has been developed in the bachelor thesis of Petrella under the supervision of Barbato at the University of Padua.
	\item We would like to generalize the proof of Theorem \ref{thm_2} for all uniform $\rho$-avoiding permutation when $\rho$ is an increasing or decreasing permutation of size greater than three.
	\item We studied the local convergence for uniform permutations in some classes. Another interesting direction could be to study non-uniform models. Motivated by the work of Crane, DeSalvo and Elizalde \cite{crane2018probability} we believe that the local limit of random permutations with Mallows distribution could be investigated. More generally, we would like to investigate different types of biased models.
	\item In a work with Bouvel, F\'eray and Stufler \cite{borga2018localsubclose} we investigate local limits for uniform permutations in substitution-closed classes. Also in this setting, the concentration phenomenon for the proportions of consecutive pattern occurrences occurs.
	\item A challenging problem (motivated by the previous results) is to find natural (but non-trivial) models where there is no concentration for the proportion of consecutive pattern occurrences (for a trivial model see Example \ref{notequivalent}). 
	
	{\em Note added in revision:} A first family of permutations where there is no concentration phenomenon has been recently investigated in \cite{borga2019square} by the author and Slivken. 
	\item In a more theoretical direction we would like to examine in depth the relationship between local convergence and permuton convergence (and other possible intermediate notion of convergence for permutations, as done for example in \cite{frenkel2018convergence} for graphs).
	\item As said before we are able to characterize random limiting objects for the annealed version of the Benjamini--Schramm convergence. It would be interesting to find a similar characterization for the stronger quenched version, where the limiting objects are random probability measures.
\end{itemize}

\subsection{Outline of the paper}
The paper is organized as follows:
\begin{itemize}
	\item Section \ref{loc_conv} deals with the general theory of local convergence for deterministic and random permutations;
	\item Section \ref{tree_notation} summarizes notation and results on trees and random trees that we need in the two following sections;
	\item Sections \ref{231} and \ref{321} contain the proof of Theorem \ref{thm_1} and Theorem \ref{thm_2} respectively and the constructions of the corresponding limiting objects.
\end{itemize}

\section{Local convergence for permutations}
\label{loc_conv}
\subsection{Permutations and patterns}
\label{notation}
We introduce some notation to be used in the sequel. For any $n\in\Z_{>0},$ we denote the set of permutations of $[n]=\{1,2,\dots,n\}$ by $\mathcal{S}^n.$ We write permutations of $\mathcal{S}^n$ in one-line notation as $\sigma=\sigma_1\sigma_2\dots\sigma_n.$ For a permutation $\sigma\in\mathcal{S}^n$ the \emph{size} $n$ of $\sigma$ is denoted by $|\sigma|.$ We let $\mathcal{S}\coloneqq\bigcup_{n\in\Z_{>0}}\mathcal{S}^n$ be the set of finite permutations. We write sequences of permutations in $\mathcal{S}$ as $(\sigma^n)_{n\in\Z_{>0}}.$ 

If $x_1\dots x_n$ is a sequence of distinct numbers, let $\text{std}(x_1\dots x_n)$ be the unique permutation $\pi$ in $\mathcal{S}^n$ that is in the same relative order as $x_1\dots x_n,$ \emph{i.e.}, $\pi_i<\pi_j$ if and only if $x_i<x_j.$
Given a permutation $\sigma\in\mathcal{S}^n$ and a subset of indices $I\subset[n]$, let $\text{pat}_I(\sigma)$ be the permutation induced by $(\sigma_i)_{i\in I},$ namely, $\text{pat}_I(\sigma)\coloneqq\text{std}\big((\sigma_i)_{i\in I}\big).$
For example, if $\sigma=87532461$ and $I=\{2,4,7\}$ then $\text{pat}_{\{2,4,7\}}(87532461)=\text{std}(736)=312$. 

Given two permutations $\sigma\in\mathcal{S}^n$ for some $n\in\Z_{>0}$ and $\pi\in\mathcal{S}^k$ for some $k\leq n,$ we say that $\sigma$ contains $\pi$ as a \textit{pattern} if $\sigma$ has a subsequence of entries order-isomorphic to $\pi,$ that is, if there exists a \emph{subset} $I\subset[n]$ such that $\text{pat}_I(\sigma)=\pi.$ In addition, we say that $\sigma$ contains $\pi$ as a \textit{consecutive pattern} if $\sigma$ has a subsequence of adjacent entries order-isomorphic to $\pi,$ that is, if there exists an \emph{interval} $I\subset[n]$ such that $\text{pat}_I(\sigma)=\pi.$ All the intervals considered in the article are to be understood as integer intervals, \emph{i.e.}, intervals contained in $\Z,$ and will be denoted by $[a,b],$ for $a,b\in\Z.$
\begin{exmp} 
The permutation $\sigma=1532467$ contains $1423$ as a pattern but not as a consecutive pattern and $321$ as consecutive pattern. Indeed $\text{pat}_{\{1,2,3,5\}}(\sigma)=1423$ but no interval of indices of $\sigma$ induces the permutation $1423.$ Moreover, $\text{pat}_{[2,4]}(\sigma)=\text{pat}_{\{2,3,4\}}(\sigma)=321.$
\end{exmp}
We say that $\sigma$ \emph{avoids} $\pi$ if $\sigma$ does not contain $\pi$ as a pattern. We point out that the definition of $\pi$-avoiding permutations refers to patterns and not to consecutive patterns. We denote by $\text{Av}^n(\pi)$ the set of $\pi$-avoiding permutations of size $n$ and by $\text{Av}(\pi)\coloneqq\bigcup_{n\in\Z_{>0}}\text{Av}^n(\pi)$ the set of $\pi$-avoiding permutations of arbitrary size.

We denote by $\coc(\pi,\sigma)$ the number of consecutive occurrences of a pattern $\pi$ in $\sigma.$ More precisely
\begin{equation*}
\label{cocc}
\coc(\pi,\sigma)\coloneqq\text{Card}\Big\{I\subseteq[n]\Big|\text{Card}(I)=|\pi|,\;I\text{ is an interval, } \text{pat}_I(\sigma)=\pi\Big\},
\end{equation*}
where $\text{Card}(\cdot)$ denotes the cardinality of a set.
Moreover, we denote by $\widetilde{\coc}(\pi,\sigma)$ the proportion of consecutive occurrences of a pattern $\pi$ in $\sigma,$ that is,
\begin{equation*}
\label{conpatden}
\widetilde{\coc}(\pi,\sigma)\coloneqq\frac{\coc(\pi,\sigma)}{n}.
\end{equation*}
\begin{rem}
The natural choice for the denominator of the previous expression should be $n-k+1$ and not $n,$ but we make this choice for later convenience (for example, in order to give a probabilistic interpretation of the quantity $\widetilde{\coc}(\pi,\sigma)$ as done in Equation (\ref{probinterpret}, p.\pageref{probinterpret}). Moreover, for every fixed $k,$ there are no difference in the asymptotics when $n$ tends to infinity. 
\end{rem}

\subsection{The set of rooted permutations}
In the first part of this section we define the notion of finite and infinite rooted permutation. Then we introduce a local distance and at the end we show that the space of (possibly infinite) rooted permutations is the natural space to study local limits of permutations. More precisely, we show that this space is compact and it contains the space of finite rooted permutations as a dense subset.

For the reader convenience, we recall the following fundamental definition.
\begin{defn}
A \emph{finite rooted permutation} is a pair $(\sigma,i),$ where $\sigma\in\mathcal{S}^n$ and $i\in[n]$ for some $n\in\Z_{>0}.$
\end{defn}
We extend the notion of size of a permutation to rooted permutations in the natural way, \emph{i.e.,} $|(\sigma,i)|\coloneqq|\sigma|.$
We denote  with $\mathcal{S}^n_{\bullet}$ the set of rooted permutations of size $n$ and with $\mathcal{S}_{\bullet}\coloneqq\bigcup_{n\in\Z_{>0}}\mathcal{S}^n_{\bullet}$ the set of finite rooted permutations. We write sequences of finite rooted permutations in $\mathcal{S}_{\bullet}$ as $(\sigma^n,i_n)_{n\in\Z_{>0}}.$ 

To a rooted permutation $(\sigma,i),$ we associate the pair $(\Asi,\leqsi),$  where $\Asi\coloneqq[-i+1,|\sigma|-i]$ is a finite interval containing 0 and $\leqsi$ is a total order on $\Asi,$ defined for all $\ell,j\in \Asi$ by
\begin{equation*}
\ell\leqsi j\qquad\text{if and only if}\qquad \sigma_{\ell+i}\leq\sigma_{j+i}\;.
\end{equation*}  
Clearly this map is a bijection from the space of finite rooted permutations $\mathcal{S}_{\bullet}$ to the space of total orders on finite integer intervals containing zero. Consequently and throughout the paper, we identify every rooted permutation  $(\sigma,i)$ with the total order $(\Asi,\leqsi).$

We want to highlight in the following example that the total order $(\Asi,\leqsi)$ associated to a rooted permutation indicates how the elements of $\sigma$ at given positions compare to each other. 

In most examples, we draw permutations as diagrams. We recall that the diagram of a permutation $\sigma$ of size $n$ is the set of points of the Cartesian plane at coordinates $(j,\sigma_j),$ for $1\leq j\leq n$. For a rooted permutation $(\sigma,i)$ we draw the root with a bigger red dot at coordinates $(i,\sigma_i).$

\begin{exmp}
	\label{firstpartexmop}
	 The diagram of the rooted permutation $(\sigma,i)=(752934861,4)$ is provided in the left-hand side of Fig.~\ref{exampleord}. The associated total order is $(\Asi,\leqsi)=\big([-3,5],\preccurlyeq\big),$ where $$5\preccurlyeq-1\preccurlyeq1\preccurlyeq2\preccurlyeq-2\preccurlyeq4\preccurlyeq-3\preccurlyeq
	3\preccurlyeq0.$$ 
\end{exmp}

We now make an observation that will be useful for the proofs of Proposition \ref{easyimplication} and Theorem \ref{conversecostruction}.

\begin{obs}
	A total order $(\Asi,\leqsi)$ (associated to a rooted permutation $(\sigma,i)$) shifted to the interval $[1,|\sigma|]$ seen as a word coincides with the inverse permutation $\sigma^{-1}.$
	
	For instance, the total order in Example \ref{firstpartexmop} shifted on the interval $[1,9]$ gives $9\preccurlyeq3\preccurlyeq5\preccurlyeq6\preccurlyeq2\preccurlyeq8\preccurlyeq1\preccurlyeq
	7\preccurlyeq4.$ The word $935628174$ corresponding to it coincides with the inverse permutation $\sigma^{-1}.$
\end{obs}

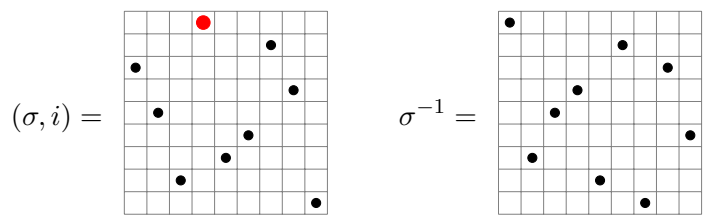
\begin{figure}[h]
	\begin{equation*}
	(\sigma,i)=
	\begin{array}{lcr}
	\begin{tikzpicture}
	\begin{scope}[scale=.3]
	\permutation{7,5,2,9,3,4,8,6,1} 
	\draw (4+.5,9+.5) [red, fill] circle (.3); 
	\end{scope}
	\end{tikzpicture}
	\end{array}
	\qquad\sigma^{-1}=
	\begin{array}{lcr}
	\begin{tikzpicture}
	\begin{scope}[scale=.3]
	\permutation{9,3,5,6,2,8,1,7,4} 
	\end{scope}
	\end{tikzpicture}
	\end{array}
	\end{equation*}
	\caption{Diagram of the rooted permutation $(\sigma,i)=(752934861,4)$ on the left. The inverse $\sigma^{-1}=935628174$ of $\sigma$ is shown on the right. \label{exampleord}}
\end{figure}

Thanks to the identification between rooted permutations and total orders, the following definition of an infinite rooted permutation is natural. 

\begin{defn}
We call \textit{infinite rooted permutation} a pair $(A,\preccurlyeq)$ where $A$ is an infinite interval of integers containing 0 and $\preccurlyeq$ is a total order on $A$. We denote the set of infinite rooted permutations by $\mathcal{S}^{\infty}_\bullet.$
\end{defn} 
We highlight that infinite rooted permutations can be thought of as rooted at 0. We set 
$$\tilde{\mathcal{S}_{\bullet}}\coloneqq\Sr\cup\mathcal{S}^\infty_{\bullet},$$
namely, the set of finite and infinite rooted permutations and we denote by $\Sr^{\leq n}\coloneqq\bigcup_{m\leq n}\Sr^m$ the set of rooted permutations with size at most $n.$ We write sequences of finite or infinite rooted permutations in $\Sri$ as $(A_n,\preccurlyeq_n)_{n\in\Z_{>0}}.$

 We now introduce the following \textit{restriction function around the root} defined, for every $h\in\Z_{>0}$, as follow
\begin{equation}
\label{rhfunct}
\begin{split}
  r_h \colon\quad &\tilde{\mathcal{S}_{\bullet}}\;\longrightarrow \qquad \;\mathcal{S}_\bullet\\
  (A,&\preccurlyeq) \mapsto \big(A\cap[-h,h],\preccurlyeq\big)\;.
\end{split}
\end{equation}
We can think of restriction functions as a notion of neighborhood around the root.
For finite rooted permutations we also have the equivalent description of the restriction functions $r_h$ in terms of consecutive patterns: if $(\sigma,i)\in\Sr$ then $r_h(\sigma,i)=(\text{pat}_{[a,b]}(\sigma),i)$ where $a=\max\{1,i-h\}$ and $b=\min\{|\sigma|,i+h\}.$

In the next example we give a graphical interpretation of the restriction function around the root.
\begin{exmp} We continue Example \ref{firstpartexmop} with the rooted permutation $(\sigma,i)=(752934861,4).$ When we consider the restriction $r_h(\sigma,i),$ we draw in gray a vertical strip ``around" the root of width $2h+1$ (or less if we are near the boundary of the diagram). An example is provided in Fig.~\ref{examplerootedperm} where $r_2(\sigma,i)$ is computed. In particular $r_2(\sigma,i)=\big([-2,2],\preccurlyeq\big),$ with $-1\preccurlyeq1\preccurlyeq2\preccurlyeq-2\preccurlyeq0.$

\begin{figure}[h]
\begin{equation*}
\begin{array}{lcr}
\begin{tikzpicture}
\begin{scope}[scale=.3]
\permutation{7,5,2,9,3,4,8,6,1}
\fill[gray!60!white, opacity=0.4] (2,1) rectangle (7,10); 
\draw (4+.5,9+.5) [red, fill] circle (.3); 
\end{scope}
\end{tikzpicture}
\end{array}
\stackrel{r_2}{\longrightarrow}
\begin{array}{lcr}
\begin{tikzpicture}
\begin{scope}[scale=.3]
\permutation{4,1,5,2,3}
\draw (3+.5,5+.5) [red, fill] circle (.3); 
\end{scope}
\end{tikzpicture}
\end{array}
\end{equation*}
\caption{Diagram of the permutation $\sigma=752934861$ rooted at $i=4$ with the corresponding restriction $r_2(\sigma,i).$}
\label{examplerootedperm}
\end{figure}
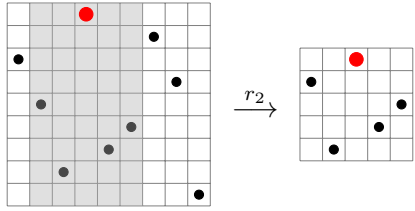

\end{exmp}

We now introduce a last definition.
\begin{defn}
	We say that a family $(A_h,\preccurlyeq_h)_{h\in\Z_{>0}}$ of elements in $\Sr$ is \emph{consistent} if
	\begin{equation*}
	\label{consprop}
	r_{h}(A_{h+1},\preccurlyeq_{h+1})=(A_h,\preccurlyeq_h),\quad \text{ for all }\quad h\in\Z_{>0}.
	\end{equation*}
\end{defn}

We have the following.

\begin{obs}
	\label{consistentobs}
	For every infinite rooted permutation $(A,\preccurlyeq),$ the corresponding family of restrictions $\big(r_h(A,\preccurlyeq)\big)_{h\in\Z_{>0}}$ is consistent.
\end{obs}

In the next example, we exhibit a concrete infinite rooted permutation, with some of its restrictions and the associated rooted permutations. 
\begin{exmp}
	We consider the following total order on $\mathbb{Z}:$
	\begin{equation*}
	\label{extotord}
	-2\preccurlyeq-4\preccurlyeq-6\preccurlyeq\dots\preccurlyeq-1\preccurlyeq-3\preccurlyeq-5
	\preccurlyeq\dots\preccurlyeq0\preccurlyeq1\preccurlyeq2\preccurlyeq3\preccurlyeq\dots,
	\end{equation*}
	that is, the standard order on the integers except that the order on negative numbers is reversed and negative even numbers are smaller than negative odd numbers. 
	Using the definition given in Equation $(\ref{rhfunct}),$ we have for $h=4,$
	\begin{equation*}
	r_4(\mathbb{Z},\preccurlyeq)=\big([-4,4],\preccurlyeq\big),\quad\text{with}\quad -2\preccurlyeq-4\preccurlyeq-1\preccurlyeq-3\preccurlyeq0\preccurlyeq1\preccurlyeq2\preccurlyeq3\preccurlyeq4,
	\end{equation*}
	which represents the rooted permutation $(\pi^4,5),$ for $\pi^4=241356789,$ whose diagram is represented in the left-hand side of Fig.~\ref{rest_exemp}.
	In particular, $r_4(\mathbb{Z},\preccurlyeq)=(A_{\pi^4,5},\preccurlyeq_{\pi^4,5}).$ 
	
	Moreover, for $h=3,$
	\begin{equation*}
	r_3(\mathbb{Z},\preccurlyeq)=\big([-3,3],\preccurlyeq\big),\quad\text{with}\quad-2\preccurlyeq-1\preccurlyeq-3\preccurlyeq0\preccurlyeq1\preccurlyeq2\preccurlyeq3,
	\end{equation*} 
	which represents the rooted permutation $(\pi^3,4),$ for $\pi^3=3124567,$ whose diagram is represented in the right-hand side of Fig.~\ref{rest_exemp}.
	In particular, $r_3(\mathbb{Z},\preccurlyeq)=(A_{\pi^3,4},\preccurlyeq_{\pi^3,4}).$ We also note that $r_3(\pi^4,5)=(\pi^3,4)$ as shown in Fig.~\ref{rest_exemp}.
\begin{figure}[htbp]
	\begin{center}	
	\begin{equation*}
	(\pi^4,5)=
	\begin{array}{lcr}
	\begin{tikzpicture}
	\begin{scope}[scale=.3]
	\permutation{2,4,1,3,5,6,7,8,9}
	\fill[gray!60!white, opacity=0.4] (2,1) rectangle (9,10); 
	\draw (5+.5,5+.5) [red, fill] circle (.3); 
	\end{scope}
	\end{tikzpicture}
	\end{array}
	\stackrel{r_3}{\longrightarrow}
	\begin{array}{lcr}
	\begin{tikzpicture}
	\begin{scope}[scale=.3]
	\permutation{3,1,2,4,5,6,7}
	\draw (4+.5,4+.5) [red, fill] circle (.3); 
	\end{scope}
	\end{tikzpicture}
	\end{array}=(\pi^3,4).
	\end{equation*}	
	\end{center}
\caption{The $r_3$-restriction of the rooted permutation $(\pi^4,5).$ \label{rest_exemp}}
\end{figure}
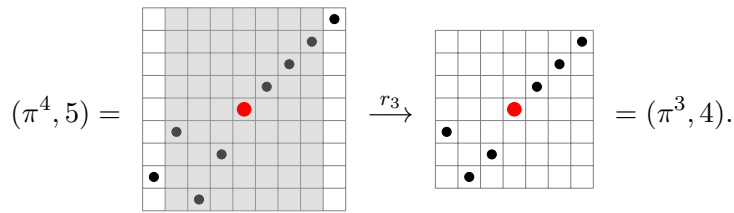
\end{exmp}

\subsection{Local convergence for rooted permutations}
We are now ready to define a notion of \emph{local distance} on the set of (possibly infinite) rooted permutations $\tilde{\mathcal{S}_{\bullet}}$. Given two rooted permutations $(A_1,\preccurlyeq_1),(A_2,\preccurlyeq_2)\in\tilde{\mathcal{S}_{\bullet}},$ we define
\begin{equation}
\label{distance}
d\big((A_1,\preccurlyeq_1),(A_2,\preccurlyeq_2)\big)\coloneqq 2^{-\sup\big\{h\in\Z_{>0}\;:\;r_h(A_1,\preccurlyeq_1)=r_h(A_2,\preccurlyeq_2)\big\}},
\end{equation}
with the classical conventions that $\sup\emptyset=0,$ $\sup\Z_{>0}=+\infty$ and $2^{-\infty}=0.$ It is a basic exercise to check that $d$ is a distance. Actually, $d$ is ultra metric, and so all the open balls of radius $2^{-h},$ $h\in\Z_{>0},$ centered in $(A,\preccurlyeq)\in\Sri$ -- which we denote by $B\big((A,\preccurlyeq),2^{-h}\big)$ -- are closed and intersecting balls are contained in each other. We will say that the balls are \emph{clopen}, \emph{i.e.,} closed and open.

Once we have a distance, the following notion of convergence is very natural. 
\begin{defn}
	\label{local_conv_def}
We say that a sequence $(A_n,\preccurlyeq_n)_{n\in\Z_{>0}}$ of rooted permutations in $\Sri$ is \emph{locally convergent} to an element $(A,\preccurlyeq)\in\Sri,$ if it converges with respect to the local distance $d.$ In this case we write $(A_n,\preccurlyeq_n)\stackrel{loc}{\longrightarrow}(A,\preccurlyeq).$
\end{defn}


We are now going to show that the space $\tilde{\mathcal{S}_{\bullet}}$ is the right space to consider in order to study the limits of finite rooted permutations with respect to the local distance $d$ defined in Equation (\ref{distance}). More precisely, as explained at the beginning, we show that $\tilde{\mathcal{S}_{\bullet}}$ is compact and that $\Sr$ is dense in $\tilde{\mathcal{S}_{\bullet}}.$ 

The latter assertion is trivial since 
\begin{equation}
d\big((A,\preccurlyeq),r_h(A,\preccurlyeq)\big)\leq 2^{-h},\quad\text{for all}\quad (A,\preccurlyeq)\in \tilde{\mathcal{S}_{\bullet}},\text{ all } h\in\Z_{>0},
\label{keydistprop}
\end{equation}
and obviously $r_h(A,\preccurlyeq)\in\Sr.$

Before proving compactness we explore some basic but important properties of our local distance $d$. 
The next proposition gives a converse to Observation \ref{consistentobs}.

\begin{prop}
	\label{consistprop}
	Given a consistent family $(A_h,\preccurlyeq_h)_{h\in\Z_{>0}}$ of elements in $\Sr$ there exists a unique (possibly infinite) rooted permutation $(A,\preccurlyeq)\in\Sri$ such that
	\begin{equation*}
	r_{h}(A,\preccurlyeq)=(A_h,\preccurlyeq_h),\quad \text{ for all }\quad h\in\Z_{>0}.
	\end{equation*}
	Moreover, $(A_h,\preccurlyeq_h)\stackrel{loc}{\longrightarrow}(A,\preccurlyeq).$
\end{prop}

\begin{proof} 
	We explicitly construct the rooted permutation $(A,\preccurlyeq).$ We set $A\coloneqq\bigcup_{h\in\Z_{>0}}A_{h}.$ From the consistency property of the family $(A_h,\preccurlyeq_h)_{h\in\Z_{>0}}$ we immediately deduce that $A_h\subseteq A_{h+1},$ for all $h\in\Z_{>0}$ and so the set $A$ is an interval of $\mathbb{Z}.$ Moreover, the set $A$ trivially contains 0. In order to construct a (possibly infinite) permutation we have to endow the interval $A$ with a total order $\preccurlyeq.$ For all $i,j\in A,$ we set $\tilde{h}=\tilde{h}(i,j)\coloneqq\min\{h\in\Z_{>0}:i\in A_h,j\in A_h\},$ and we make the following choice: 
	\begin{equation}
	\label{oreder}
	i\preccurlyeq j\quad\text{ if and only if } \quad i\preccurlyeq_{\tilde{h}} j.
	\end{equation}
	Note that the transitivity of $\preccurlyeq$ easily follows from the consistency property.
	
	By construction, $r_h\big(A,\preccurlyeq\big)=(A_h,\preccurlyeq_h),$ for all $h\in\Z_{>0}$ and so we can conclude that the sequence $(A_h,\preccurlyeq_h)_h$ locally converges to $\big(A,\preccurlyeq\big).$
	Uniqueness follows from the uniqueness of the limit in a metric space. 
\end{proof}

Proposition \ref{consistprop} allows us to see the space $(\tilde{\mathcal{S}_{\bullet}},d)$ as a subset of a product space of finite sets as explained in the following observation.
\begin{obs}
	\label{inclusion}
	 We consider the product space $\prod_{h\in\Z_{>0}}\Sr^{\leq2h+1}$ endowed with the product topology, that is metrizable by the distance
	 \begin{equation*}
	 \tilde{d}\big((x_i)_{i\in\Z_{>0}},(y_i)_{i\in\Z_{>0}}\big)=2^{-\sup\big\{H\in\Z_{>0}\;:\;r_h(x_h)=r_h(y_h),\forall h\leq H\big\}}.
	 \end{equation*}
	 
	 We have the following isometric embedding
	\begin{equation}
	\label{inclmap}
	\begin{split}
	\mathcal{I} \colon\quad &\tilde{\mathcal{S}_{\bullet}} \;\longrightarrow \prod_{h\in\Z_{>0}}\Sr^{\leq2h+1}\\
	(A&,\preccurlyeq) \mapsto \big(r_h(A,\preccurlyeq)\big)_{h\in\Z_{>0}},
	\end{split}
	\end{equation}
	with inverse
	\begin{equation}
	\label{invinclmap}
	\begin{split}
	\mathcal{I}^{-1} \colon\quad \mathcal{I}&(\tilde{\mathcal{S}_{\bullet}}) \quad\;\;\longrightarrow\quad \tilde{\mathcal{S}_{\bullet}}\\
	(A_h&,\preccurlyeq_h)_{h\in\Z_{>0}} \mapsto \lim_{h\to\infty}{(A_h,\preccurlyeq_h)}.
	\end{split}
	\end{equation}
	We note that $\mathcal{I}^{-1}$ is well-defined and coincides with the inverse of $\mathcal{I}$ thanks to Proposition \ref{consistprop}. Moreover, the embedding is isometric since obviously the two distances coincide on $\tilde{\mathcal{S}_{\bullet}}$ and $\mathcal{I}(\tilde{\mathcal{S}_{\bullet}}).$
	
	Therefore, $\tilde{\mathcal{S}_{\bullet}}$ is homeomorphic to $\mathcal{I}(\tilde{\mathcal{S}_{\bullet}}).$ 
\end{obs}

Since $\tilde{\mathcal{S}_{\bullet}}$ and $\mathcal{I}(\tilde{\mathcal{S}_{\bullet}})$ are homeomorphic, then the corresponding notions of convergence are equivalent. An immediate consequence is the following.

\begin{prop}
\label{contrh}
Given a sequence $(A_n,\preccurlyeq_n)_{n\in\Z_{>0}}$ of rooted permutations in $\Sri,$ the following are equivalent:
\begin{enumerate}[(a)]
\item there exists $(A,\preccurlyeq)\in\Sri$ such that $$(A_n,\preccurlyeq_n)\stackrel{loc}{\longrightarrow}(A,\preccurlyeq);$$
\item there exists a family $(B_h,\ll_h)_{h\in\Z_{>0}}$ of finite rooted permutations such that
\begin{equation*}
r_h(A_n,\preccurlyeq_n)\stackrel{loc}{\longrightarrow}(B_h,\ll_h),\quad\text{for all}\quad h\in\Z_{>0}.
\end{equation*}  
\end{enumerate}
In particular if one of the two conditions holds (and so both), then $(B_h,\ll_h)=r_h(A,\preccurlyeq),$ for all $h\in\Z_{>0}.$
\end{prop}

\begin{obs}
	\label{continuityrh}
	Note that for all $h\in\Z_{>0},$ the restriction functions $r_h$ are continuous. This is a simple consequence of implication $(a)\Rightarrow(b)$ in the previous proposition.
\end{obs}

We end this section with the following.
\begin{thm}
	\label{compactpolish}
The metric space $(\tilde{\mathcal{S}_{\bullet}},d)$ is compact.
\end{thm}
\begin{proof}
By Proposition \ref{consistprop}, $\mathcal{I}(\tilde{\mathcal{S}_{\bullet}})$ is the set of consistent sequences. Therefore, it is an intersection of clopen subsets (the sets of sequences that are consistent in the first $\ell>0$ restrictions) of the compact product space $\prod_{h\in\Z_{>0}}\Sr^{\leq2h+1}.$ Therefore $\mathcal{I}(\tilde{\mathcal{S}_{\bullet}})$ is compact and so also $\tilde{\mathcal{S}_{\bullet}}$ since they are homeomorphic.
\end{proof}

\subsection{Benjamini--Schramm convergence: the deterministic case}

In this section we want to define a notion of weak-local convergence for a deterministic sequence of finite permutations $(\sigma^n)_{n\in\Z_{>0}}.$ The case of a random sequence $(\bm{\sigma}^n)_{n\in\Z_{>0}}$ will be discussed in Section \ref{random case}. 

Some of the results contained in this section are just special cases of the those in Section \ref{wbsconv} about the annealed version of the Benjamini--Schramm convergence. However we choose to present the deterministic case separately with a double purpose: first, in this way, the definitions are clearer, secondly, we believe that the subsequent quenched version of the Benjamini--Schramm convergence (see Section \ref{sbsconv}) will be more intuitive.

A possible way to define this weak-local convergence is to use the notion of local distance defined in Equation $(\ref{distance}).$ Therefore, we need to construct a sequence of rooted permutations from the sequence $(\sigma^n)_{n\in\Z_{>0}}.$ We can see a fixed permutation $\sigma$ as an element of $\mathcal{S}_\bullet$ only after a root $i$ has been chosen. A natural way to choose a root is to make the choice at random, and uniformly among the indices of $\sigma.$
In this way, a fixed permutation $\sigma$ naturally identifies a probability measure $\mu_{\sigma}$ on the space $\mathcal{S}_{\bullet},$ that is 
\begin{equation}
\label{permprob}
\text{for all Borel sets }A\subset\mathcal{S}_\bullet,\qquad\mu_{\sigma}(A)\coloneqq\mathbb{P}\big((\sigma,\bm{i})\in A\big),
\end{equation}
where $\bm{i}$ is a uniform index of $\sigma.$ Equivalently, the measure $\mu_{\sigma}$ is the law of the random variable $(\sigma,\bm{i})$ that takes values in the set $\Sr.$
\begin{notn}
In order to avoid any confusion, as done in Equation (\ref{permprob}), we write random quantities using \textbf{bold characters} to distinguish them from deterministic quantities. Moreover, given a probability measure $\mu$ we denote with $\E_{\mu}$ the expectation with respect to $\mu$ and given a random variable $\bm{X},$ we denote with $\mathcal{L}(\bm{X})$ its law. Given a sequence of random variables $(\bm{X}_n)_{n\in\Z_{>0}}$ we write $\bm{X}_n\stackrel{(d)}{\rightarrow}\bm{X}$ to denote the convergence in distribution and $\bm{X}_n\stackrel{P}{\rightarrow}\bm{X}$ to denote the convergence in probability. Finally, given an event $A$ we denote with $A^c$ the complement event of $A.$
\end{notn}

We are now ready to define our notion of Benjamini--Schramm convergence for sequences of permutations.

\begin{defn}
\label{weakconv}
Given a sequence $(\sigma^n)_{n\in\Z_{>0}}$ of elements in $\mathcal{S},$ we say that $(\sigma^n)_{n\in\Z_{>0}}$ \emph{Benjamini--Schramm converges} to a random (possibly infinite) rooted permutation $(\bm{A},\bm{\preccurlyeq})$, if the sequence $(\sigma^n,\bm{i}_n)_{n\in\Z_{>0}},$ where $\bm{i}_n$ is a uniform index of $|\sigma|,$ converges in distribution to $(\bm{A},\bm{\preccurlyeq})$ with respect to the local distance $d$ defined in Equation $(\ref{distance}).$
In this case we simply write $\sigma^n\stackrel{BS}{\longrightarrow} (\bm{A},\bm{\preccurlyeq})$ instead of $(\sigma^n,\bm{i}_n)\stackrel{(d)}{\rightarrow}(\bm{A},\bm{\preccurlyeq}).$
\end{defn}

Sometimes, in order to simplify notation, we will denote the Benjamini--Schramm limit of a sequence $(\sigma^n)_{n\in\Z_{>0}}$ simply by $\bm{\sigma}^{\infty}$ instead of $(\bm{A},\bm{\preccurlyeq}).$

Our main theorem in this section deals with the following interesting relation between the Benjamini--Schramm convergence and the convergence of consecutive pattern densities.

\begin{thm}
	\label{bscharact}
	For any $n\in\Z_{>0},$ let $\sigma^n$ be a permutation of size $n.$ Then the following are equivalent:
	\begin{enumerate}[(a)]
		\item there exists a random rooted infinite permutation $\bm{\sigma}^\infty$ such that $$\sigma^n\stackrel{BS}{\longrightarrow}\bm{\sigma}^\infty;$$ 
		\item there exist non-negative real numbers $(\Delta_{\pi})_{\pi\in\mathcal{S}}$ such that for all patterns $\pi\in\mathcal{S},$ $$\widetilde{\coc}(\pi,\sigma^n)\to\Delta_{\pi}.$$
	\end{enumerate}
	Moreover, if one of the two conditions holds (and so both) we have the following relation between the limiting objects:
	\begin{equation*}
	\mathbb{P}\big(r_h(\bm{\sigma}^\infty)=(\pi,h+1)\big)=\Delta_{\pi},\quad\text{for all}\quad h\in\Z_{>0},\text{ all}\quad\pi\in\mathcal{S}^{2h+1}.
	\end{equation*}
\end{thm}

We skip for the moment the proof of this theorem since it will be a particular case of the more general Theorem \ref{weakbsequivalence}.

\subsection{Benjamini--Schramm convergence: the random case}
\label{random case}
The goal of this section is to characterize the convergence in distribution of a sequence of \emph{random} permutations $(\bm{\sigma}^n)_{n\in\Z_{>0}}$ with respect to the local distance defined in Equation (\ref{distance}). We have two different natural choices for this definition, one stronger than the other (and not equivalent as shown in Example \ref{notequivalent}). The weaker definition is an analogue of the notion of Benjamini--Schramm convergence for random graphs developed  in \cite{benjamini2001recurrence}.

Before starting, we give some more explanations about our notation for random quantities. We will use a superscript notation on probability measure $\P$ (and on the corresponding expectation $\E$) to record the source of randomness. Specifically, given two independent random variables $\bm{X}$ and $\bm{Y}$ (with values in two spaces $E$ and $F$ respectively) and a set $A\subseteq E\times F,$ we write 
\begin{equation*}
\P^{\bm{Y}}\big((\bm{X},\bm{Y})\in A\big)\coloneqq\P\big((\bm{X},\bm{Y})\in A|\bm{X}\big),
\end{equation*}
and similarly
\begin{equation*}
\P^{\bm{X}}\big((\bm{X},\bm{Y})\in A\big)\coloneqq\P\big((\bm{X},\bm{Y})\in A|\bm{Y}\big).
\end{equation*}
Moreover, we recall the following standard relation
\begin{equation}
\label{condlaw}
\begin{split}
\P\big((\bm{X},\bm{Y})\in A\big)=\E\big[\mathds{1}_{(\bm{X},\bm{Y})\in A}\big]=\E^{\bm{X}}\Big[\E^{\bm{Y}}\big[\mathds{1}_{(\bm{X},\bm{Y})\in A}\big]\Big]&=\E^{\bm{X}}\Big[\P^{\bm{Y}}\big((\bm{X},\bm{Y})\in A\big)\Big]\\&=\E^{\bm{Y}}\Big[\P^{\bm{X}}\big((\bm{X},\bm{Y})\in A\big)\Big].
\end{split}
\end{equation}

\subsubsection{The annealed version of the Benjamini--Schramm convergence}
\label{wbsconv}
In analogy with the definition of Benjamini--Schramm convergence for deterministic sequences of permutations (and the definition of Benjamini--Schramm convergence for graphs) we can state the following.
\begin{defn}[Annealed version of the Benjamini--Schramm convergence]\label{weakweakconv}
Given a sequence $(\bm{\sigma}^n)_{n\in\Z_{>0}}$ of random permutations in $\mathcal{S},$ let $\bm{i}_n$ be a uniform index of $\bm{\sigma}^n$ conditionally on $\bm{\sigma}^n$. We say that $(\bm{\sigma}^n)_{n\in\Z_{>0}}$ \emph{converges in the annealed Benjamini--Schramm sense} to a random variable $\bm{\sigma}^{\infty}$ with values in $\Sri$ if the sequence of random variables $(\bm{\sigma}^n,\bm{i}_n)_{n\in\Z_{>0}}$ converges in distribution to $\bm{\sigma}^{\infty}$ with respect to the local distance. In this case we write $\bm{\sigma}^n\stackrel{aBS}{\longrightarrow}\bm{\sigma}^\infty$ instead of  $(\bm{\sigma}^n,\bm{i}_n)\stackrel{(d)}{\to}\bm{\sigma}^\infty.$
\end{defn}

Like in the deterministic case, the limiting object $\bm{\sigma}^{\infty}$ is a random rooted infinite permutation.

\begin{rem}
	Even though the above definition is stated for sequences of random permutations of arbitrary size, we will often consider the case when, for all $n\in\Z_{>0},$ $|\bm{\sigma}^n|=n$ almost surely and the random choice of the root $\bm{i}_n$ is uniform in $[n]$ and independent of $\bm{\sigma}^n$.
\end{rem}

Before stating the main theorem of this section we clarify two interesting properties of the statistics $\coc(\pi,\sigma)$ introduced in Section \ref{notation}. 

Given a permutation $\pi\in\mathcal{S}^k,$ we set 
$$\pi^{*m}\coloneqq\text{std}(\pi_1,\dots,\pi_{k},m-1/2)\in\mathcal{S}^{k+1},\quad \text{for all}\quad 1\leq m\leq k+1.$$ 
In the following example we give a graphical interpretation of the operation $\pi^{*m}$.
\begin{exmp}
Let $\pi=3421=
\begin{array}{lcr}
\begin{tikzpicture}
\begin{scope}[scale=.3]
\permutation{3,4,2,1}
\fill[blue!60!white, opacity=0.4] (1,1) rectangle (5,5); 
\end{scope}
\end{tikzpicture}
\end{array}$ 
then $\pi^{*3}=\text{std}(3,4,2,1,5/2)=
\begin{array}{lcr}
\begin{tikzpicture}
\begin{scope}[scale=.3]
\permutation{4,5,2,1,3}
\fill[blue!60!white, opacity=0.4] (1,1) rectangle (5,3); 
\fill[blue!60!white, opacity=0.4] (1,4) rectangle (5,6); 
\end{scope}
\end{tikzpicture}
\end{array}
=45213,$
where we highlight in light blue the entries already presented in the permutation $\pi=3421.$
\end{exmp}
Now, for all $\pi\in\mathcal{S}^{k}$ and $\sigma\in\mathcal{S}^{n}$ with $k<n,$ we just observe that
\begin{equation}
\label{coocrel}
\coc(\pi,\sigma)=\sum_{m=1}^{k+1}\coc(\pi^{*m},\sigma)+\mathds{1}_{\big\{\text{pat}_{[n-k+1,n]}(\sigma)=\pi\big\}},
\end{equation}
and so the convergence of $\widetilde{\coc}(\pi,\sigma)$ depends only on the convergence of $\widetilde{\coc}(\pi^{*m},\sigma),$ for all $1\leq m\leq k+1.$

We also give a probabilistic interpretation of $\widetilde{\coc}(\pi,\sigma)$ for $\pi\in\mathcal{S}^{2h+1}$ and $\sigma\in\mathcal{S}^n$ with $n\geq2h+1.$ Namely, let $\bm{i}_n$ be uniform in $[n],$ then
\begin{equation}
\label{probinterpret}
\widetilde{\coc}(\pi,\sigma)=\P\big(r_h(\sigma,\bm{i}_n)=(\pi,h+1)\big).
\end{equation}

We need the following observation.
\begin{obs}
	\label{separating class}
	We claim that the set of clopen balls (see the discussion after Equation (\ref{distance}))
	\begin{equation*}
	\label{convergece_determ}
	\mathcal{A}=\Big\{B\big((A,\preccurlyeq),2^{-h}\big):h\in\Z_{>0},(A,\preccurlyeq)\in\Sri\Big\}
	\end{equation*}
	is a separating class for the space $(\Sri,d)$, \emph{i.e.,} if
	two probability measures agree on $\mathcal{A}$ then they necessarily agree also on the
	whole space. This is a trivial consequence of the monotone class theorem, recalling that the intersection of two balls is either empty or one of them.
\end{obs}

We give in the following theorem some characterizations of the  annealed version of the Benjamini--Schramm convergence. 
\begin{thm}
\label{weakbsequivalence}
For any $n\in\Z_{>0},$ let $\bm{\sigma}^n$ be a random permutation of size $n$ and $\bm{i}_n$ be a uniform random index in $[n]$, independent of $\bm{\sigma}^n.$ Then the following are equivalent:
\begin{enumerate}[(a)]
\item there exists a random rooted infinite permutation $\bm{\sigma}^\infty$ such that $$\bm{\sigma}^n\stackrel{aBS}{\longrightarrow}\bm{\sigma}^\infty,$$ i.e., $(\bm{\sigma}^n,\bm{i}_n)\stackrel{(d)}{\to}\bm{\sigma}^\infty$ w.r.t.\ the local distance $d$ on $\Sri;$
\item for all $h\in\Z_{>0},$ there exist non-negative real numbers $(\Gamma^h_{\pi})_{\pi\in\mathcal{S}^{2h+1}}$ such that  $$\P\big(r_h(\bm{\sigma}^n,\bm{i}_n)=(\pi,h+1)\big)\xrightarrow[n\to\infty]{}\Gamma^h_{\pi},\quad\text{for all}\quad\pi\in\mathcal{S}^{2h+1};$$
\item there exist non-negative real numbers $(\Delta_{\pi})_{\pi\in\mathcal{S}}$ such that for all patterns $\pi\in\mathcal{S},$ $$\E[\widetilde{\coc}(\pi,\bm{\sigma}^n)]\to\Delta_{\pi}.$$
\end{enumerate}
Moreover, if one of the three conditions holds (and so all of them), for every fixed $h\in\Z_{>0},$ we have the following relations between the limiting objects,
\begin{equation}
\label{limrel}
  \mathbb{P}\big(r_h(\bm{\sigma}^\infty)=(\pi,h+1)\big)=\Delta_{\pi}=\Gamma^h_{\pi},\quad\text{for all}\quad\pi\in\mathcal{S}^{2h+1}.
\end{equation}
\end{thm}

Before proving the theorem we point out two important facts.

\begin{rem}
	Note that the theorem proves the existence of a random rooted infinite permutation $\bm{\sigma}^\infty$ but does not furnish any explicit construction of this object as a random total order on $\mathbb{Z}.$
\end{rem}

\begin{rem}
	\label{uyfvuoe2}
	For every fixed $h\in\Z_{>0},$ the condition (b) in the previous theorem considers only the probabilities $\P\big(r_h(\bm{\sigma}^n,\bm{i}_n)=(\pi,j)\big)$ for $\pi\in\mathcal{S}^{2h+1}$ and $j=h+1.$ We remark that for all the other cases, \emph{i.e.,} when $\pi\in\mathcal{S}^{2h+1}$ and $j\neq h+1,$ or $\pi\in\mathcal{S}^{\leq2h+1}$ with $|\pi|<2h+1$ and $j\in[|\pi|]$, it is easy to show that
	\begin{equation*}
	\P\big(r_h(\bm{\sigma}^n,\bm{i}_n)=(\pi,j)\big)\to 0.
	\end{equation*}
\end{rem}

\begin{proof}[Proof of Theorem \ref{weakbsequivalence}]
$(a)\Rightarrow(b).$ For all $h\in\Z_{>0},$ the convergence in distribution of the sequence $\big(r_h(\bm{\sigma}^n,\bm{i}_n)\big)_{n\in\Z_{>0}},$ follows from the continuity of the functions $r_h$ (Observation \ref{continuityrh}). Then (b) is a trivial consequence of the fact that $r_h(\bm{\sigma}^n,\bm{i}_n)$ takes its values in the finite set $\Sr^{\leq 2h+1}.$

$(b)\Rightarrow(a).$ Thanks to Theorem \ref{compactpolish}, $(\Sri,d)$  is a compact (and so Polish) space. Therefore, applying Prokhorov's Theorem, in order to show that $(\bm{\sigma}^n,\bm{i}_n)\stackrel{(d)}{\rightarrow}\bm{\sigma}^{\infty}$, for some random infinite permutation $\bm{\sigma}^{\infty},$ it is enough to show that for every pair of convergent subsequences  $(\bm{\sigma}^{n_k},\bm{i}_{n_k})_{k\in\Z_{>0}}$ and $(\bm{\sigma}^{n_\ell},\bm{i}_{n_\ell})_{\ell\in\Z_{>0}}$ with limits $\bm{\sigma}^\infty_1$ and $\bm{\sigma}^\infty_2$ respectively, then
\begin{equation}
\label{proofgoal}
\bm{\sigma}^\infty_1\stackrel{(d)}{=}\bm{\sigma}^\infty_2.
\end{equation}

Noting that the distributions of $\bm{\sigma}^\infty_1$ and $\bm{\sigma}^\infty_2$ must coincide on $\mathcal{A}$ and using Observation \ref{separating class}, we can conclude that Equation (\ref{proofgoal}) holds.

$(b)\Rightarrow(c).$
Note that if $\pi\in\mathcal{S}$ then there exist $h\in\Z_{>0}$ such that either $\pi\in\mathcal{S}^{2h+1}$ or $\pi\in\mathcal{S}^{2h}.$  
If $\pi\in\mathcal{S}^{2h+1},$ using relations (\ref{probinterpret}) and (\ref{condlaw}) with the independence between $\bm{\sigma}^n$ and $\bm{i}_n$, we have
\begin{equation}
\label{relprobexp}
\E\Big[\widetilde{\coc}(\pi,\bm{\sigma}^n)\Big]\stackrel{(\ref{probinterpret})}{=}\E^{\bm{\sigma}^n}\Big[\P^{\bm{i}_n}\big(r_h(\bm{\sigma}^n,\bm{i}_n)=(\pi,h+1)\big)\Big]\stackrel{(\ref{condlaw})}{=}\P\big(r_h(\bm{\sigma}^n,\bm{i}_n)=(\pi,h+1)\big),
\end{equation}
which converges if (b) holds.
Otherwise, if $\pi\in\mathcal{S}^{2h}$ we just use the observation done in Equation (\ref{coocrel}) and so, the convergence of $\E^{\sigma^n}\Big[\widetilde{\coc}(\pi,\bm{\sigma}^n)\Big],$ for $\pi\in\mathcal{S}^{2h},$ follows.

$(c)\Rightarrow(b).$ As before, if $\pi\in\mathcal{S}^{2h+1},$ the convergence of $\big(\P\big(r_h(\bm{\sigma}^n,\bm{i}_n)=(\pi,h+1)\big)\big)_{n\in\Z_{>0}}$ follows from Equation (\ref{relprobexp}).
\end{proof}

\subsubsection{The quenched version of the Benjamini--Schramm convergence}
\label{sbsconv}
We start by recalling that given a permutation $\sigma,$ the associated probability measure $\mu_{\sigma}$ is defined by Equation (\ref{permprob}). The quenched version of the Benjamini--Schramm convergence is inspired by the following equivalent reformulation of Definition \ref{weakconv}.

Given a sequence $(\sigma^n)_{n\in\Z_{>0}}$ of deterministic elements in $\mathcal{S},$ we can equivalently say that $(\sigma^n)_{n\in\Z_{>0}}$ Benjamini--Schramm converges to a probability measure $\mu\in\mathcal{P}(\tilde{\mathcal{S}}_\bullet)$, if the sequence $(\mu_{\sigma^n})_{n\in\Z_{>0}}$  converges to $\mu$ with respect to the weak topology induced by the local distance $d.$

In analogy, we can state the following for the random case.
\begin{defn}[Quenched version of the Benjamini--Schramm convergence]
	\label{strongconv}
Given a sequence $(\bm{\sigma}^n)_{n\in\Z_{>0}}$ of random permutation in $\mathcal{S}$ and a random measure $\bm{\mu}^\infty$ on $\Sri,$  we say that $(\bm{\sigma}^n)_{n\in\Z_{>0}}$ \emph{converges in the quenched Benjamini--Schramm sense} to $\bm{\mu}^\infty$ if the sequence of random measures $(\bm{\mu}_{\bm{\sigma}^n})_{n\in\Z_{>0}}$ converges in distribution to $\bm{\mu}^\infty$ with respect to the weak topology induced by the local distance $d.$ In this case we write $\bm{\sigma}^n\stackrel{qBS}{\longrightarrow}\bm{\mu}^\infty$ instead of $\bm{\mu}_{\bm{\sigma}^n}\stackrel{(d)}{\to}\bm{\mu}^\infty.$
\end{defn}

Unlike the annealed version of the Benjamini--Schramm convergence, the limiting object $\bm{\mu}^{\infty}$ is a random measure on $\Sri.$

\begin{rem}
	Note that, given a random permutation $\bm{\sigma},$ the corresponding random measure $\bm{\mu}_{\bm{\sigma}}$ is the conditional law of the random variable $\big((\bm{\sigma},\bm{i})\big|\bm{\sigma}\big).$
\end{rem}

We want to remark some important topological facts.
\begin{rem}
\label{remcompactness}
Given a random permutation $\bm{\sigma},$ the associated random uniformly rooted permutation $(\bm{\sigma},\bm{i})$ can be viewed as a random variable with values in the compact (and so Polish) space $(\Sri,d).$ Similarly the associated random probability measure $\bm{\mu}_{\bm{\sigma}}$ given by Equation (\ref{permprob}) can be viewed as a random variable with values in the set of probability measures $\mathcal{P}(\Sri).$ The space $\mathcal{P}(\Sri)$ endowed with the weak convergence topology is Polish.

We recall that given a metric space $(E,d)$ endowed with a $\sigma$-algebra $\mathcal{E},$ then $\mathcal{P}(E)$ is a metric space once equipped with the Prokorov metric $\nu.$ 

Moreover, if $(E,d)$ is a Polish space, the metric $\nu$ is such that the weak convergence is equivalent to $\nu$-convergence and $(\mathcal{P}(E),\nu)$ is also a Polish space (see for instance \cite[Theorem 6.8, p.73]{billingsley2013convergence}). Finally, if $(E,d)$ is compact, Prokorov's theorem ensures that $(\mathcal{P}(E),\nu)$ is also compact (for general results on convergence of measure, we refer to \cite{billingsley2013convergence}).
\end{rem}

Thanks to the previous remark we can interpret a random probability measure on $\Sri$ as a random variable with values in the compact (and so Polish) space $(\mathcal{P}(\Sri),\nu).$

We now generalize the notion of consecutive pattern density to probability measures on $\Sri.$
\begin{defn}
Given a probability measure $\mu$ on $\Sri,$ we define the \emph{consecutive pattern density} $\widetilde{\coc}(\pi,\mu)$ as
\begin{equation}
\label{gencocc}
\begin{split}
\widetilde{\coc}(\pi,\mu)\coloneqq\mu\Big(B\big((\pi,h+1),2^{-h}\big)\Big),\quad\text{for all}\quad \pi\in\mathcal{S}^{2h+1},\text{ all }h\in\Z_{>0},\\
\widetilde{\coc}(\pi,\mu)\coloneqq\mu\Big(\bigcup_{m=1}^{2h+1}B\big((\pi^{*m},h+1),2^{-h}\big)\Big),\quad\text{for all}\quad \pi\in\mathcal{S}^{2h},\text{ all }h\in\Z_{>0},
\end{split}
\end{equation}
\end{defn}
Note that if $\mu\equiv\mu_{\sigma}$ for some $\sigma\in\mathcal{S}$ then 
\begin{equation}
\label{coccrel}
\widetilde{\coc}(\pi,\mu_{\sigma})=\widetilde{\coc}(\pi,\sigma)+ O(1/|\sigma|),\quad\text{for all}\quad\pi\in\mathcal{S}.
\end{equation} 
The relation easily follows using Equations (\ref{permprob}) and (\ref{probinterpret}) (and Equation (\ref{coocrel}) when $\pi\in\mathcal{S}^{2h}$).

Before stating our main theorem of this section we state and prove a technical proposition.

\begin{prop}
\label{mapcontinuity}
For all $\pi\in\mathcal{S},$ the function
\begin{equation*}
\begin{split}
  \widetilde{\coc}(\pi,\cdot) \colon\quad \big(\mathcal{P}(\Sri&),\nu\big) \;\longrightarrow\;\; [0,1]\\
  &\mu \quad\;\;\mapsto \widetilde{\coc}(\pi,\mu)
\end{split}
\end{equation*}
is continuous.
\end{prop}

\begin{proof}
	By definition, depending on the parity of $|\pi|,$ either $\widetilde{\coc}(\pi,\mu)=\mu\Big(B\big((\pi,h+1),2^{-h}\big)\Big)$ or $\widetilde{\coc}(\pi,\mu)=\mu\Big(\bigcup_{m=1}^{2h+1}B\big((\pi^{*m},h+1),2^{-h}\big)\Big).$  Since a finite union of clopen balls is clopen, using the Portmanteau theorem,  if $\mu_n\to\mu$ in the weak sense then $\widetilde{\coc}(\pi,\mu_n)\to\widetilde{\coc}(\pi,\mu).$ Since weak convergence is equivalent to $\nu$-convergence, then $\widetilde{\coc}(\pi,\cdot)$ is continuous.
\end{proof}

We are now ready to state and prove our main theorem that gives us a characterization of the quenched version of the Benjamini--Schramm convergence. 

\begin{thm}
\label{strongbsconditions}
For any $n\in\Z_{>0},$ let $\bm{\sigma}^n$ be a random permutation of size $n$ and $\bm{i}_n$ be a uniform random index in $[n]$, independent of $\bm{\sigma}^n.$ Then the following are equivalent:
\begin{enumerate}[(a)]
\item there exists a random measure $\bm{\mu}^\infty$ on $\Sri$ such that   $$\bm{\sigma}^n\stackrel{qBS}{\longrightarrow}\bm{\mu}^\infty,$$ i.e., $\bm{\mu}_{\bm{\sigma}^n}\stackrel{(d)}{\to}\bm{\mu}^\infty$ w.r.t.\ the weak topology induced by the local distance $d$ on $\Sri;$
\item there exists a family of non-negative real random variables $(\bm{\Gamma}^h_{\pi})_{h\in\Z_{>0},\pi\in\mathcal{S}^{2h+1}}$ such that  $$\Big(\P^{{\bm{i}}_n}\big(r_h(\bm{\sigma}^n,\bm{i}_n)=(\pi,h+1)\big)\Big)_{h\in\Z_{>0},\pi\in\mathcal{S}^{2h+1}}\stackrel{(d)}{\to}(\bm{\Gamma}^h_{\pi})_{h\in\Z_{>0},\pi\in\mathcal{S}^{2h+1}},$$
w.r.t. the product topology;
\item there exists an infinite vector of non-negative real random variables $(\bm{\Lambda}_{\pi})_{\pi\in\mathcal{S}}$ such that $$\big(\widetilde{\coc}(\pi,\bm{\sigma}^n)\big)_{\pi\in\mathcal{S}}\stackrel{(d)}{\to}(\bm{\Lambda}_{\pi})_{\pi\in\mathcal{S}}$$ w.r.t.\ the product topology.
\end{enumerate}
In particular, if one of the three conditions holds (and so all of them) then 
\begin{equation}
\label{limrel2}
\widetilde{\coc}(\pi,\bm{\mu}^{\infty})\stackrel{(d)}{=}\bm{\Lambda}_{\pi},\quad\text{for all}\quad\pi\in\mathcal{S},
\end{equation}
and moreover, for every fixed $h\in\Z_{>0},$
\begin{equation}
\label{limrel3}
\begin{split}
\widetilde{\coc}(\pi,\bm{\mu}^{\infty})\stackrel{(d)}{=}\bm{\Gamma}^h_{\pi},\quad\text{for all}\quad \pi\in\mathcal{S}^{2h+1},\\
\widetilde{\coc}(\pi,\bm{\mu}^{\infty})\stackrel{(d)}{=}\sum_{m=1}^{2h+1}\bm{\Gamma}^h_{\pi^{*m}},\quad\text{for all}\quad \pi\in\mathcal{S}^{2h}.
\end{split}
\end{equation}
\end{thm}

Before giving the proof, as in Theorem \ref{weakbsequivalence}, we highlight two important facts.

\begin{rem}
	Note that the theorem states the existence of a random measure $\bm{\mu}^\infty$ but does not furnish any explicit construction of this object.
\end{rem}

\begin{rem}
	\label{uyfvuoe}
For every fixed $h\in\Z_{>0},$ the condition (b) in the previous theorem considers only the conditional probabilities $\P^{{\bm{i}}_n}\big(r_h(\bm{\sigma}^n,\bm{i}_n)=(\pi,j)\big)$ for $\pi\in\mathcal{S}^{2h+1}$ and $j=h+1.$ We remark that all the other cases are trivial. For more details see Remark \ref{uyfvuoe2}.
\end{rem}

\begin{proof}
	$(a)\Rightarrow(c)$ Let $\pi_1,\dots,\pi_r$ be a finite sequence of patterns. By Proposition \ref{mapcontinuity}, the map $\mu\mapsto (\widetilde{\coc}(\pi_i,\mu))_{1\leq i\leq r}$ is continuous. Therefore, the convergence of $\bm{\sigma}^n\stackrel{qBS}{\longrightarrow}\bm{\mu}^\infty$ implies the convergence of $(\widetilde{\coc}(\pi_i,\bm{\mu}_{\bm{\sigma}^n}))_{1\leq i\leq r}\stackrel{(d)}{\to} (\widetilde{\coc}(\pi_i,\bm{\mu}^\infty))_{1\leq i\leq r}.$  Since $\widetilde{\coc}(\pi_i,\bm{\mu}_{\bm{\sigma}^n})\stackrel{(\ref{coccrel})}{=}\widetilde{\coc}(\pi_i,{\bm{\sigma}^n})+O(1/n)$ we have the convergence in distribution of all the finite-dimensional marginals of the infinite vector $\big(\widetilde{\coc}(\pi,\bm{\sigma}^n)\big)_{\pi\in\mathcal{S}},$ and this proves $(c)$ (see for instance \cite[ex. 2.4, p.19]{billingsley2013convergence}).

$(c)\Rightarrow(b).$ Note that for all $\pi\in\mathcal{S}^{2h+1},$ $h\in\Z_{>0}$ we have by Equation (\ref{probinterpret}),
\begin{equation}
\label{pcocrel}
\P^{\bm{i}_n}\big(r_h(\bm{\sigma}^n,\bm{i}_n)=(\pi,h+1)\big)=\widetilde{\coc}(\pi,\bm{\sigma}^n).
\end{equation}

We can conclude that if $(c)$ holds then the vector $\Big(\P^{{\bm{i}}_n}\big(r_h(\bm{\sigma}^n,\bm{i}_n)=(\pi,h+1)\big)\Big)_{h\in\Z_{>0},\pi\in\mathcal{S}^{2h+1}}$ converges in distribution.

$(b)\Rightarrow(a).$ Thanks to Remark \ref{remcompactness}, $\bm{\mu}_{\bm{\sigma}^n}$  is a random variable with values in the compact space $(\mathcal{P}(\Sri),\nu).$ Therefore, applying Prokhorov's theorem to $(\mathcal{P}(\Sri),\nu),$ in order to show that $\bm{\mu}_{\bm{\sigma}^n}\stackrel{(d)}{\to}\bm{\mu}^\infty$ for some random measure $\bm{\mu}^\infty,$ it is enough to show that for every pair of convergent subsequences $(\bm{\mu}_{\bm{\sigma}^{n_k}})_{k\in\Z_{>0}}$ and $(\bm{\mu}_{\bm{\sigma}^{n_\ell}})_{\ell\in\Z_{>0}}$ with limits $\bm{\mu}^\infty_1$ and $\bm{\mu}^\infty_2$ respectively, it holds that
\begin{equation*}
\bm{\mu}^\infty_1\stackrel{(d)}{=}\bm{\mu}^\infty_2.
\end{equation*}
Thanks to \cite[Theorem 2.2]{kallenberg2017random}, in order to prove the above equality in distribution it is enough to show that
\begin{equation*}
\big(\bm{\mu}^\infty_1(B_1),\dots,\bm{\mu}^\infty_1(B_n)\big)\stackrel{(d)}{=}\big(\bm{\mu}^\infty_2(B_1),\dots,\bm{\mu}^\infty_2(B_n)\big),\quad\text{for all}\quad n\in\Z_{>0},\;B_1,\dots,B_n\in\mathcal{A},
\end{equation*}
where $\mathcal{A}$ is the separating class defined in Observation \ref{separating class}.

Fix $n\in\Z_{>0}$ and for all $1\leq i\leq n,$ let $B_i=B\big((A_i,\preccurlyeq_i),2^{-h_i}\big)$ for some $(A_i,\preccurlyeq_i)\in\Sri,$ $h_i\in\Z_{>0}.$ Since by assumption $\bm{\mu}_{\bm{\sigma}^{n_k}}\stackrel{(d)}{\rightarrow}\bm{\mu}^\infty_1$ then applying \cite[Theorem 4.11]{kallenberg2017random} (which is a generalization of the Portmanteau theorem for random measures) we have
\begin{equation*}
\begin{split}
\big(\bm{\mu}^\infty_1(B_i)\big)_{1\leq i\leq n}&\stackrel{(d)}{=}\lim_{k\to\infty}\big(\bm{\mu}_{\bm{\sigma}^{n_k}}(B_i)\big)_{1\leq i\leq n}\\
&\stackrel{(d)}{=}\lim_{k\to\infty}\Big(\P^{\bm{i}_{n_k}}\big(r_{h_i}(\bm{\sigma}^{n_k},\bm{i}_{n_k})=r_{h_i}(A_i,\preccurlyeq_i)\big)\Big)_{1\leq i\leq n}.
\end{split}
\end{equation*}
Therefore, using condition (c) and Remark \ref{uyfvuoe}, we can conclude that 
$$\big(\bm{\mu}^\infty_1(B_i)\big)_{1\leq i\leq n}\stackrel{(d)}{=}\big(\Gamma^{h_i}_{r_{h_i}(A_i,\preccurlyeq_i)}\big)_{1\leq i\leq n},$$
where $\Gamma^{h_i}_{r_{h_i}(A_i,\preccurlyeq_i)}=\Gamma^{h_i}_{\pi},$ if $r_{h_i}(A_i,\preccurlyeq_i)=(\pi,h_i+1)$, for some $\pi\in\mathcal{S}^{2h_i+1}$ and $\Gamma^{h_i}_{r_{h_i}(A_i,\preccurlyeq_i)}=0,$ otherwise.
Similarly we have the following equality in distribution,
\begin{equation*}
\big(\bm{\mu}^\infty_2(B_i)\big)_{1\leq i\leq n}\stackrel{(d)}{=}\big(\Gamma^{h_i}_{r_{h_i}(A_i,\preccurlyeq_i)}\big)_{1\leq i\leq n}.
\end{equation*}
Therefore $\bm{\mu}^\infty_1\stackrel{(d)}{=}\bm{\mu}^\infty_2.$
\end{proof}

\subsubsection{Relation between the annealed and the quenched versions of the Benjamini--Schramm convergence}

We recall that the \emph{intensity (measure)} $\E[\bm{\mu}]$ of a random measure $\bm{\mu}$ is defined as the expectation $\E[\bm{\mu}](A)\coloneqq\E\big[\bm{\mu}(A)\big]$ for all measurable sets $A.$ We have the following expected implication.
\begin{prop}
	\label{wsrel}
 For any $n\in\Z_{>0},$ let $\bm{\sigma}^n$ be a random permutation of size $n$ and $\bm{i}_n$ be a uniform random index in $[n]$, independent of $\bm{\sigma}^n.$ If $\bm{\sigma}^n\stackrel{qBS}{\longrightarrow}\bm{\mu}^\infty,$  for some random measure $\bm{\mu}^\infty$ on $\Sri,$ then  $$\bm{\sigma}^n\stackrel{aBS}{\longrightarrow}\bm{\sigma}^\infty,$$
 where $\bm{\sigma}^\infty$ is the random rooted infinite permutation with law $\mathcal{L}(\bm{\sigma}^\infty)=\E[\bm{\mu}^\infty].$  
\end{prop}
\begin{proof}
	The first part of the statement, in particular the existence of the random rooted infinite permutation $\bm{\sigma}^\infty,$ is a trivial consequence of Theorem \ref{weakbsequivalence} and Theorem \ref{strongbsconditions}. For example, condition (c) in the second theorem trivially implies the condition (c) in the first theorem. 
	
	Therefore we just have to show that $\mathcal{L}(\bm{\sigma}^\infty)=\E[\bm{\mu}^\infty].$ For every clopen ball $B$ contained in the separating class $\mathcal{A}$ (see Observation \ref{convergece_determ}) we have
	\begin{equation*}
	\begin{split}
	\E[\bm{\mu}^\infty](B)&=\E\big[\bm{\mu}^\infty(B)\big]=\lim_{n\to\infty}\E\big[\bm{\mu}_{\bm{\sigma}^n}(B)\big]=\lim_{n\to\infty}\E\big[\P^{\bm{i}_n}\big((\bm{\sigma}^n,\bm{i}_n)\in B\big)\big]\\
	&\stackrel{(\ref{condlaw})}{=}\lim_{n\to\infty}\P\big((\bm{\sigma}^n,\bm{i}_n)\in B\big)=\P(\bm{\sigma}^\infty\in B)=\mathcal{L}(\bm{\sigma}^\infty)(B),
	\end{split}
	\end{equation*}
	where in the second equality we used that $\bm{\sigma}^n\stackrel{qBS}{\longrightarrow}\bm{\mu}^\infty$ and in the fifth equality that $\bm{\sigma}^n\stackrel{aBS}{\longrightarrow}\bm{\sigma}^\infty.$ Since the two measures agree in a separating class we can conclude that they are equal.
\end{proof}

We now show in the next example that in general the two versions of Benjamini--Schramm convergence are not equivalent.

\begin{exmp}
	\label{notequivalent}
	For all $n\in\Z_{>0},$ let $\bm{\sigma}^n$ be the random permutation defined by
	$$\P(\bm{\sigma}^n=12\dots n)=\frac{1}{2}=\P(\bm{\sigma}^n=n\;n\text-1\dots 1),$$
	and $\tau_n$  be the deterministic permutation (see also Fig.~\ref{examplesequence})
	\begin{equation*}
	\begin{split}
	&\tau_n=135\dots n\;n\text-1\;n\text-3\dots 2,\quad\text{if } n\text{ is odd,}\\
	&\tau_n=135\dots n\text-1\;n\;n\text-2\dots 2,\quad\text{if } n\text{ is even.}\\
	\end{split}
	\end{equation*}
It is easy to show that the two sequences $(\bm{\sigma}^n)_{n\in\Z_{>0}}$  and $(\tau_n)_{n\in\Z_{>0}}$ have two different quenched Benjamini--Schramm limits but the same annealed Benjamini--Schramm limit. Therefore, the alternating sequence between $\bm{\sigma}^n$ and $\tau_n$ converges in the annealed Benjamini--Schramm sense but does not converge in the quenched Benjamini--Schramm sense.
	
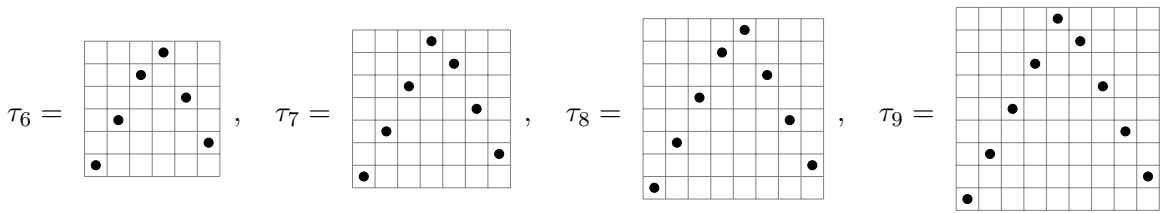
\begin{figure}[h]
	\begin{equation*}
	\tau_6=
	\begin{array}{lcr}
	\begin{tikzpicture}
	\begin{scope}[scale=.3]
	\permutation{1,3,5,6,4,2} 
	\end{scope}
	\end{tikzpicture}
	\end{array},
	\quad\tau_7=
	\begin{array}{lcr}
	\begin{tikzpicture}
	\begin{scope}[scale=.3]
	\permutation{1,3,5,7,6,4,2} 
	\end{scope}
	\end{tikzpicture}
	\end{array},
	\quad\tau_8=
	\begin{array}{lcr}
	\begin{tikzpicture}
	\begin{scope}[scale=.3]
	\permutation{1,3,5,7,8,6,4,2} 
	\end{scope}
	\end{tikzpicture}
	\end{array},
	\quad\tau_9=
	\begin{array}{lcr}
	\begin{tikzpicture}
	\begin{scope}[scale=.3]
	\permutation{1,3,5,7,9,8,6,4,2} 
	\end{scope}
	\end{tikzpicture}
	\end{array}
	\end{equation*}
	\caption{Some elements of the sequence $(\tau_n)_{n\in\Z_{>0}}.$ \label{examplesequence}}
\end{figure} 

Indeed, 
\begin{equation*}
\big(\widetilde{\coc}(\pi,\bm{\sigma}^n)\big)_{\pi\in\mathcal{S}}\stackrel{(d)}{\to}(\bm{X}_{\pi})_{\pi\in\mathcal{S}},
\end{equation*}
where, taking a Bernoulli random variable $\bm{Y}$ with parameter $1/2,$ $\bm{X}_{\pi}=\bm{Y}$ if $\pi=12\dots k$ for some $k\in\Z_{>0},$ $\bm{X}_{\pi}=1-\bm{Y}$ if $\pi=k\;k\text{-}1\dots 1$  for some $k\in\Z_{>0}$ and $\bm{X}_{\pi}=0$ otherwise.

On the contrary, 
\begin{equation*}
\big(\widetilde{\coc}(\pi,\tau_n)\big)_{\pi\in\mathcal{S}}\stackrel{(d)}{\to}(Z_{\pi})_{\pi\in\mathcal{S}},
\end{equation*}
where $Z_{\pi}=\frac{1}{2}$ if $\pi=12\dots k,$ or $\pi=k\;k\text{-}1\dots 1$ and $Z_{\pi}=0$ otherwise.

We can conclude, using condition (c) in Theorem \ref{strongbsconditions}, that the two sequences $(\bm{\sigma}^n)_{n\in\Z_{>0}}$  and $(\tau_n)_{n\in\Z_{>0}}$ have two different quenched Benjamini--Schramm limits (they are different by Equation (\ref{limrel2})). On the contrary, noting that $\lim_{n\to\infty}\E\big[\widetilde{\coc}(\pi,\bm{\sigma}^n)\big]=\lim_{n\to\infty}\E\big[\widetilde{\coc}(\pi,\tau_n)\big]$ for all $\pi\in\mathcal{S},$ then, using condition (c) in Theorem \ref{weakbsequivalence}, we conclude that the two sequences $(\bm{\sigma}^n)_{n\in\Z_{>0}}$  and $(\tau_n)_{n\in\Z_{>0}}$ have the same annealed Benjamini--Schramm limit.
\end{exmp}

We now analyze the particular case when the limiting objects $(\bm{\Lambda}_{\pi})_{\pi\in\mathcal{S}}$ (or $(\bm{\Gamma}^h_{\pi})_{h\in\Z_{>0},\pi\in\mathcal{S}^{2h+1}}$) in Theorem \ref{strongbsconditions} are deterministic, \emph{i.e.,} when there is a concentration phenomenon (this will be the case of Sections \ref{231} and \ref{321}). Before stating our result we need the following.
\begin{rem}
	When a random measure $\bm{\mu}$ on $\Sri$ is almost surely equal to a deterministic measure $\mu$ on $\Sri$, we will simply denote it with $\mu.$ In particular, if a sequence of random permutations $(\bm{\sigma}^n)_{n\in\Z_{>0}}$ converges in the quenched Benjamini--Schramm sense to a deterministic measure $\mu$ (instead of random measure) on $\Sri$, we will simply write $\bm{\sigma}^n\stackrel{qBS}{\longrightarrow}\mu.$
\end{rem}

\begin{cor}
	\label{detstrongbsconditions}
	For any $n\in\Z_{>0},$ let $\bm{\sigma}^n$ be a random permutation of size $n$ and $\bm{i}_n$ be a uniform random index in $[n]$, independent of $\bm{\sigma}^n.$ Then the following are equivalent:
	\begin{enumerate}[(a)]
		\item there exists a (deterministic) measure $\mu$ on $\Sri$ such that   $$\bm{\sigma}^n\stackrel{qBS}{\longrightarrow}\mu;$$
		\item there exists an infinite vector of non-negative real numbers $(\Lambda_{\pi})_{\pi\in\mathcal{S}}$ such that $$\big(\widetilde{\coc}(\pi,\bm{\sigma}^n)\big)_{\pi\in\mathcal{S}}\stackrel{P}{\to}(\Lambda_{\pi})_{\pi\in\mathcal{S}}$$ w.r.t.\ the product topology;
		\item there exists an infinite vector of non-negative real numbers $(\Lambda'_{\pi})_{\pi\in\mathcal{S}}$ such that for all $\pi\in\mathcal{S},$ $$\widetilde{\coc}(\pi,\bm{\sigma}^n)\stackrel{P}{\to}\Lambda'_{\pi};$$
		\item for all $h\in\Z_{>0},$ there exists a family of non-negative real numbers $(\Gamma^h_{\pi})_{\pi\in\mathcal{S}^{2h+1}}$ such that  $$\P^{{\bm{i}}_n}\big(r_h(\bm{\sigma}^n,\bm{i}_n)=(\pi,h+1)\big)\stackrel{P}{\to}\Gamma^h_{\pi},\quad\text{for all}\quad \pi\in\mathcal{S}^{2h+1}.$$
	\end{enumerate}
	In particular, if one of the four conditions holds (and so all of them) we have the same relations as in Theorem $\ref{strongbsconditions}$ with the additional relation $\Lambda_{\pi}=\Lambda'_{\pi}$ for all $\pi\in\mathcal{S}.$
\end{cor}
\begin{rem}
	Thanks to Proposition \ref{wsrel} note that if (a) holds then $\bm{\sigma}^n\stackrel{aBS}{\longrightarrow}\bm{\sigma}^\infty,$
	where $\bm{\sigma}^\infty$ is the random rooted infinite permutation with law $\mathcal{L}(\bm{\sigma}^\infty)=\mu.$ 
\end{rem}
\begin{proof}
	$(a)\Rightarrow(b)\Rightarrow(c)\Leftrightarrow(d).$ The first and the last relations are simple consequences of Theorem \ref{strongbsconditions}, where the determinism is given by Equations (\ref{limrel2}) and (\ref{limrel3}). The second relation is trivial.
	
	$(c)\Rightarrow(a).$ Since when the limits are deterministic, the pointwise convergence in distribution is equivalent to the convergence in distribution for the product topology, assumption (c) implies condition (b). Therefore, using Theorem \ref{strongbsconditions} and Proposition \ref{wsrel}, we obtain that there exists a random measure $\bm{\mu}^{\infty}$ on $\Sri$ and a random rooted permutation $\bm{\sigma}^{\infty}$ (with $\mathcal{L}(\bm{\sigma}^\infty)=\E[\bm{\mu}^\infty]$) such that
	$$\bm{\sigma}^n\stackrel{qBS}{\longrightarrow}\bm{\mu}^{\infty}\quad\text{and}\quad\bm{\sigma}^n\stackrel{aBS}{\longrightarrow}\bm{\sigma}^\infty.$$
	In order to conclude, we need only to show that, under our assumption, we have $\bm{\mu}^{\infty}=\mathcal{L}(\bm{\sigma}^\infty)$  a.s.\ (this obviously implies that $\bm{\mu}^{\infty}$ is deterministic).
	Using again \cite[Theorem 2.2]{kallenberg2017random} and Observation \ref{separating class}, in order to prove the above equality it is enough to show that
	\begin{equation*}
	\big(\bm{\mu}^\infty(B_1),\dots,\bm{\mu}^\infty(B_k)\big)=\big(\mathcal{L}(\bm{\sigma}^\infty)(B_1),\dots,\mathcal{L}(\bm{\sigma}^\infty)(B_k)\big),\quad\text{for all}\quad k\in\Z_{>0},\;B_1,\dots,B_k\in\mathcal{A}.
	\end{equation*}
	Fix $k\in\Z_{>0}$ and for all $1\leq i\leq k$ let $B_i=B\big((A_i,\preccurlyeq_i),2^{-h_i}\big)$ for some $(A_i,\preccurlyeq_i)\in\Sri,$ $h_i\in\Z_{>0}.$ Since $\bm{\mu}_{\bm{\sigma}^{n}}\stackrel{(d)}{\rightarrow}\bm{\mu}^\infty$ then applying \cite[Theorem 4.11]{kallenberg2017random} we have
	\begin{equation*}
	\begin{split}
	\big(\bm{\mu}^\infty(B_i)\big)_{1\leq i\leq k}&\stackrel{(d)}{=}\lim_{n\to\infty}\big(\bm{\mu}_{\bm{\sigma}^n}(B_i)\big)_{1\leq i\leq k}\\
	&\stackrel{(d)}{=}\lim_{n\to\infty}\Big(\P^{\bm{i}_n}\big(r_{h_i}(\bm{\sigma}^{n},\bm{i}_{n})=r_{h_i}(A_i,\preccurlyeq_i)\big)\Big)_{1\leq i\leq k}\stackrel{(d)}{=}\big(\Gamma^{h_i}_{r_{h_i}(A_i,\preccurlyeq_i)}\big)_{1\leq i\leq k},
	\end{split}
	\end{equation*}
	where in the last equality we use assumption (d) that is equivalent to (c). We also use again the fact that when the limits are deterministic, the pointwise convergence in distribution is equivalent to the convergence in distribution for the product topology.
	Similarly, since $(\bm{\sigma}^{n},\bm{i}_n)\stackrel{(d)}{\rightarrow}\bm{\sigma}^\infty$, using the Portmanteau theorem, we have 
	\begin{equation*}
	\begin{split}
	\big(\mathcal{L}(\bm{\sigma}^\infty)(B_i)\big)_{1\leq i\leq k}&=\lim_{n\to\infty}\big(\P\big((\bm{\sigma}^{n},\bm{i}_{n})\in B_i\big)\big)_{1\leq i\leq k}\\
	&=\lim_{n\to\infty}\Big(\P\big(r_{h_i}(\bm{\sigma}^{n},\bm{i}_{n})=r_{h_i}(A_i,\preccurlyeq_i)\big)\Big)_{1\leq i\leq k}=\big(\Gamma^{h_i}_{r_{h_i}(A_i,\preccurlyeq_i)}\big)_{1\leq i\leq k}.
	\end{split}
	\end{equation*}
	Therefore $\bm{\mu}^{\infty}=\mathcal{L}(\bm{\sigma}^\infty).$ 
\end{proof}

We conclude this section with the following.

\begin{rem}
	We have seen that in order to prove the convergence in the quenched Benjamini--Schramm sense when the limiting object is deterministic, it is enough to show the pointwise convergence for the vector $\big(\widetilde{\coc}(\pi,\bm{\sigma}^n)\big)_{\pi\in\mathcal{S}}.$ This is not true when the limiting object is random. Indeed it is quite easy to explicitly construct a counter example.
	
	For every $n\in\Z_{>0},$ we can consider the three permutations
	$$\sigma^n=12\dots n,\quad \tau^n=n\dots21,\quad \rho^n=
	\begin{array}{lcr}
	\includegraphics[scale=0.6]{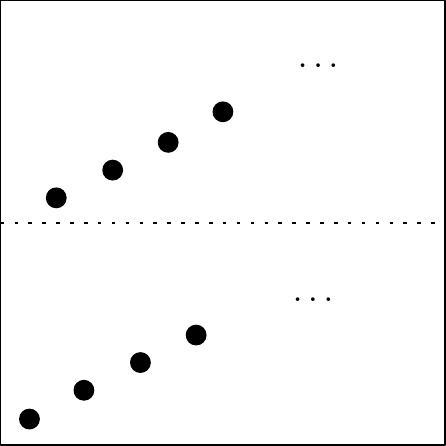}\\
	\end{array},$$
	and we can define the two following random permutations
	\begin{equation*}
	\bm{\pi}^n\coloneqq\begin{cases} 
	\sigma^{2n}\oplus\tau^n, & \text{with probability } 1/3, \\
	\sigma^{n}\oplus\rho^{2n}, & \text{with probability } 1/3, \\
	\tau^{2n}\oplus\rho^n, & \text{with probability } 1/3,
	\end{cases}\qquad
	\tilde{\bm{\pi}}^n\coloneqq\begin{cases} 
	\sigma^{n}\oplus\tau^{2n}, & \text{with probability } 1/3, \\
	\sigma^{2n}\oplus\rho^n, & \text{with probability } 1/3, \\
	\tau^{n}\oplus\rho^{2n}, & \text{with probability } 1/3,
	\end{cases}
	\end{equation*}
	where $\oplus$ denotes the direct sum of two permutations, \emph{i.e.,} for $\pi\in\mathcal{S}^k$ and $\sigma\in\mathcal{S}^n,$ $$\pi\oplus\sigma=\pi_1\dots\pi_k(\sigma_1+k)\dots(\sigma_n+k).$$
	For each pattern $\pi\in\mathcal{S},$ both $\widetilde{\coc}(\pi,\bm{\pi}^n)$ and $\widetilde{\coc}(\pi,\tilde{\bm{\pi}}^n)$ converge to the same limit. However the joint vectors $\big(\widetilde{\coc}(\pi,\bm{\pi}^n)\big)_{\pi\in\mathcal{S}}$ and $\big(\widetilde{\coc}(\pi,\tilde{\bm{\pi}}^n)\big)_{\pi\in\mathcal{S}}$ have different limits in distribution for the product topology.
\end{rem}

\subsection{Characterization of the annealed Benjamini--Schramm limits}
We now characterize the annealed Benjamini--Schramm limiting objects. More precisely we show that such a limiting object satisfies a ``shift-invariant" property (see Definition \ref{shiftinvariant} below). Conversely, we also prove in Theorem \ref{conversecostruction} that every random ``shift-invariant" infinite rooted permutation is the annealed Benjamini--Schramm limit of a sequence of random finite rooted permutations. 

We start by considering for all patterns $\pi$ of size $k$ and every shift $s\in\Z,$ the following sets
\begin{equation*}
O^s(\pi)=\big\{(A,\preccurlyeq)\in\Sri:[1+s,k+s]\subseteq A\text{ and }\pi_{1}+s\preccurlyeq\pi_{2}+s\preccurlyeq\dots\preccurlyeq\pi_{k}+s\big\}.
\end{equation*}
Moreover we set $O(\pi)\coloneqq O^0(\pi).$
\begin{defn}
\label{shiftinvariant}
	We say that a random rooted permutation $(\bm{A},\bm{\preccurlyeq})\in\Sri$ is \emph{shift-invariant} if for all patterns $\pi\in\mathcal{S},$
	\begin{equation}
	\label{fin_dim_marg}
	\P\big((\bm{A},\bm{\preccurlyeq})\in O^s(\pi)\big)=\P\big((\bm{A},\bm{\preccurlyeq})\in O(\pi)\big),\quad\text{for all}\quad s\in\Z.
	\end{equation}
\end{defn}

\begin{obs}
	Note that for a random rooted shift-invariant permutation $(\bm{A},\bm{\preccurlyeq})$ we have $\bm{A}=\mathbb{Z}$ a.s.
	Indeed,
	\begin{equation}
	\P\big(\bm{A}=\mathbb{Z}\big)=\P\Big(\bigcap_{s\in\Z_{\geq0}}\big\{\{s,-s\}\subset\bm{A}\big\}\Big)=\lim_{s\to \infty}\P\big(\{s,-s\}\subset\bm{A}\big),
	\end{equation}
	where we used that $\bm{A}$ is an interval. Rewriting the last term as $$\P\big(\{s,-s\}\subset\bm{A}\big)=\P\big((\bm{A},\bm{\preccurlyeq})\in O^{(s-1)}(1),(\bm{A},\bm{\preccurlyeq})\in O^{-(s-1)}(1)\big)$$ and noting that trivially $\P\big((\bm{A},\bm{\preccurlyeq})\in O^{(-1)}(1)\big)=1,$ then using the shift-invariant property we conclude that $\P\big(\bm{A}=\mathbb{Z}\big)=1.$
\end{obs}

\begin{obs}
	We recall that in the introduction (see the discussion before Theorem \ref{shiftinvthm}) we said that a random infinite rooted permutation, or equivalently a random total order, is shift-invariant if it has the same distribution than its shift. Note that the probabilities in Equation (\ref{fin_dim_marg}), when $s=1,$ are the finite-dimensional distributions of the random infinite rooted permutation $(\bm{A},\bm{\preccurlyeq})$ and its shift. Therefore, the two definitions coincide since a random variable on a product space is completely determine by its finite-dimensional distributions (see for instance \cite[ex. 1.2, pp. 9-11]{billingsley2013convergence}).
\end{obs}

\begin{prop}
	\label{easyimplication}
	Let $\bm{\sigma}^{\infty}$ be the annealed Benjamini--Schramm limit of a sequence $(\bm{\sigma}^n)_{n\in\Z_{>0}}$ of random elements in $\mathcal{S}$ with $|\bm{\sigma}^n|=n$ a.s.\ (and $\bm{i}_n$ a uniform random index in $[n]$, independent of $\bm{\sigma}^n$)  then $\bm{\sigma}^{\infty}$ is shift-invariant.
\end{prop}
\begin{proof}
	Fix a pattern $\pi\in\mathcal{S}^k,$ a shift $s\in\Z$ and suppose that $\bm{\sigma}^n\stackrel{aBS}{\longrightarrow} \bm{\sigma}^{\infty}.$ Since $O^s(\pi)$ is clopen, by the Portmanteau theorem,  we have,
	\begin{equation*}
	\P\big(\bm{\sigma}^{\infty}\in O^s(\pi)\big)=\lim_{n\to\infty}\P\big((\bm{\sigma}^n,\bm{i}_n)\in O^s(\pi)\big)\stackrel{(\ref{condlaw})}{=}\lim_{n\to\infty}\E^{\bm{\sigma}^n}\big[\P^{\bm{i}_n}\big((\bm{\sigma}^n,\bm{i}_n)\in O^s(\pi)\big)\big].
	\end{equation*}
	We note that $(\bm{\sigma}^n,\bm{i}_n)\in O^s(\pi)$ if and only if
	\begin{equation*}
	[1+s,k+s]\subseteq A_{\bm{\sigma}^n,\bm{i}_n}\quad\text{and}\quad\bm{\sigma}^n_{{\pi_1}+s+\bm{i}_n}\leq\bm{\sigma}^n_{{\pi_2}+s+\bm{i}_n}\leq\dots\leq\bm{\sigma}^n_{{\pi_k}+s+\bm{i}_n}.
	\end{equation*} 
	The first event has probability that tends to 1 since $|\bm{\sigma}^n|\to\infty.$ The second event (when well-defined, \emph{i.e.,} conditionally on the first) is equivalent to $\text{pat}_{[\bm{i}_n+s+1,\bm{i}_n+s+k]}(\bm{\sigma}^n)=\pi^{-1}.$
	Since $\bm{i}_n$ is uniform using relation (\ref{probinterpret}), we obtain that
	\begin{equation*}
	\begin{split}
	\P\big(\bm{\sigma}^{\infty}\in O^s(\pi)\big)=&\lim_{n\to\infty}\E^{\bm{\sigma}^n}\big[\P^{\bm{i}_n}\big((\bm{\sigma}^n,\bm{i}_n)\in O^s(\pi)\big)\big]\\
	=&\lim_{n\to\infty}\E^{\bm{\sigma}^n}\Big[\P^{\bm{i}_n}\big(\text{pat}_{[\bm{i}_n+s+1,\bm{i}_n+s+k]}(\bm{\sigma}^n)=\pi^{-1},[1+s,k+s]\subseteq A_{\bm{\sigma}^n,\bm{i}_n}\big)\Big]\\
	=&\lim_{n\to\infty}\E^{\bm{\sigma}^n}\big[\widetilde{\coc}(\pi^{-1},\bm{\sigma}^n)\big]
	\end{split}
	\end{equation*}
	 and the last term is independent of $s.$ This concludes the proof.
\end{proof}

We now prove the inverse statement of the previous proposition.

\begin{thm}
	\label{conversecostruction}
	Let $(\mathbb{Z},\bm{\preccurlyeq})$ be a random shift-invariant rooted permutation. Then the sequence of random permutations $(\bm{\sigma}^n)_{n\in\Z_{>0}}$ defined, for all $n\in\Z_{>0},$ by
	\begin{equation}
	\label{pattdensitydef}
	\P(\bm{\sigma}^n=\rho)=\P\big((\mathbb{Z},\bm{\preccurlyeq})\in O(\rho^{-1})\big),\quad\text{for all}\quad\rho\in\mathcal{S}^n,
	\end{equation}
	(rooted at a uniform random index $\bm{i}_n$ in $[n]$, independent of $\bm{\sigma}^n$)
	converges in the annealed Benjamini--Schramm sense to $(\mathbb{Z},\bm{\preccurlyeq}).$
\end{thm}

\begin{proof}
	Thanks to Theorem \ref{weakbsequivalence}, it is enough to prove that for all $h\in\Z_{>0},$
	\begin{equation}
	\label{wwwtp}
	\P\big(r_h(\bm{\sigma}^n,\bm{i}_n)=(\pi,h+1)\big)\xrightarrow[n\to\infty]{}\P\big(r_h(\mathbb{Z},\bm{\preccurlyeq})=(\pi,h+1)\big),\quad\text{for all}\quad\pi\in\mathcal{S}^{2h+1},
	\end{equation}
	where $\bm{i_n}$ is uniform in $[n].$
	
	Fix $h\in\Z_{>0}.$ Note that if $\pi\in\mathcal{S}^{2h+1}$ then for $n\geq |\pi|,$ 
	
	\begin{equation*}
	\begin{split}
	\P\big(r_h(\bm{\sigma}^n,\bm{i}_n)=(\pi,h+1)\big)&=\sum_{\rho\in\mathcal{S}^n}\P\big(\bm{\sigma}^n=\rho\big)\P^{\bm{i}_n}\big(r_h(\rho,\bm{i }_n)=(\pi,h+1)\big)\\
	&=\sum_{\rho\in\mathcal{S}^n}
	\P\big((\Z,\bm{\preccurlyeq})\in O(\rho^{-1})\big)\widetilde{\coc}(\pi,\rho),
	\end{split}
	\end{equation*}
	where in the last equality we used Equation (\ref{pattdensitydef}).
	In particular, if the limits exist,
	\begin{equation}
	\label{jhvuwirbc}
	\lim_{n\to\infty}\P\big(r_h(\bm{\sigma}^n,\bm{i}_n)=(\pi,h+1)\big)=\lim_{n\to\infty}\sum_{\rho\in\mathcal{S}^n}
	\P\big((\Z,\bm{\preccurlyeq})\in O(\rho^{-1})\big)\widetilde{\coc}(\pi,\rho).
	\end{equation}

	In order to prove the existence of the previous limit we rewrite $\P\big(r_h(\mathbb{Z},\bm{\preccurlyeq})=(\pi,h+1)\big)$ in a more convenient way. Before proceeding with this, we note that for $n$ large enough, we have
	\begin{equation*}
	\label{fwigfrwuifb}
	\P\big((\mathbb{Z},\bm{\preccurlyeq})\in O(\pi^{-1})\big)=\sum_{\rho\in\mathcal{S}^{n}}\P\big((\mathbb{Z},\bm{\preccurlyeq})\in O(\rho)\big)\mathds{1}_{\big\{\text{pat}_{[1,2h+1]}(\rho^{-1})=\pi\big\}}
	\end{equation*}
	and more generally, for all $s\in\Z_{>0},$
	\begin{equation}
	\label{wuryviyrw cowl}
	\P\big((\mathbb{Z},\bm{\preccurlyeq})\in O^s(\pi^{-1})\big)=\sum_{\rho\in\mathcal{S}^{n}}\P\big((\mathbb{Z},\bm{\preccurlyeq})\in O(\rho)\big)\mathds{1}_{\big\{\text{pat}_{[s+1,s+2h+1]}(\rho^{-1})=\pi\big\}}.
	\end{equation}

	We start by noting that,
	\begin{equation}
	\label{jboubvowvb}
	\P\big(r_h(\mathbb{Z},\bm{\preccurlyeq})=(\pi,h+1)\big)=\P\big(([-h,h],\bm{\preccurlyeq})=(\pi,h+1)\big)=\P\big((\mathbb{Z},\bm{\preccurlyeq})\in O^{-(h+1)}(\pi^{-1})\big),
	\end{equation}
	where we used in the last equality that $([-h,h],\bm{\preccurlyeq})=(\pi,h+1)$ if and only if $$\pi^{-1}_{1}-(h+1)\bm{\preccurlyeq}\pi^{-1}_{2}-(h+1)\bm{\preccurlyeq} ... \bm{\preccurlyeq}\pi^{-1}_{2h+1}-(h+1).$$
	
	Using the shift-invariant property we can rewrite the last term of Equation (\ref{jboubvowvb}) as follows,
	\begin{equation}
	\label{bwuivrouvb}
	\begin{split}
	\P\big((\mathbb{Z},\bm{\preccurlyeq})\in O^{-(h+1)}(\pi^{-1})\big)&=\sum_{s=0}^{n-2h-1}\frac{\P\big((\mathbb{Z},\bm{\preccurlyeq})\in O^s(\pi^{-1})\big)}{n-2h}\\
	&=\sum_{\rho\in\mathcal{S}^{n}}\P\big((\mathbb{Z},\bm{\preccurlyeq})\in O(\rho)\big)\frac{\sum_{s=0}^{n-2h-1}\mathds{1}_{\{\text{pat}_{[s+1,s+2h+1]}(\rho^{-1})=\pi\}}}{n-2h}.
	\end{split}
	\end{equation}
	where in the last equality we applied Equation (\ref{wuryviyrw cowl}).
	
	Moreover, noting that $\sum_{s=0}^{n-2h-1}\mathds{1}_{\{\text{pat}_{[s+1,s+2h+1]}(\rho^{-1})=\pi\}}=\coc(\pi,\rho^{-1})$ and combining Equations (\ref{jboubvowvb}) and (\ref{bwuivrouvb}), we obtain, for all $n$ large enough,
	\begin{equation*}
    \P\big(r_h(\mathbb{Z},\bm{\preccurlyeq})=(\pi,h+1)\big)=\sum_{\rho\in\mathcal{S}^{n}}\P\big((\mathbb{Z},\bm{\preccurlyeq})\in O(\rho)\big)\frac{\coc(\pi,\rho^{-1})}{n-2h}.
	\end{equation*}
	The right-hand side coincides in the limit for $n$ that goes to infinity with the limit in Equation (\ref{jhvuwirbc}) and so , since the left-hand side is independent of $n,$ we proved the existence of the limit. In particular, this proves Equation (\ref{wwwtp}) and concludes the proof.
\end{proof}

\section{Trees and random trees, a toolbox}
\label{tree_notation}
 In this section we summarize notation and results on trees and random trees that we need in the two following sections. The reader comfortable with this topic may skip this section.
\subsection{Rooted ordered trees and Galton--Watson trees}
We briefly recall in this section the definition of Galton--Watson trees (for a complete introduction, see \emph{e.g.,} \cite{abraham2015introduction}). We recall Neveu's formalism (see \cite{neveu1986arbres}) for rooted ordered trees. We set 
$$\mathcal{U}\coloneqq\bigcup_{n\geq0}\big(\Z_{>0}\big)^n,$$
the set of finite sequences of positive integers with the convention $(\Z_{>0})^0=\{\varnothing\}.$ For $n\geq 1$ and $u=u_1\dots u_n\in\mathcal{U},$ we set $|u|=n$ the length of $u$ with the convention $|\varnothing|=0.$ If $u$ and $v$ are two sequences of $\mathcal{U},$ we denote by $uv$ the concatenation of the two sequences. Let $u,v\in\mathcal{U}.$ We say that $v$ is an ancestor of $u$ and write $v\prec u,$ if there exists $w\in\mathcal{U},$ $w\neq\varnothing$ such that $u=vw.$ The set of ancestors of $u$ is denoted by $A_u.$ 
\begin{defn}
	\label{orderedtree}
	A \emph{rooted ordered tree} $T$ is a subset of $\mathcal{U}$ that satisfies:
	\begin{itemize}
		\item $\varnothing\in T$ (this vertex is called root);
		\item if $u\in T,$ then $A_u\subseteq T;$
		\item for every $u\in T,$ there exists $d^+_u(T)\in\Z_{>0}\cup\{+\infty\}$ such that, for every $i\in\Z_{>0},$ $ui\in T$ if and only if $1\leq i\leq d^+_u(T).$ The number $d^+_u(T)$ represents the number of children of the vertex $u\in T.$ 
	\end{itemize}	
\end{defn}
 
We always think of (and draw) a rooted ordered tree $T$ as a graph (embedded in the plane) where the vertices are the elements of $T$ and the edges connect vertices $u$ and $ui,$ for all $i\in\Z_{>0},$ $u,ui\in T.$ We embed rooted ordered trees in the plane with the root at the top of the trees in such a way that the trees grow downward and the labels of the $n$-th generation (\emph{i.e.,} of all the vertices with length $n$) increase from left to right. An example is given in Fig.~\ref{rotordtree}.

\begin{figure}[htbp]
	\begin{center}
		\includegraphics[scale=1]{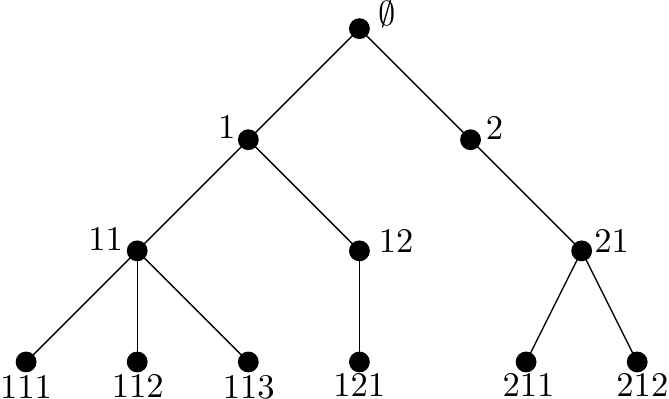}\\
		\caption{A rooted ordered tree embedded in the plane.}\label{rotordtree}
	\end{center}
\end{figure}

We denote with $\mathbb{T}$ the set of rooted ordered trees with no vertex having infinitely many children, 
$$\mathbb{T}\coloneqq\{T:d^+_u(T)<\infty,\; \forall u\in T\},$$
and with $\mathbb{T}^n\subset\mathbb{T}$ the set of rooted ordered trees with $n$ vertices. We often denote the root of the tree also by $r$ instead of $\varnothing.$ We recall that given $T\in\mathbb{T},$ for each vertex $v\in T,$ the height of $v$ is defined as $h(v)\coloneqq|v|.$ We denote with $d$ the classical distance on trees, \emph{i.e.,} for every $u,v\in T,$ $d(u,v)$ is equal to the number of edges in the unique path between $u$ and $v.$ Note that $h(v)=d(r,v).$ We also denote by $|T|$ the number of vertices of $T$. Moreover we set $\mathcal{L}(T)\coloneqq\big\{v\in T:d^+_v(T)=0\big\},$ namely $\mathcal{L}(T)$ is the set of leaves in $T$. We denote by $a(v)$ the parent of $v$ in $T,$ \emph{i.e.,} the largest ancestor of $v$ (for the order $\prec$) and we also write $a^i(v),$ $i\geq 0,$ for the $i$-th ancestor of $v$ in $T,$ \emph{i.e.,} $a^i(v)=\underbrace{a\circ a\circ\dots\circ a}_{i\text{-times}}(v)$ (with the convention that if $r=a^k(v)$ then $a^i(v)=r$ for all $i\geq k$). 

We denote with $c(u,v)$ the youngest common ancestor of $u$ and $v,$ namely $c(u,v)$ is the largest vertex $w$ (for the order $\prec$) such that $u=wu'$ and $v=wv'$ for some $u',v'\in\mathcal{U}.$ We also denote by $f(T,v)$ the fringe subtree of $T$ rooted at $v,$ \emph{i.e.,} the subtree of $T$ containing all the vertices $u\in\mathcal{U}$ such that $vu\in T.$

Finally, we recall the following definition of a standard model for random trees. 

\begin{defn}
	A $\mathbb{T}$-valued random variable $\bm{T}$ is said to have the \emph{branching property} if for $n\in\Z_{>0},$ conditionally on $\{d^+_{\varnothing}(\bm{T})=n\}$, the $n$ fringe subtrees, rooted at the $n$ children of the root, are independent
	and distributed as the original tree $\bm{T}$.
	A $\mathbb{T}$-valued random variable $\bm{T}^\eta$ is a \emph{Galton--Watson tree} with offspring distribution $\eta$ if it has the branching
	property and the distribution of $d^+_{\varnothing}(\bm{T})$ is $\eta.$
\end{defn}

\subsection{Depth-first traversal}
We now recall two classical ways of visiting the vertices of a rooted ordered tree that go under the name of \emph{depth-first traversal.} These traversals originally developed for algorithms searching vertices in a tree, work as follows.
Let $T$ be a rooted ordered tree with root $r$. 

\begin{defn}[\emph{Pre-order}]
	If $T$ consists only of $r$, then $r$ is the pre-order traversal of $T$. Otherwise, suppose that $T_1,T_2,\dots,T_d$ are the fringe subtrees in $T$ respectively rooted at the children of $r$ from left to right. The pre-order traversal begins by visiting $r$. It continues by traversing $T_1$ in pre-order, then $T_2$ in pre-order,	and so on, until $T_d$ is traversed in pre-order.
\end{defn}

\begin{defn}[\emph{Post-order}]
	If $T$ consists only of $r$, then $r$ is the post-order traversal of $T$. Otherwise, suppose that $T_1,T_2,\dots,T_d$ are the fringe subtrees in $T$ respectively rooted at the children of $r$ from left to right. The post-order traversal begins by traversing $T_1$ in post-order, then $T_2$ in post-order, and so on, until $T_d$ is traversed in post-order. It ends by visiting $r.$
\end{defn}

We will also say that $T$ is labeled with the pre-order (resp. post-order) labeling starting from $k$, if we associate the tag $i+k-1$ to the $i$-th visited vertex by the pre-order (resp. post-order) traversal. Moreover, if we simply say that $T$ is labeled with the pre-order (resp. post-order), we will assume that we are starting from $k=1$.

We will show an example in Fig. \ref{exemporder1} in the case of binary trees.

\subsection{Binary trees}We now introduce binary trees. We set
$$\mathcal{U}_{1,2}\coloneqq\bigcup_{n\geq0}\{1,2\}^n$$
the set of finite binary words with the convention $\{1,2\}^0=\{\varnothing\}.$ We use the same notation as in the case of ordered trees.
\begin{defn}
A \emph{binary tree} $T$ is a subset of $\mathcal{U}_{1,2}$ that satisfies:
\begin{itemize}
	\item $\varnothing\in T;$
	\item if $u\in T,$ then $A_u\subseteq T.$
\end{itemize}
\end{defn}

\begin{rem}
	Note that a binary tree is not a particular rooted ordered tree. Indeed the third condition in Definition \ref{orderedtree} is omitted in the case of binary trees. This implies that, if a vertex of a binary tree has only one child then it is either a left vertex (if it is marked with 1) or a right vertex (if it is marked with 2). See Fig.~\ref{lrd} for an example.
\end{rem}

Given a vertex $v$ in the tree, we denote by $T^v_L$ and $T^v_R$ the left and the right fringe subtrees of $T$ hanging below the vertex $v.$ We simply write $T_L$ and $T_R$ if $v=\varnothing$ namely, for the left and the right fringe subtrees below the root. Finally we let $\mathbb{T}_b$ be the set of all binary trees.

\begin{figure}[htbp]
	\begin{center}
		\includegraphics[scale=.80]{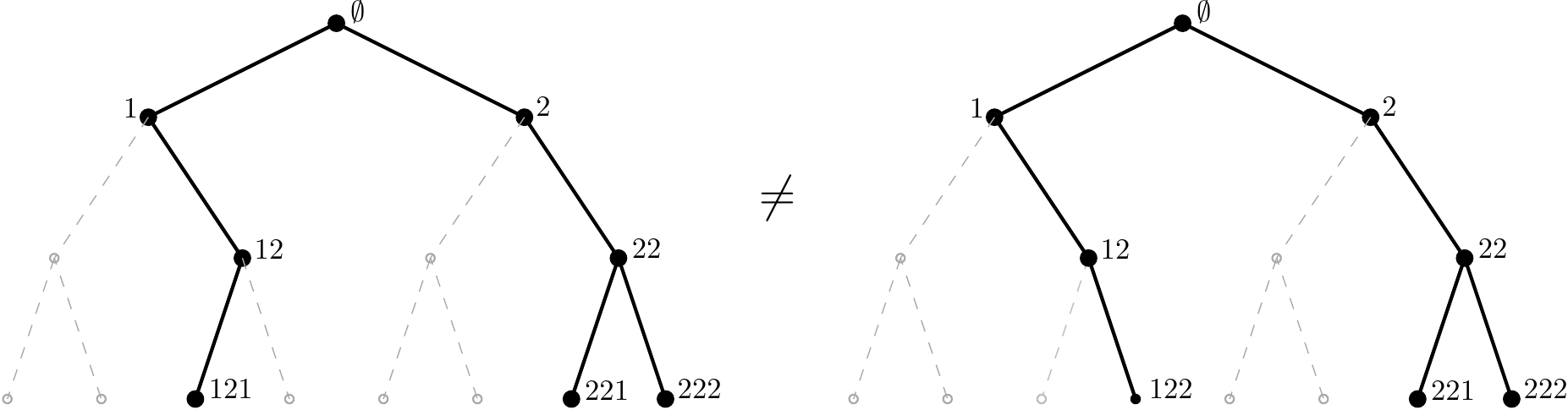}\\
		\caption{Two different binary trees.}\label{lrd}
	\end{center}
\end{figure}

The notions of pre-order and post-order traversal extend immediately to binary trees. Moreover, in the case of binary trees, the in-order traversal is also available. Let $T$ be a binary tree with root $r$. 
\begin{defn}[\emph{In-order}]
	If $T$ consists only of $r$, then $r$ is the in-order traversal of $T$. Otherwise the in-order traversal begins by traversing $T_L$ in in-order, then visits $r,$ and concludes by traversing $T_R$ in in-order.
\end{defn}
The notion of in-order labeling is defined similarly to the pre-order labeling or post-order labeling.
We exhibit an example of the three different labelings in Fig.~\ref{exemporder1}.

\begin{figure}
	\begin{center}
		\includegraphics[scale=.80]{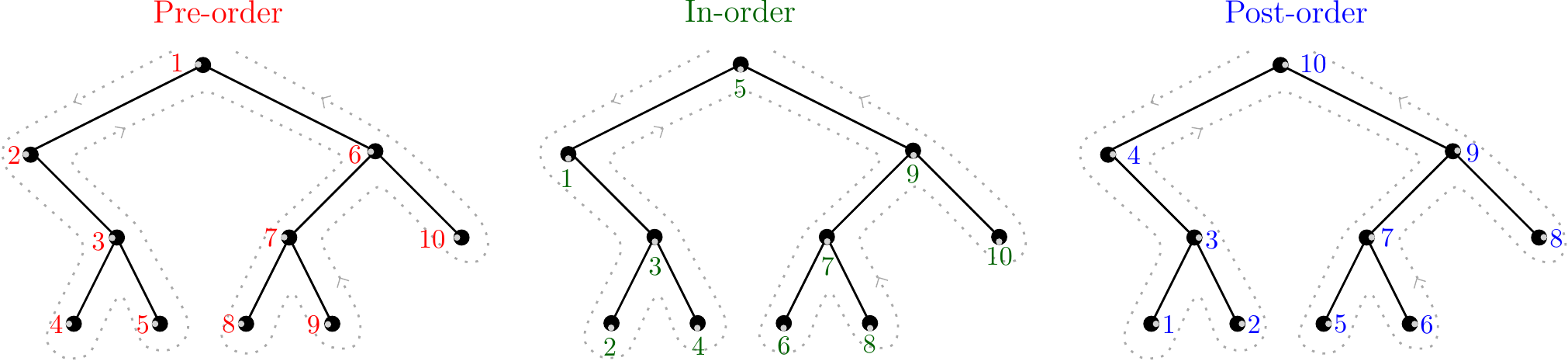}\\
		\caption{We label in red the left picture using the pre-order, in green the center picture using the in-order and in blue the right picture using the post-order. The three labelings can be obtained following the gray dotted line from left to right and putting a tag at each time we reach the small gray dots inside each vertex (that are on the left of the vertex for the pre-order, on the bottom of the vertex for the in-order and on the right of the vertex for the post-order).
		}\label{exemporder1}
	\end{center}
\end{figure}

We finally extend the notion of Galton--Watson trees to \emph{binary} Galton--Watson trees. 
\begin{defn}
	A \emph{binary Galton--Watson tree} $\bm{T}^\eta$ with offspring distribution $\eta=(\eta_0,\eta_L,\eta_R,\eta_2),$ is a $\mathbb{T}_b$-valued random variable with the branching property and such that
	\begin{equation*}
	\begin{split}
	&\P(\text{The root has 0 child})=\eta_0,\\
	&\P(\text{The root has only a left child})=\eta_L,\\
	&\P(\text{The root has only a right child})=\eta_R,\\
	&\P(\text{The root has 2 children})=\eta_2.
	\end{split}
	\end{equation*}
\end{defn}

\section{The local limit for uniform random 231-avoiding permutations}
\label{231}
The goal of this section is to prove that uniform random 231-avoiding permutations converge in the quenched Benjamini--Schramm sense. 
Before stating our precise result we fix some notation. Given a permutation $\pi\in\mathcal{S}$ we denote with LRMax($\pi$) the set of indices of the left-to-right maxima of $\pi,$ \emph{i.e.,} the values
$i\in [|\pi|]$ such that $\pi_j<\pi_i$ for every $j<i$. Similarly we denote with RLMax($\pi$) the set of indices of the right-to-left maxima of $\pi$ \emph{i.e.,} values
$i\in [|\pi|]$ such that $\pi_i>\pi_j$ for every $j>i.$ Moreover, we set $|$LRMax($\pi$)$|\coloneqq $Card$\big($LRMax($\pi$)$\big)$ and  $|$RLMax($\pi$)$|\coloneqq $Card$\big($RLMax($\pi$)$\big).$ We also fix $\text{Max}(\pi)\coloneqq\text{LRMax}(\pi)\cup\text{RLMax}(\pi),$ and similarly we set $|\text{Max}(\pi)|\coloneqq\text{Card}(\text{Max}(\pi)).$ Note that $|\text{Max}(\pi)|=|\text{LRMax}(\pi)|+|\text{RLMax}(\pi)|-1$ since $\text{LRMax}(\pi)\cap\text{RLMax}(\pi)$ contains only the index of the maximum of $\pi.$

For all $k>0,$ we define the following probability distribution on $\text{Av}^k(231),$ 
$$P_{231}(\pi)\coloneqq\frac{2^{|\text{LRMax}(\pi)|+|\text{RLMax}(\pi)|}}{2^{2|\pi|}},\quad\text{for all}\quad \pi\in\text{Av}^k(231).$$

\begin{rem}
	The fact that the above equation defines a probability distribution on $\text{Av}^k(231)$ is a consequence of Theorem \ref{231theorem} below, noting that in Equation (\ref{wfowrbfv3}), summing over all $\pi\in\text{Av}^k(231)$ the left-hand side, we trivially obtain $\frac{n-k+1}{n}$ that tends to 1. We also present in Appendix \ref{combint} a generating function proof of this fact.  
\end{rem}

\begin{exmp}
	\label{exempmax} 
	We consider the permutation  
	$$\pi=132985476=
	\begin{array}{lcr}
	\begin{tikzpicture}
	\begin{scope}[scale=.3]
	\permutation{1,3,2,9,8,5,4,7,6}
	\draw (1+.5,1+.5) [green, fill] circle (.21);  
	\draw (2+.5,3+.5) [green, fill] circle (.21); 
	\draw (4+.5,9+.5) [orange, fill] circle (.21);
	\draw (5+.5,8+.5) [cyan, fill] circle (.21);
	\draw (8+.5,7+.5) [cyan, fill] circle (.21); 
	\draw (9+.5,6+.5) [cyan, fill] circle (.21); 
	\end{scope}
	\end{tikzpicture}
	\end{array},$$ where we draw in green the left-to-right maxima, in blue the right-to-left maxima, and in orange the maximum of the permutation, which is both a left-to-right maximum and a right-to-left maximum. We have LRMax($\pi$)$=\{1,2,4\}$  and RLMax($\pi$)$=\{9,8,5,4\}.$ Thus $P_{231}(\pi)=\frac{2^{3+4}}{2^{18}}=\big(\frac{1}{2}\big)^{11}.$
\end{exmp}
Before stating our main result, we recall the following trivial fact.
\begin{obs}
	If $\pi\notin\text{Av}(231)$ and $\sigma\in\text{Av}(231)$ then obviously $\coc(\pi,\sigma)=0.$ Therefore, in what follow, we simply analyze the consecutive occurrences densities $\widetilde{\coc}(\pi,\sigma)$ for $\pi\in\text{Av}(231).$
\end{obs}

\begin{thm}
	\label{231theorem}	
	Let $\bm{\sigma}^n$ be a uniform random permutation in $\emph{Av}^n(231)$ for all $n\in\Z_{>0}.$ For all $\pi\in\emph{Av}(231),$ we have the following convergence in probability,
	\begin{equation}
	\label{wfowrbfv3}
	\widetilde{\coc}(\pi,\bm{\sigma}^n)\stackrel{P}{\rightarrow}P_{231}(\pi).
	\end{equation}
\end{thm}
Since the limiting frequencies $\big(P_{231}(\pi)\big)_{\pi\in\text{Av}(231)}$ are deterministic, using Corollary \ref{detstrongbsconditions}, we have the following.
\begin{cor}
	\label{231corol}
	Let $\bm{\sigma}^n$ be a uniform random permutation in $\emph{Av}^n(231)$ for all $n\in\Z_{>0}.$ There exists a random infinite rooted permutation $\bm{\sigma}^\infty_{231}$ such that   $$\bm{\sigma}^n\stackrel{qBS}{\longrightarrow}\mathcal{L}(\bm{\sigma}^\infty_{231}).$$
\end{cor}	

\begin{rem}
	We highlight that Corollary \ref{231corol} proves the existence of a random infinite rooted permutation $\bm{\sigma}^\infty_{231}$ without explicitly constructing this limiting object. Nevertheless we are able to provide a construction of $\bm{\sigma}^\infty_{231}$ in Section \ref{explcon}. 
\end{rem}
	
Section \ref{231} is structured as follows:
\begin{itemize}
	\item in Section \ref{premres} we introduce a well-known bijection (see \emph{e.g.,} \cite{bona2012surprising}) between binary trees and 231-avoiding permutations. Moreover we present a technique due to Janson \cite{janson2017patterns}. With this tool we show that in order to study $\E\big[\widetilde{\coc}(\pi,\bm{\sigma}^n)\big],$ it is enough to study a similar expectation in terms of a specific binary Galton--Watson tree;
	\item in Section \ref{weakres} we study the behavior of $\E\big[\widetilde{\coc}(\pi,\bm{\sigma}^n)\big]$ (see in particular Proposition \ref{weakprop}) using the results from the previous section;
	\item in Section \ref{mainres} we prove Theorem \ref{231theorem}: with similar techniques we study the second moment $\E\big[\widetilde{\coc}(\pi,\bm{\sigma}^n)^2\big]$ and then we conclude the proof applying the Second moment method;
	\item in Section \ref{explcon} we give an explicit construction of the limiting object $\bm{\sigma}^\infty_{231}.$
\end{itemize}

\subsection{Notation and preliminary results}
\label{premres}
\subsubsection{Notation}
We introduce some more notation. When convenient, we extend in the trivial way the various notation introduced for permutations to arbitrary sequences of distinct numbers. For example, we extend the notion of $\coc(\pi,\sigma)$ for two arbitrary sequences of distinct numbers $x_1\dots x_n$ and $y_1\dots y_k$  as
\begin{equation*}
\coc(y_1\dots y_k,x_1\dots x_n)\coloneqq\coc\big(\text{std}(y_1\dots y_k),\text{std}(x_1\dots x_n)\big). 
\end{equation*}  

Given $\sigma\in\mathcal{S}^n,$ let indmax$(\sigma)$ be the index of the maximal value $n,$ \emph{i.e.}, $\sigma_{\text{indmax}(\sigma)}=n.$ If $\ell=\text{indmax}(\sigma),$ we set $\sigma_L\coloneqq\sigma_1\dots\sigma_{\ell-1}$ and $\sigma_R\coloneqq\sigma_{\ell+1}\dots\sigma_{n},$ respectively the (possibly empty) left and right subsequences of $\sigma,$ before and after the maximal value. In particular we have $\sigma=\sigma_L\sigma_\ell\sigma_R,$ where we point out that $\sigma_L\sigma_\ell\sigma_R$ is not the composition of permutations but just the concatenation of $\sigma_L,\sigma_\ell$ and $\sigma_R$ seen as words.

We will write $\text{pat}_{b(k)}(\sigma)$ for $\text{pat}_{[1,k]}(\sigma)$ namely, for the pattern occurring in the first $k$ positions of $\sigma.$ Similarly we will write $\text{pat}_{e(k)}(\sigma)$ for $\text{pat}_{[|\sigma|-k+1,|\sigma|]}(\sigma)$ namely, for the pattern occurring in the last $k$ positions of $\sigma.$ Note that $b(k)$ and $e(k)$ stands for \emph{beginning} and \emph{end}. Moreover, if either $k=0$ or $k>|\sigma|$ we set 
\begin{equation*}
\text{pat}_{b(k)}(\sigma)\coloneqq\emptyset\quad\text{and}\quad\text{pat}_{e(k)}(\sigma)\coloneqq\emptyset.
\end{equation*}
To simplify notation, sometimes, if it is clear from the context, we will simply write $\text{pat}_{b}(\sigma)=\pi$ or $\text{pat}_{e}(\sigma)=\pi,$ instead of $\text{pat}_{b(|\pi|)}(\sigma)=\pi$ or $\text{pat}_{e(|\pi|)}(\sigma)=\pi.$

We finally recall that a Laurent polynomial over $\mathbb{R}$ is a linear combination of positive and negative powers of the variable with coefficients in $\mathbb{R},$ \emph{i.e.}, of the form
\begin{equation*}
P(x)=a_{-m}x^{-m}+\dots+ a_{-1}x^{-1}+a_0+a_1x+\dots+a_n x^n,
\end{equation*}
where $m,n\in\Z_{\geq0}$ and $a_i\in\mathbb{R}$ for all $-m\leq i\leq n.$
We  denote by $O(x^{-\alpha}),$ $\alpha\in\Z_{>0},$ an arbitrary Laurent polynomial in $x$ of valuation at least $-\alpha,$ \emph{i.e.,} of the form $a_{-\alpha}x^{-\alpha}+\dots+a_n x^n.$ Moreover we denote by $O(x^{\alpha}),$ $\alpha\in\mathbb{Z}_{\geq0},$ a classical polynomial in $x$ of valuation at least $\alpha,$ \emph{i.e.}, of the form $a_{\alpha}x^{\alpha}+\dots+a_nx^n.$ We remark that this is \emph{not} the classical definition of $O(\cdot)$ since we are adding the additional hypothesis that the elements of $O(\cdot)$ are Laurent polynomials and not general functions.

\subsubsection{A bijection between binary trees and 231-avoiding permutations}
\label{231bij}
We present in this section a well-known bijection between 231-avoiding permutations and the set of binary trees (see \emph{e.g.,} \cite{bona2012surprising}) which sends the size of the permutation to the number of vertices of the tree. 

This bijection is based on the following result (see \emph{e.g.,} \cite{bona2010absence}). 
\begin{obs} 
	\label{permfact}
	Let $\sigma$ be a permutation and $\ell=\emph{indmax}(\sigma).$ The permutation $\sigma$ avoids $231$ if and only if $\sigma_L$ and $\sigma_R$ both avoid $231$ and furthermore $\sigma_i<\sigma_j$ whenever $i<\ell$ and $j>\ell.$
\end{obs}

Given a permutation $\sigma\in\text{Av}(231)$ we build the following binary tree $T=\varphi(\sigma)$: if $\sigma$ is empty then $T$ is the empty tree. Otherwise we add the root in $T,$ which corresponds to the maximal element of $\sigma$ and we split $\sigma$ in $\sigma_L$ and $\sigma_R;$ then the left subtree of $T$ will be the tree induced by $\sigma_L,$ \emph{i.e.,} $T_L=\varphi(\sigma_L)$ and similarly, the right subtree of $T$ will be the tree induced by $\text{std}(\sigma_R),$ \emph{i.e.,} $T_R=\varphi(\text{std}(\sigma_R)).$

Conversely, given a binary tree $T$ we construct the corresponding permutation $\sigma=\psi(T)$ in $\text{Av}(231)$ as follows (\emph{cf.} Fig.~\ref{bijtreeperm} below): if $T$ is empty then $\sigma$ is the empty permutation. Otherwise  we split $T$ in $T_L$ and $T_R$ and we set $n\coloneqq|T|$, the number of vertices of $T,$ $S_L\coloneqq\psi(T_L)$ and $S_R\coloneqq\psi(T_R).$ Then we define $\sigma$ as $\sigma=\sigma_Ln\sigma_R$ where $\sigma_L$ is simply $S_L$ and $\sigma_R$ is obtained by shifting all entries of $S_R$ by $|T_L|,$ \emph{i.e.,} $(\sigma_R)_j=(S_R)_j+|T_L|,$ for all $j\leq |T_R|.$

It is clear that $\varphi$ and $\psi$ are inverse of each other, hence providing the desired bijection.

\begin{figure}[htbp]
	\begin{center}
		\includegraphics[scale=.70]{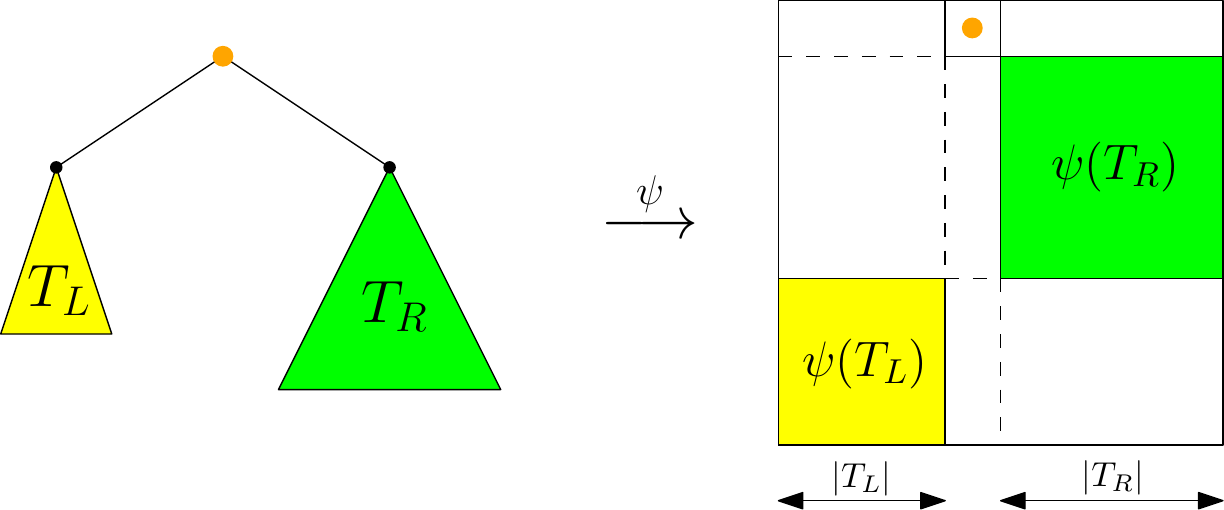}\\
		\caption{A binary tree and the corresponding 231-avoiding permutation represented as a diagram.}\label{bijtreeperm}
	\end{center}
\end{figure}

From now until the end, we denote by $T_{\sigma}$ the binary tree associated to $\sigma$ and, analogously, by $\sigma_T$ the permutation associated to the binary tree $T.$

Moreover, we define 
$$\coc(\pi,T)\coloneqq\coc(\pi,\sigma_T),\quad\text{for all}\quad\pi\in\mathcal{S},$$
and similarly
$$\text{pat}_{I}(T)\coloneqq\text{pat}_I(\sigma_T),\quad\text{ for all finite interval}\quad I\subset\Z_{>0}.$$
\begin{obs}
	\label{maxnode}
	By construction, the left-to-right maxima in a permutation $\sigma$ correspond to the vertices in the left branch of $T_{\sigma}$, \emph{i.e.,} to the vertices of the form $v=1\dots1$ in the Neveu's formalism. Similarly, the right-to-left maxima correspond to the vertices in the right branch of $T_{\sigma}$, \emph{i.e.,} to the vertices of the form $v=2\dots2.$ Obviously the maximum corresponds to the root of the tree. An example is given in Fig.~\ref{maxontree}.
\end{obs}

\begin{figure}[htbp]
	\begin{center}
		\includegraphics[scale=.90]{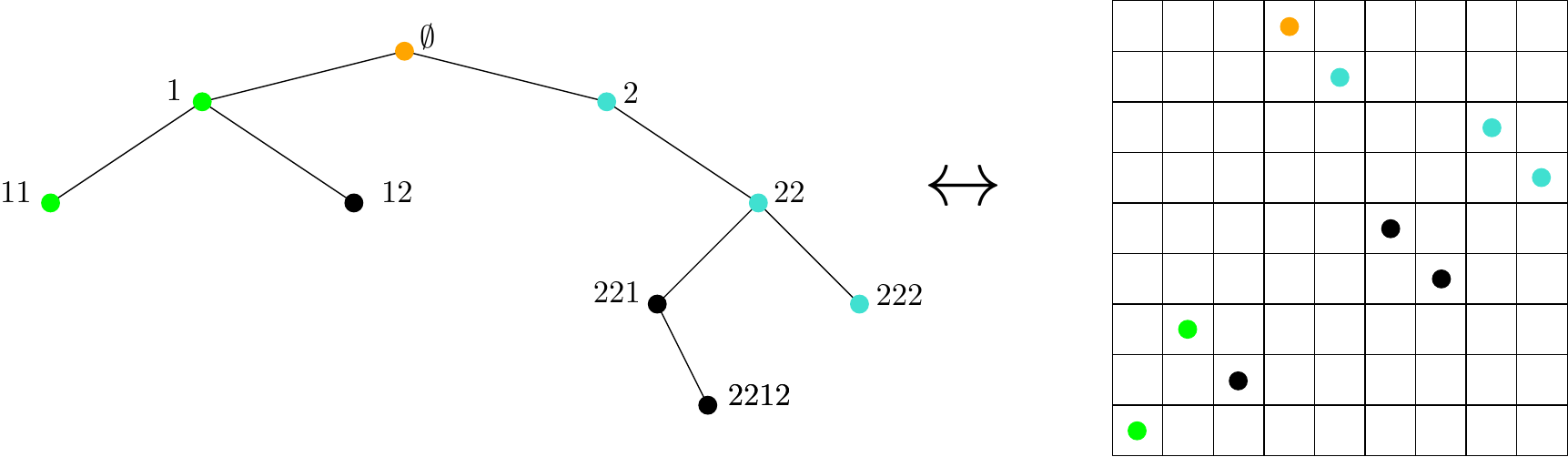}\\
		\caption{Correspondence between maxima of a permutation and vertices on the left and right branches in the associated tree. We highlight in green the left-to-right maxima, in blue the right-to-left maxima and in orange the maximum.
		}\label{maxontree}
	\end{center}
\end{figure}

We make the following final remark that is not strictly necessary in this section but useful for comparison with the next one.

\begin{rem}
	It is not difficult to prove the following equivalent characterization of the bijection between trees and permutations based on the use of the tree traversals. Given a binary tree $T$ with $n$ vertices we reconstruct the corresponding permutation in $\text{Av}(231)$ setting $\sigma_i$ to be equal to the label given by the post-order to the $i$-th vertex visited by the in-order. Namely, we are using the in-order labeling for the indices of the permutation and the post-order labeling for the values. An example is provided in Fig.~\ref{Av(231)bij}.
\end{rem}

\begin{figure}[htbp]
	\begin{center}
		\includegraphics[scale=.90]{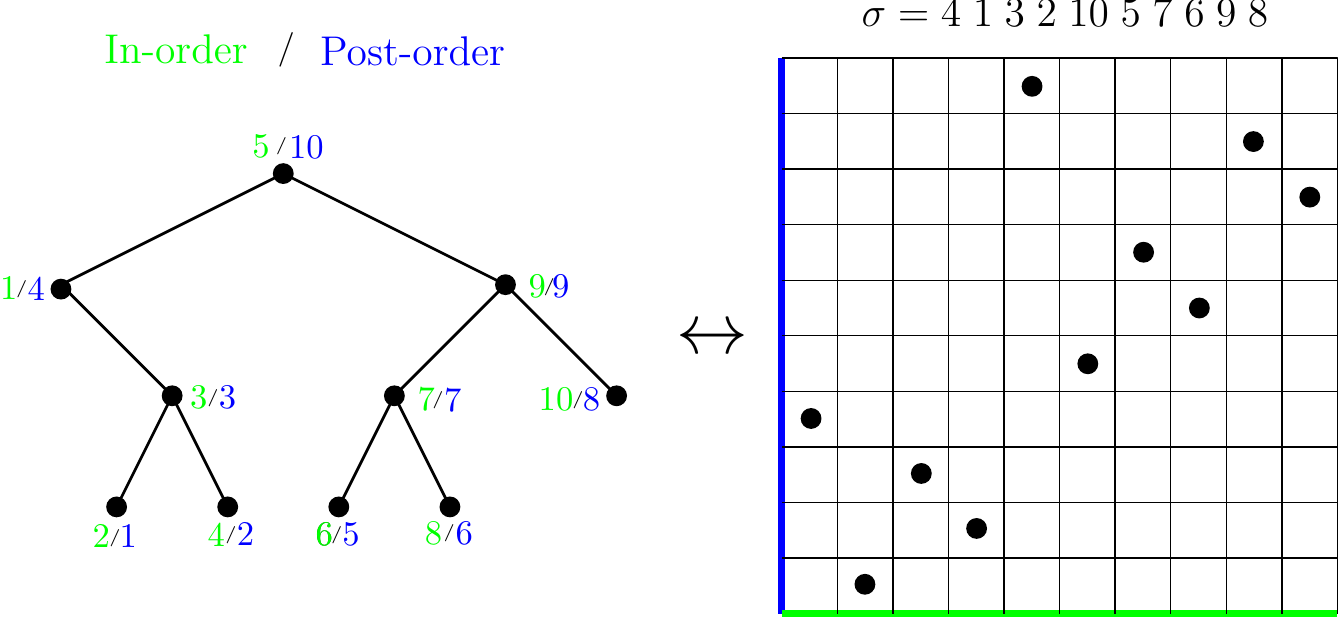}\\
		\caption{A binary tree and the corresponding 231-avoiding permutation. We are using the in-order labeling (in green) for the indices of the permutation and the post-order labeling (in blue) for the values.
		}\label{Av(231)bij}
	\end{center}
\end{figure}

\subsubsection{From uniform 231-avoiding permutations to binary Galton--Watson tree}
\label{behavtree}

Thanks to the bijection between binary trees and 231-avoiding permutations, instead of studying $\E\big[\widetilde{\coc}(\pi,\bm{\sigma}^n)\big]$ we can equivalently study the expectation $\E\big[\widetilde{\coc}(\pi,\bm{T}_n)\big]$ where $\bm{T}_n$ denotes a uniform random binary trees with $n$ vertices. A  method for studying study the behavior of the latter expectation is to analyze the expectation $\E\big[\widetilde{\coc}(\pi,\bm{T}_{\delta})\big]$ for a binary Galton--Watson tree $\bm{T}_\delta$ defined as follows.

For $\delta\in(0,1),$ we set $p=\frac{1-\delta}{2},$ and we consider a binary Galton--Watson tree $\bm{T}_\delta$ with offspring distribution $\eta$ defined by:
\begin{equation}
\label{bintree}
\begin{split}
&\eta_0=\P(\text{The root has 0 child})=(1-p)^2,\\
&\eta_L=\P(\text{The root has only one left child})=p(1-p),\\
&\eta_R=\P(\text{The root has only one right child})=p(1-p),\\
&\eta_2=\P(\text{The root has 2 children})=p^2.
\end{split}
\end{equation}
\begin{rem}
	Note that with this particular offspring distribution, having a left child is independent from having a right child.
\end{rem}
We now state two lemmas due to Janson (see \cite[Lemmas 5.2-5.4]{janson2017patterns}). The first one is a result that rewrites the expectation of a function of $\bm{T}_\delta$ in terms of uniform random binary trees with $n$ vertices $\bm{T}_n,$ using the classical fact that $\mathcal{L}(\bm{T}_n)=\mathcal{L}\big(\bm{T}_{\delta}\big||\bm{T}_{\delta}|=n\big).$ More precisely,
\begin{lem}
	Let $F$ be a functional from the set of binary trees $\mathbb{T}_b$ to $\mathbb{R}$ such that $|F(T)|\leq C\cdot|T|^m$ for some constants $C$ and $m$ (this guarantees that all expectations and sums below converge). Then
	$$\E\big[F(\bm{T}_\delta)\big]=\frac{1+\delta}{1-\delta}\sum_{n=1}^{+\infty}\E\big[F(\bm{T}_n)\big]\cdot C_n\cdot\Big(\frac{1-\delta^2}{4}\Big)^{n},$$
	where $C_n$ is the $n$-th Catalan number.
\end{lem}

Applying singularity analysis (see \cite[Theorem VI.3]{flajolet2009analytic}) one obtains:

\begin{lem}
	\label{svantelemma}
	If $\E\big[F(\bm{T}_\delta)\big]=a\cdot\delta^{-m}+O(\delta^{-(m-1)}),$ for $\delta\to 0,$ where $m\geq 1$ and $a\neq 0,$ then 
	$$\E\big[F(\bm{T}_n)\big]\sim\frac{\Gamma(1/2)}{\Gamma(m/2)}\cdot a\cdot  n^{\frac{m+1}{2}},\quad\text{as}\quad n\to\infty.$$
\end{lem} 

\begin{rem}
	We recall that since by definition the elements of $O(\delta^{-(m+1)})$ are Laurent polynomials (and not general functions) then they are $\Delta$–analytic in $z=\frac{1-\delta^2}{4}$(for a precise definition see \cite[Definition VI.1.]{flajolet2009analytic}) and so the standard hypothesis to apply singularity analysis is fulfilled. 
\end{rem}

Lemma \ref{svantelemma} shows us why it is enough to study the behavior of $\E[\widetilde{\coc}(\pi,\bm{T}_\delta)]$ in order to derive information about the behavior of $\E[\widetilde{\coc}(\pi,\bm{T}_n)].$ This will be the goal of the next section.

\subsection{The annealed version of the Benjamini--Schramm convergence}
\label{weakres}
In this section we are going to prove the following weaker version of Theorem \ref{231theorem}.	
\begin{prop}
	\label{weakprop}
	Let $\bm{\sigma}^n$ be a uniform random permutation in $\emph{Av}^n(231)$ for all $n\in\Z_{>0}.$ It holds that
	$$\E\big[\widetilde{\coc}(\pi,\bm{\sigma}^n)\big]\to P_{231}(\pi),\quad\text{for all}\quad\pi\in\emph{Av}(231).$$
\end{prop}

\subsubsection{A basic recursion}

Observation \ref{permfact} gives us a basic recursion for $\coc(\pi,\sigma)$ as we will show in the next lemmas.
\begin{lem}
	Let $\pi\in\emph{Av}^k(231)$ with $k\geq 1.$ Set $m=\emph{indmax}(\pi).$ Then, for any permutation $\sigma\in\emph{Av}^{n}(231)$ with $n\geq 1,$
	\begin{equation}
	\label{recccoc}
	\coc(\pi,\sigma)=\coc(\pi,\sigma_L)+\coc(\pi,\sigma_R)+\mathds{1}_{\big\{\emph{pat}_{[\ell-m+1,\ell+k-m]}(\sigma)=\pi\big\}},
	\end{equation} 
	where as always $\ell=\emph{indmax}(\sigma).$
\end{lem}
\begin{proof}
	Recall that consecutive occurrences of $\pi$ in $\sigma$ correspond to intervals $I\subseteq[n]$ such that $\text{pat}_I(\sigma)=\pi.$ We have three different possibilities for $I$:
	\begin{enumerate}[1.]
		\item if $I\subseteq[1,\ell-1]$ then $\pi$ is a consecutive occurrence in $\sigma_L$;
		\item if $I\subseteq[\ell+1,n]$ then $\pi$ is a consecutive occurrence in $\sigma_R$;
		\item if $\ell\in I$ then $I$ has to be the interval $[\ell-m+1,\ell+k-m].$ Indeed $\text{Card}(I)=k$ and the value induced in $\text{pat}_{I}(\pi)$ by $\sigma_{\ell}$ has to be $\pi_m.$
	\end{enumerate}
This is enough to prove Equation (\ref{recccoc}).
\end{proof}

We can translate this result in term of trees.

\begin{lem}
	\label{ricors}
	Let $\pi\in\emph{Av}^k(231)$ with $k\geq 1$ and  define $m$ as in the previous lemma. Then, for every binary tree $T,$ denoting $\ell=\emph{indmax}(\sigma_{T}),$
	\begin{equation}
	\label{recccoctrees}
	\coc(\pi,T)=\coc(\pi,T_L)+\coc(\pi,T_R)+\mathds{1}_{\big\{\emph{pat}_{[\ell-m+1,\ell+k-m]}(T)=\pi\big\}}.
	\end{equation}  
\end{lem}

\subsubsection{The behaviour of $\E\big[\widetilde{\coc}(\pi,\bm{\sigma}^n)\big]$.}
We now focus on the behavior of $\E\big[\widetilde{\coc}(\pi,\bm{T}_\delta)\big]$ for\break
$\pi\in\text{Av}(231),$ and then we will recover results for $\E\big[\widetilde{\coc}(\pi,\bm{\sigma}^n)\big]$ using Lemma \ref{svantelemma} and the bijection between binary trees and 231-avoiding permutation as explained in Section \ref{behavtree}. In order to simplify notation we set $\bm{T}\coloneqq\bm{T}_\delta.$ Thanks to Lemma \ref{ricors}, we know that, for all $\pi\in\text{Av}(231),$
\begin{equation}
\label{twostar}
\coc(\pi,\bm{T})=\coc(\pi,\bm{T}_L)+\coc(\pi,\bm{T}_R)+\mathds{1}_{\big\{\text{pat}_{\bm{J}}(\bm{T})=\pi\big\}},
\end{equation}
where $\bm{J}=[\bm{\ell}-m+1,\bm{\ell}+k-m]$, $\bm{\ell}=\text{indmax}(\bm{\sigma}_{\bm{T}})$ and $m=\text{indmax}(\pi).$ 
Taking the expectation in Equation (\ref{twostar}) we obtain,
$$\E\big[\coc(\pi,\bm{T})\big]=\E\big[\coc(\pi,\bm{T}_L)\big]+\E\big[\coc(\pi,\bm{T}_R)\big]+\P\big(\text{pat}_{\bm{J}}(\bm{T})=\pi\big).$$
Since $\bm{T}_L$ is an independent copy of $\bm{T}$ with probability $p$ and empty with probability $1-p,$ and the same holds for $\bm{T}_R,$ we have,
\begin{equation}
\label{step1}
\E\big[\coc(\pi,\bm{T})\big]=\frac{\P\big(\text{pat}_{\bm{J}}(\bm{T})=\pi\big)}{1-2p}=\delta^{-1}\cdot\P\big(\text{pat}_{\bm{J}}(\bm{T})=\pi\big),
\end{equation}
where in the last equality we used that $p=\frac{1-\delta}{2}.$
We now focus on the term $\P\big(\text{pat}_{\bm{J}}(\bm{T})=\pi\big).$ 

\begin{lem}
	\label{onecross}
	Let $\pi\in\emph{Av}(231).$ Using notation as above and decomposing as usual $\pi$ in\\ $\pi=\pi_L\pi_m\pi_R,$
	\begin{equation}
	\label{oneredcross}
	\P\big(\emph{pat}_{\bm{J}}(\bm{T})=\pi\big)=
	\begin{cases}
	p^2\cdot\P\big(\emph{pat}_e(\bm{T})=\pi_L\big)\cdot\P\big(\emph{pat}_b(\bm{T})=\pi_R\big), &\quad\text{if }\pi_L\neq\emptyset,\;\pi_R\neq\emptyset,\\
	p\cdot\P\big(\emph{pat}_e(\bm{T})=\pi_L\big), &\quad\text{if }\pi_L\neq\emptyset,\;\pi_R=\emptyset,\\
	p\cdot\P\big(\emph{pat}_b(\bm{T})=\pi_R\big), &\quad\text{if }\pi_L=\emptyset,\;\pi_R\neq\emptyset,\\
	1, &\quad\text{if }\pi=1.
	\end{cases}
	\end{equation}
\end{lem}
\begin{proof}
	First of all note that
	\begin{equation} \label{keyeq}
	\P\big(\text{pat}_{\bm{J}}(\bm{T})=\pi\big)=\P\big(\text{pat}_e(\bm{T}_L)=\pi_L,\text{pat}_b(\bm{T}_R)=\pi_R\big).
	\end{equation}
	We recall that if $\pi_L$ (resp. $\pi_R$) is empty then $\text{pat}_e(\bm{T}_L)$ (resp. $\text{pat}_b(\bm{T}_R)$) is empty by definition.
	
	We consider the case when $\pi_L\neq\emptyset,\;\pi_R\neq\emptyset.$ We have,
		$$\P\big(\text{pat}_{\bm{J}}(\bm{T})=\pi\big)=p^2\cdot\P\big(\text{pat}_e(\bm{T})=\pi_L\big)\cdot\P\big(\text{pat}_b(\bm{T})=\pi_R\big),$$
		since in Equation (\ref{keyeq}) the random tree $\bm{T}_L$ is an independent copy of $\bm{T}$ with probability $p$ and empty with probability $1-p$ and obviously the same hold also for $\bm{T}_R.$
		The other three cases are similar.		
\end{proof}
In view of Lemma \ref{onecross}, we now focus on $\P\big(\text{pat}_e(\bm{T})=\pi)$ (the analysis for $\P\big(\text{pat}_b(\bm{T})=\pi)$ following by symmetry). We want to rewrite the event $\big\{\text{pat}_e(\bm{T})=\pi\big\}$ conditioning on the position of the maximum among the last $|\pi|$ values of $\sigma_{\bm{T}}.$ Using Observation \ref{maxnode}, we know that this maximum is always reached at an index of $\sigma_{\bm{T}}$ corresponding to a vertex of $\bm{T}$ of the form
$$v=\underbrace{2\dots2}_{n\text{-times}}\eqqcolon v_{2^n},\quad\text{for some}\quad n\in\Z_{\geq 0},$$
with the convention $v_{2^0}\coloneqq\emptyset.$ For an example, see Fig.~\ref{lastmax}.
\begin{figure}[htbp]
	\begin{center}
		\includegraphics[scale=.70]{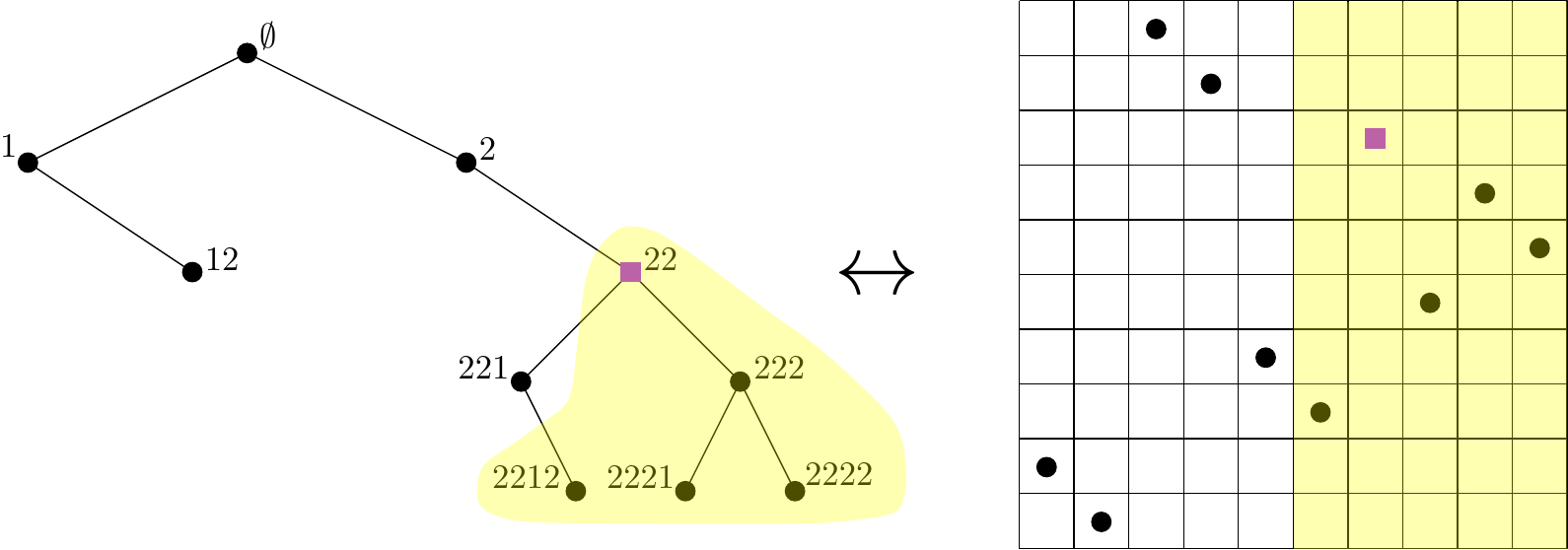}\\
		\caption{We marked with a purple square the maximum among the last 5 values (highlighted in yellow) of the permutation $\sigma=2\;1\;10\;9\;4\;3\;8\;5\;7\;6.$}\label{lastmax}
	\end{center}
\end{figure}

Therefore, defining the following events, for all $n\in\Z_{\geq0},$ all $\pi\in\text{Av}(231),$
\begin{equation*}
\begin{split}
M^n_{\pi}\coloneqq\big\{&v_{2^n}\text{ is the vertex in } \bm{T} \text{ corresponding to the}\\
&\text{maximum among the last }|\pi|\text{ values of }\sigma_{\bm{T}}\big\}
\end{split}
\end{equation*}
and using the formula of total probability, we have
\begin{equation}
\label{disintform}
\P\big(\text{pat}_e(\bm{T})=\pi)=\sum_{n=0}^\infty\P\big(\text{pat}_e(\bm{T})=\pi,M^n_{\pi}).
\end{equation}
We also introduce the events, for all $n\in\Z_{\geq0},$
$$R^n\coloneqq\big\{\text{The vertex }v_{2^n}\text{ is in }\bm{T}\big\}.$$
Note that 
\begin{equation}
\label{sumrn}
\sum_{n=0}^\infty\mathbb{P}(R^n)=\sum_{n=0}^\infty p^n=\frac{1}{1-p}.
 \end{equation}
The following lemma computes recursively $\P\big(\text{pat}_e(\bm{T})=\pi\big)$ and $\P\big(\text{pat}_b(\bm{T})=\pi\big).$
\begin{lem}
	\label{twoandthreecross}
	Let $\pi\in\emph{Av}(231).$ Using notation as above,
	\begin{equation}
	\label{twocross}
	\P\big(\emph{pat}_e(\bm{T})=\pi\big)=
	\begin{cases}
	\frac{p^2}{1-p}\cdot\P\big(\emph{pat}_e(\bm{T})=\pi_L\big)\cdot\P\big(\bm{T}=T_{\pi_R}\big), &\quad\text{if }\pi_L\neq\emptyset,\;\pi_R\neq\emptyset,\\
	p\cdot\P\big(\emph{pat}_e(\bm{T})=\pi_L\big), &\quad\text{if }\pi_L\neq\emptyset,\;\pi_R=\emptyset,\\
	\frac{p}{1-p}\cdot\P\big(\bm{T}=T_{\pi_R}\big), &\quad\text{if }\pi_L=\emptyset,\;\pi_R\neq\emptyset,\\
	1, &\quad\text{if }\pi=1,
	\end{cases}
	\end{equation}
and
	\begin{equation}
	\label{threecross}
	\P\big(\emph{pat}_b(\bm{T})=\pi\big)=
	\begin{cases}
	\frac{p^2}{1-p}\cdot\P\big(\bm{T}=T_{\pi_L}\big)\cdot\P\big(\emph{pat}_b(\bm{T})=\pi_R\big), &\quad\text{if }\pi_L\neq\emptyset,\;\pi_R\neq\emptyset,\\
	\frac{p}{1-p}\cdot\P\big(\bm{T}=T_{\pi_L}\big), &\quad\text{if }\pi_L\neq\emptyset,\;\pi_R=\emptyset,\\
	p\cdot\P\big(\emph{pat}_b(\bm{T})=\pi_R\big), &\quad\text{if }\pi_L=\emptyset,\;\pi_R\neq\emptyset,\\
	1, &\quad\text{if }\pi=1.
	\end{cases}
	\end{equation}
\end{lem}

\begin{proof}
	We start with the study of $\P\big(\text{pat}_e(\bm{T})=\pi\big).$ We distinguish again four different cases according to the structure of $\pi_L$ and $\pi_R$.
	\begin{itemize}
		\item \underline{$\pi_L\neq\emptyset,\;\pi_R\neq\emptyset:$} By Equation (\ref{disintform}), we know that 
		\begin{equation*}
		\begin{split}
		\P\big(\text{pat}_e(\bm{T})=\pi)&=\sum_{n\in\Z_{\geq0}}\P\big(\text{pat}_e(\bm{T})=\pi,M^n_{\pi})\\
		&=\sum_{n\in\Z_{\geq0}}\P\big(\text{pat}_e(\bm{T}^{v_{2^n}}_L)=\pi_L,\bm{T}^{v_{2^n}}_R=T_{\pi_R},M^n_{\pi}),
		\end{split}
		\end{equation*}
		where in the last equality we used that conditioning on $v_{2^n}$ being the vertex in $\bm{T}$ corresponding to the
		maximum among the last $|\pi|$ values of $\sigma_{\bm{T}},$ then $\text{pat}_{e(|\pi|)}(\bm{T})=\text{pat}_{e(|\pi|)}(\bm{T}^{v_{2^n}}).$
		Since the event $\big\{\text{pat}_e(\bm{T}^{v_{2^n}}_L)=\pi_L\big\}\cap\big\{\bm{T}^{v_{2^n}}_R=T_{\pi_R}\big\}$ is contained both in $M^n_{\pi}$ and in $R^n,$ then
		\begin{equation*}
		\P\big(\text{pat}_e(\bm{T})=\pi)=\sum_{n\in\Z_{\geq0}}\P\big(\text{pat}_e(\bm{T}^{v_{2^n}}_L)=\pi_L,\bm{T}^{v_{2^n}}_R=T_{\pi_R},R^n).
		\end{equation*}
		Using the independence between $\bm{T}^{v_{2^n}}_L$ and $\bm{T}^{v_{2^n}}_R$ conditionally on $R^n$ and continuing the sequence of equalities,
		\MLine{
		\P\big(\text{pat}_e(\bm{T})=\pi)=\sum_{n\in\Z_{\geq0}}\P\big(\text{pat}_e(\bm{T}^{v_{2^n}}_L)=\pi_L|R^n\big)\cdot\P\big(\bm{T}^{v_{2^n}}_R=T_{\pi_R}|R^n\big)\cdot\P\big(R^n).
		}
		Since, conditionally on $R^n,$ $\bm{T}^{v_{2^n}}_L$ is an independent copy of $\bm{T}$   with probability $p$ and empty with probability $1-p$ and the same obviously holds for $\bm{T}^{v_{2^n}}_R,$ we can rewrite the last term as
		\MLine{
		\begin{split}
		\P\big(\text{pat}_e(\bm{T})=\pi)&=\sum_{n\in\Z_{\geq0}}p^2\cdot\P\big(\text{pat}_e(\bm{T})=\pi_L\big)\cdot\P\big(\bm{T}=T_{\pi_R}\big)\cdot\P\big(R^n)\\
		&=p^2\cdot\P\big(\text{pat}_e(\bm{T})=\pi_L\big)\cdot\P\big(\bm{T}=T_{\pi_R}\big)\cdot\sum_{n\in\Z_{\geq0}}\P\big(R^n)\\
		&=\frac{p^2}{1-p}\cdot\P\big(\text{pat}_e(\bm{T})=\pi_L\big)\cdot\P\big(\bm{T}=T_{\pi_R}\big),
		\end{split}
		}
		where in the last equality we used Equation (\ref{sumrn}).
		\item \underline{$\pi_L\neq\emptyset,\;\pi_R=\emptyset:$} Similarly as before
		\begin{equation*}
		\begin{split}
		\P\big(\text{pat}_e(\bm{T})=\pi)&=\sum_{n\in\Z_{\geq0}}\P\big(\text{pat}_e(\bm{T}^{v_{2^n}}_L)=\pi_L,\bm{T}^{v_{2^n}}_R=\emptyset|R^n\big)\cdot\P\big(R^n).
		\end{split}
		\end{equation*}
		Noting that with similar arguments as before we have $$\P\big(\text{pat}_e(\bm{T}^{v_{2^n}}_L)=\pi_L,\bm{T}^{v_{2^n}}_R=\emptyset|R^n\big)=p\cdot\P\big(\text{pat}_e(\bm{T})=\pi_L\big)\cdot(1-p).$$ 
		and using again Equation (\ref{sumrn}), we can conclude that
		$$\P\big(\text{pat}_e(\bm{T})=\pi)=p\cdot\P\big(\text{pat}_e(\bm{T})=\pi_L\big).$$
		\item \underline{$\pi_L=\emptyset,\;\pi_R\neq\emptyset:$} Again
		\begin{equation*}
		\begin{split}
		\P\big(\text{pat}_e(\bm{T})=\pi)&=\sum_{n\in\Z_{\geq0}}\P\big(\text{pat}_e(\bm{T}^{v_{2^n}}_L)=\emptyset,\bm{T}^{v_{2^n}}_R=T_{\pi_R}|R^n\big)\cdot\P\big(R^n).
		\end{split}
		\end{equation*}
		Noting that $\text{pat}_{e(0)}(\bm{T}^{v_{2^n}}_L)$ is empty by definition and using similar arguments as before, we have $\P\big(\text{pat}_e(\bm{T}^{v_{2^n}}_L)=\emptyset,\bm{T}^{v_{2^n}}_R=T_{\pi_R}|R^n\big)=p\cdot\P\big(\bm{T}_R=T_{\pi_R}\big).$ Therefore, using again Equation (\ref{sumrn}), we can write 
		$$\P\big(\text{pat}_e(\bm{T})=\pi)=\frac{p}{1-p}\cdot\P\big(\bm{T}_R=T_{\pi_R}\big).$$
		\item \underline{$\pi_L=\emptyset,\;\pi_R=\emptyset:$} Since $\pi=1$ then
		$\P\big(\text{pat}_e(\bm{T})=\pi)=1.$	
	\end{itemize}
	Since the result for $\P\big(\text{pat}_b(\bm{T})=\pi\big)$ follows by symmetry, we can conclude the proof. 
\end{proof}

We now continue the analysis of $\P\big(\text{pat}_{\bm{J}}(\bm{T})=\pi\big).$ In order to do that, we need a formula for $\P(\bm{T}=T),$ for a given tree $T,$ because such probabilities appear in Equations (\ref{twocross}) and (\ref{threecross}).

\begin{obs}
	\label{obs:bintreecomp}
	In a binary tree with $n$ vertices, every vertex has two potential children. Out of these $2n$ potential children $n-1$ exist and $n+1$ do not exist. Hence
	\begin{equation}
	\label{bintreecomp}
	\P\big(\bm{T}=T\big)=p^{|T|-1}\cdot(1-p)^{|T|+1}.
	\end{equation}	
\end{obs}
Using together Lemmas \ref{onecross} and \ref{twoandthreecross} and the above observation we have an explicit recursion to compute the probability $\P\big(\text{pat}_{\bm{J}}(\bm{T})=\pi\big).$ We show an example of the recursion obtained for an explicit pattern $\pi$.

\begin{exmp}
	\label{bigexemp}
	Let $\pi$ be the following 231-avoiding permutation,  
	$$\pi=4\;1\;3\;2\;6\;5\;7\;10\;8\;9\;11\;12\;16\;13\;15\;14=
	\begin{array}{lcr}
	\begin{tikzpicture}
	\begin{scope}[scale=.3]
	\permutation{4,1,3,2,6,5,7,10,8,9,11,12,16,13,15,14}
	\draw (1+.5,4+.5) [green, fill] circle (.21);  
	\draw (5+.5,6+.5) [green, fill] circle (.21); 
	\draw (7+.5,7+.5) [green, fill] circle (.21); 
	\draw (8+.5,10+.5) [green, fill] circle (.21); 
	\draw (11+.5,11+.5) [green, fill] circle (.21); 
	\draw (12+.5,12+.5) [green, fill] circle (.21);
	\draw (13+.5,16+.5) [orange, fill] circle (.21);
	\draw (15+.5,15+.5) [blue, fill] circle (.21); 
	\draw (16+.5,14+.5) [blue, fill] circle (.21); 
	\end{scope}
	\end{tikzpicture}
	\end{array},$$
	where, as before, we draw in green the left-to-right maxima, in blue the right-to-left maxima, and in orange the maximum.
	
	We now make the explicit computation for $\P\big(\text{pat}_{\bm{J}}(\bm{T})=\pi\big)$ using Lemmas \ref{onecross} and \ref{twoandthreecross}. First of all, we recursively split our permutation as shown in Fig.~\ref{decomp_tree}.
	
	\begin{figure}[htbp]
		\begin{center}
			\includegraphics[scale=.70]{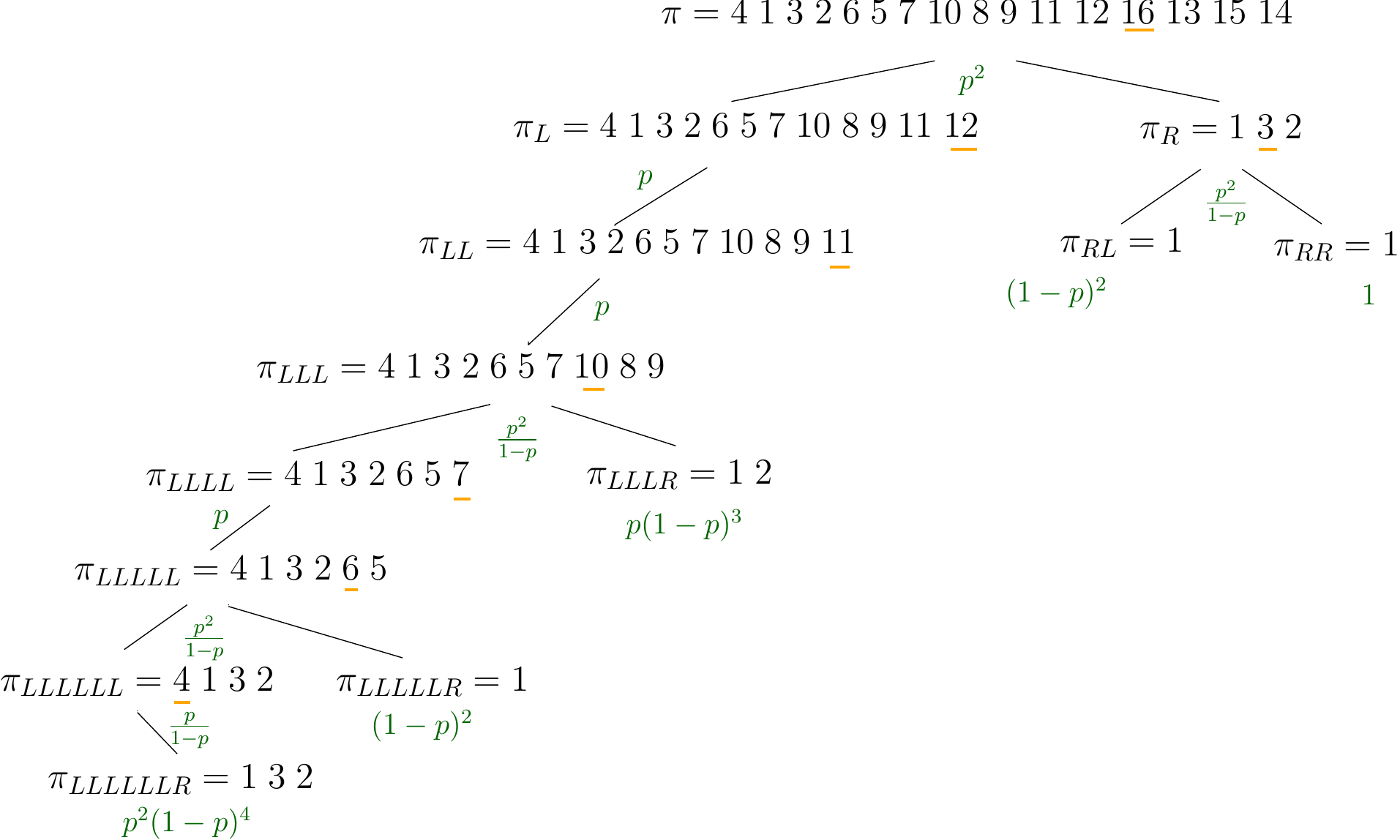}\\
			\caption{We draw the recursive decompositions in the left and right part (w.r.t.\ the position of the maximum underlined in orange) for the permutation $\pi=4\;1\;3\;2\;6\;5\;7\;10\;8\;9\;11\;12\;16\;13\;15\;14.$ Moreover we write in green the factors $p^\alpha\cdot(1-p)^{\beta}$ that we are adding at each step (coming from Lemmas \ref{onecross} and \ref{twoandthreecross}).}\label{decomp_tree}
		\end{center}
	\end{figure}
	
	Then, using Equation (\ref{oneredcross}) with the decomposition of $\pi$ in $\pi_L\pi_m\pi_R$ (shown at the root of the tree in Fig.~\ref{decomp_tree}),  we have
	\begin{equation*}
	\P\big(\text{pat}_{\bm{J}}(\bm{T})=\pi\big)= p^2\cdot\P\big(\text{pat}_e(\bm{T})=\pi_L\big)\cdot\P\big(\text{pat}_b(\bm{T})=\pi_R\big).
	\end{equation*}
	We continue decomposing $\pi_L$ and $\pi_R$ around their maxima (which correspond to the left and right children of the root in the tree in Fig.~\ref{decomp_tree}). Using Equation (\ref{twocross}) for the left part, we obtain 
	\begin{equation*}
	\P\big(\text{pat}_e(\bm{T})=\pi_L\big)=p\cdot\P\big(\text{pat}_e(\bm{T})=\pi_{LL}\big),
	\end{equation*}
	and using Equation (\ref{threecross}) for the right part, we obtain
	\begin{equation*}
	\P\big(\text{pat}_b(\bm{T})=\pi_R\big)=\frac{p^2}{1-p}\cdot\underbrace{\P\big(\bm{T}=T_{\pi_{RL}}\big)}_{\stackrel{(\ref{bintreecomp})}{=}(1-p)^2}\cdot\underbrace{\P\big(\text{pat}_b(\bm{T})=1\big)}_{1},
	\end{equation*}
	where the first probability in the right-hand side is computed using Equation (\ref{bintreecomp}).
	Therefore, summing up the last three equations and then proceeding similarly through the left subtree in Fig.~\ref{decomp_tree}, we deduce that $\P\big(\text{pat}_{\bm{J}}(\bm{T})=\pi\big)$ is the product of all the green factors in Fig.~\ref{decomp_tree}, that is
	\begin{equation*}
	\P\big(\text{pat}_{\bm{J}}(\bm{T})=\pi\big)=p^{15}(1-p)^{7}=\Big(\frac{1-\delta}{2}\Big)^{15}\cdot\Big(\frac{1+\delta}{2}\Big)^{7}=\Big(\frac{1}{2}\Big)^{22}+O(\delta),
	\end{equation*}
	and this concludes Example \ref{bigexemp}.
\end{exmp}

We now proceed with the analysis of the general case. Using the recursion obtained by combining Lemmas \ref{onecross} and \ref{twoandthreecross} and Observation \ref{obs:bintreecomp}, we immediately realize that
\begin{equation}
\label{vov0thg}
\P\big(\text{pat}_{\bm{J}}(\bm{T})=\pi\big)=p^{\alpha}\cdot(1-p)^{\beta},\quad\text{for some}\quad\alpha,\beta\in\Z_{\geq0}.
\end{equation}

Note that $\alpha\geq 0$ since in Lemmas \ref{onecross} and \ref{twoandthreecross} and Observation \ref{obs:bintreecomp} the $p^*$ factors always appear with positive exponent. Moreover $\beta\geq 0$ since in Equations (\ref{twocross}) and (\ref{threecross}), each time a $(1-p)^{-1}$ factor appears, there is also a $\P\big(\bm{T}=T\big)$ factor that contains a $(1-p)^{\gamma}$ factor with $\gamma\geq2.$

Since $p=\frac{1-\delta}{2},$ it follows that
$$\P\big(\text{pat}_{\bm{J}}(\bm{T})=\pi\big)=\Big(\frac{1}{2}-\frac{\delta}{2}\Big)^{\alpha}\cdot\Big(\frac{1}{2}+\frac{\delta}{2}\Big)^{\beta}=\Big(\frac{1}{2}\Big)^{\alpha+\beta}+O(\delta),\quad\text{for some}\quad\alpha,\beta\in\Z_{\geq0}.$$
We now compute the value $\alpha+\beta.$ In order to do that, we have to compute how many factors $p$ and $(1-p)$ appear in $\P\big(\text{pat}_{\bm{J}}(\bm{T})=\pi\big).$

Before stating our result we need to introduce a notion of distance between maxima. Given an index $j\in\text{LRMax}(\pi)$ which is not the maximum of $\pi$, we define its distance from the following left-to-right maximum as 
$$\partial_{\pi}(j)\coloneqq\min\big\{a\geq0:j+1+a\in\text{LRMax}(\pi)\big\}.$$
Analogously, given an index $j\in\text{RLMax}(\pi)$ which is not the maximum of $\pi$, we define its distance from the following right-to-left maximum as 
$$\partial_{\pi}(j)\coloneqq\min\big\{a\geq0:j-1-a\in\text{RLMax}(\pi)\big\}.$$
By convention, if $j$ is the index of the maximum of $\pi$, then $\partial_{\pi}(j)\coloneqq 0.$

\begin{exmp}
	We continue Example \ref{exempmax}, where we considered the permutation
		$$\pi=132985476=
		\begin{array}{lcr}
		\begin{tikzpicture}
		\begin{scope}[scale=.3]
		\permutation{1,3,2,9,8,5,4,7,6}
		\draw (1+.5,1+.5) [green, fill] circle (.21);  
		\draw (2+.5,3+.5) [green, fill] circle (.21); 
		\draw (4+.5,9+.5) [orange, fill] circle (.21);
		\draw (5+.5,8+.5) [blue, fill] circle (.21);
		\draw (8+.5,7+.5) [blue, fill] circle (.21); 
		\draw (9+.5,6+.5) [blue, fill] circle (.21); 
		\end{scope}
		\end{tikzpicture}
		\end{array}.$$
		For LRMax($\pi$)$=\{1,2,4\}$ we have $\partial_{\pi}(1)=0,$ $\partial_{\pi}(2)=1,$ $\partial_{\pi}(4)=0$ and for RLMax($\pi$)$=\{9,8,5,4\}$ we have $\partial_{\pi}(9)=0,$ $\partial_{\pi}(8)=2,$ $\partial_{\pi}(5)=0$  and $\partial_{\pi}(4)=0.$ 
\end{exmp}

We are now ready to state the key result of this section.

\begin{prop}
	\label{patprob231}
	Let $\pi\in\emph{Av}(231)$ and $\bm{T}=\bm{T}_{\delta}$ be a Galton--Watson tree defined as above, then
	\begin{equation}
	\label{probpat231}
	\P\big(\emph{pat}_{\bm{J}}(\bm{T})=\pi\big)=\frac{2^{|\emph{LRMax}(\pi)|+|\emph{RLMax}(\pi)|}}{2^{2|\pi|}}+O(\delta)=P_{231}(\pi)+O(\delta).
	\end{equation}
\end{prop}

Note that Proposition \ref{patprob231} claims that for the specific pattern $\pi$ given in Example \ref{bigexemp} we have,
\begin{equation}
	\label{predresult}
	\P\big(\text{pat}_{\bm{J}}(\bm{T})=\pi\big)=\frac{2^{7+3}}{2^{32}}+O(\delta)=\Big(\frac{1}{2}\Big)^{22}+O(\delta),
\end{equation}
as previously computed. With this example in our hands, we can now use the same ideas in order to prove the proposition.

\begin{proof}[Proof of Proposition \ref{patprob231}]
Suppose that we are able to prove that
\begin{equation}
\label{goal}
\P\big(\text{pat}_{\bm{J}}(\bm{T})=\pi\big)=\Big(\frac{1}{2}\Big)^{|\text{Max}(\pi)|-1+2\sum_{j\in \text{Max}(\pi)}\partial_{\pi}(j)}+O(\delta).
\end{equation}
Then noting that $|\text{Max}(\pi)|+\sum_{j\in \text{Max}(\pi)}\partial_{\pi}(j)=|\pi|$ and that $$|\text{Max}(\pi)|+1=|\text{LRMax}(\pi)|+|\text{RLMax}(\pi)|,$$ we immediately derives the desired expression
\begin{equation*}
\P\big(\text{pat}_{\bm{J}}(\bm{T})=\pi\big)=\frac{2^{|\text{LRMax}(\pi)|+|\text{RLMax}(\pi)|}}{2^{2|\pi|}}+O(\delta).
\end{equation*}
We now prove Equation (\ref{goal}). As explained in the discussion below Equation (\ref{vov0thg}), we just need to count how many factors $p$ and $(1-p)$ appear in the formula for $\P\big(\text{pat}_{\bm{J}}(\bm{T})=\pi\big)$. We can suppose that $|\pi|\geq 2$ (if $|\pi|=1$ then the statement is trivial). Then by Equation (\ref{oneredcross}),
\begin{equation}
\label{onegreenstar}
\P\big(\text{pat}_{\bm{J}}(\bm{T})=\pi\big)=
\begin{cases}
p^2\cdot\P\big(\text{pat}_e(\bm{T})=\pi_L\big)\cdot\P\big(\text{pat}_b(\bm{T})=\pi_R\big), &\quad\text{if }\pi_L\neq\emptyset,\;\pi_R\neq\emptyset,\\
p\cdot\P\big(\text{pat}_e(\bm{T})=\pi_L\big), &\quad\text{if }\pi_L\neq\emptyset,\;\pi_R=\emptyset,\\
p\cdot\P\big(\text{pat}_b(\bm{T})=\pi_R\big), &\quad\text{if }\pi_L=\emptyset,\;\pi_R\neq\emptyset,
\end{cases}
\end{equation}
and so we obtain either a factor $p\sim\frac{1}{2}$ (if the maximum of $\pi$ is reached at the first or at the last index of $\pi$) or a factor $p^2\sim\Big(\frac{1}{2}\Big)^{2}$ (otherwise). We recall that the asymptotics are for $\delta\to 0.$

We now look at the factors of the form $\P\big(\text{pat}_e(\bm{T})=\pi_L\big).$ By Equation (\ref{twocross}),
\begin{equation}
\label{ergwtrgtrwgtwg}
\P\big(\text{pat}_e(\bm{T})=\pi_L\big)=
\begin{cases}
\frac{p^2}{1-p}\cdot\P\big(\bm{T}=T_{\pi_{LR}}\big)\cdot\P\big(\text{pat}_e(\bm{T})=\pi_{LL}\big), &\quad\text{if }\pi_{LL}\neq\emptyset,\;\pi_{LR}\neq\emptyset,\\
p\cdot\P\big(\text{pat}_e(\bm{T})=\pi_{LL}\big), &\quad\text{if }\pi_{LL}\neq\emptyset,\;\pi_{LR}=\emptyset,\\
\frac{p}{1-p}\cdot\P\big(\bm{T}=T_{\pi_{LR}}\big), &\quad\text{if }\pi_{LL}=\emptyset,\;\pi_{LR}\neq\emptyset,\\
1, &\quad\text{if }\pi_L=1.
\end{cases}
\end{equation}
We are going to focus on the factors that are not of the form $\P\big(\text{pat}_e(\bm{T})=\pi_{LL}\big).$ We call them \emph{prefactors}. For example, when $\pi_{LL}\neq\emptyset$ and $\pi_{LR}\neq\emptyset$ the prefactor is $\frac{p^2}{1-p}\cdot\P\big(\bm{T}=T_{\pi_{LR}}\big).$

Therefore, denoting by $j\coloneqq\text{indmax}(\pi_L)$, we obtain:
\begin{enumerate}
	\item for $|\pi_L|\geq 2:$ 
		\begin{itemize}
			\item if $j\neq1$ and $j\neq |\pi_L|,$ then we have the following prefactor,
			 $$\frac{p^2\cdot p^{\partial_{\pi}(j)-1}\cdot (1-p)^{\partial_{\pi}(j)+1}}{1-p}\sim\Big(\frac{1}{2}\Big)^{1+2\partial_{\pi}(j)};$$ 
			\item if $j= |\pi_L|$ (in this case $\partial_{\pi}(j)=0$), then we have the following prefactor, $$p\sim\frac{1}{2}=\Big(\frac{1}{2}\Big)^{1+2\partial_{\pi}(j)};$$ 
			\item if $j=1$ (in this case $\partial_{\pi}(j)=|\pi_{LR}|$), then we have the following prefactor,
			$$\frac{p\cdot p^{\partial_{\pi}(j)-1}\cdot(1-p)^{\partial_{\pi}(j)+1}}{1-p}\sim\Big(\frac{1}{2}\Big)^{2\partial_{\pi}(j)};$$ 
		\end{itemize}
	\item for $|\pi_L|=1,$ \emph{i.e.,} $\pi_L=1:$
		\begin{itemize}
			\item in this case the index of the trivial maximum is $j=1$ and obviously $\partial_{\pi}(j)=0,$ hence we have the following prefactor,
			$$1=\Big(\frac{1}{2}\Big)^{2\partial_{\pi}(j)}.$$ 
		\end{itemize}
\end{enumerate}
Summing up, and noting that the analysis of the term $\P\big(\text{pat}_b(\bm{T})=\pi_R\big)$ is symmetric, we conclude that we have the following correspondence between maxima of $\pi$ and $\Big(\frac{1}{2}\Big)^*$ factor in the asymptotic expansion of $\P\big(\text{pat}_{\bm{J}}(\bm{T})=\pi\big):$
\begin{enumerate}[(a)]
	\item if $j$ is the index of the maximum of $\pi$, by the comment below Equation (\ref{onegreenstar}), its contribution is
	\begin{equation*}
	\begin{split}
	\Big(&\frac{1}{2}\Big)^2,\quad\text{if }\quad j\neq 1 \text{ and }j\neq |\pi|,\\
	&\frac{1}{2},\qquad\,\text{otherwise;}
	\end{split}
	\end{equation*}
	\item if $j\in\text{Max}(\pi),$ $j\neq1$ and $j\neq|\pi|,$ but $j$ is not the index of the maximum, by the previous analysis, its contribution is
	\begin{equation*}
	\Big(\frac{1}{2}\Big)^{1+2\partial_{\pi}(j)};
	\end{equation*}
	\item if $j\in\text{Max}(\pi),$ $j=1$ or $j=|\pi|,$ but $j$ is not the index of the maximum, again by the previous analysis, its contribution is
	\begin{equation*}
	\Big(\frac{1}{2}\Big)^{2\partial_{\pi}(j)}.
	\end{equation*}
\end{enumerate}
Finally, we note that,
\begin{itemize}
	\item if $\text{indmax}(\pi)=1$ or $\text{indmax}(\pi)=|\pi|$ then
	$$\P\big(\text{pat}_{\bm{J}}(\bm{T})=\pi\big)=\Big(\frac{1}{2}\Big)^{1+|\text{Max}(\pi)|-2+2\sum_{j\in \text{Max}(\pi)}\partial_{\pi}(j)}+O(\delta),$$
	where the term 1 in the exponent comes from $(a)$ and the term $|\text{Max}(\pi)|-2$ comes from $(b);$
	\item if $\text{indmax}(\pi)\neq1$ and $\text{indmax}(\pi)\neq|\pi|$ then
	$$\P\big(\text{pat}_{\bm{J}}(\bm{T})=\pi\big)=\Big(\frac{1}{2}\Big)^{2+|\text{Max}(\pi)|-3+2\sum_{j\in \text{Max}(\pi)}\partial_{\pi}(j)}+O(\delta),$$
	where the term 2 in the exponent comes from $(a)$ and the term $|\text{Max}(\pi)|-3$ comes from $(b).$	 
\end{itemize}
This proves Equation (\ref{goal}) and so it concludes the proof.
\end{proof}

\begin{obs}
	\label{future_need}
	For later goals (see the proof of Proposition \ref{explicit_constuction_limit}), we note that from the analysis below Equation (\ref{ergwtrgtrwgtwg}), it also follows easily that
	\begin{equation*}
	\begin{split}
	&\P\big(\text{pat}_b(\bm{T})=\pi\big)=\frac{2^{|\text{RLMax}(\pi)|+1}}{2^{2|\pi|}},\quad\text{for all}\quad\pi\in\text{Av}(231),\\
	&\P\big(\text{pat}_e(\bm{T})=\pi\big)=\frac{2^{|\text{LRMax}(\pi)|+1}}{2^{2|\pi|}},\quad\text{for all}\quad\pi\in\text{Av}(231).
	\end{split}
	\end{equation*}
\end{obs}

We can finally prove Proposition \ref{weakprop}.

\begin{proof}[Proof of Proposition \ref{weakprop}]
Summing up all the results and recalling that $\bm{T}=\bm{T}_{\delta}$, we obtain,
\begin{equation}
\begin{split}
\label{coccT}
\E\big[\coc(\pi,\bm{T}_{\delta})\big]&\stackrel{(\ref{step1})}{=}\delta^{-1}\P\big(\text{pat}_{\bm{J}}(\bm{T}_{\delta})=\pi\big)\stackrel{(\ref{probpat231})}{=}\delta^{-1}\big(P_{231}(\pi)+O(\delta)\big)\\&\;=\delta^{-1}\cdot P_{231}(\pi)+O(1).
\end{split}
\end{equation}

	Applying Lemma \ref{svantelemma} and using the bijection between $231$-avoiding permutations and binary trees, we conclude that for $n\to\infty,$
	\begin{equation*}
	\E\big[\coc(\pi,\bm{\sigma}^n)\big]\sim P_{231}(\pi)\cdot n,\quad\text{for all}\quad \pi\in\text{Av}(231). 
	\end{equation*}
	Dividing by $n$ yields Proposition \ref{weakprop}.
\end{proof}

\subsection{The quenched version of the Benjamini--Schramm convergence}
\label{mainres}
In this section we finally prove Theorem \ref{231theorem}. We are going to use a standard technique in probability theory called the \emph{Second moment method}. We have to study the second moment $\E\big[\widetilde{\coc}(\pi,\bm{\sigma}^n)^2\big].$
As before we start with a result regarding trees and then we will extend it to permutations.
\begin{prop}
	\label{figrfigpqofhfpuh}
	Using notation as before and setting $\bm{T}=\bm{T}_\delta,$ we have,  
	$$\E\big[\coc(\pi,\bm{T})^2\big]=\frac{P_{231}(\pi)^2}{2}\cdot\delta^{-3}+O(\delta^{-2}),\quad\text{for all}\quad\pi\in\emph{Av}(231).$$
\end{prop}
\begin{proof}
	We use again the decomposition given by Lemma \ref{ricors},
	$$\coc(\pi,\bm{T})=\coc(\pi,\bm{T}_L)+\coc(\pi,\bm{T}_R)+\mathds{1}_{\{\text{pat}_{\bm{J}}(\bm{T})=\pi\}}.$$ 
	Taking the expectation of the square of the previous equation we obtain,
	\begin{equation*}
	\begin{split}
	\E\big[\coc(\pi,\bm{T})^2\big]=&\E\big[\coc(\pi,\bm{T}_L)^2\big]+\E\big[\coc(\pi,\bm{T}_R)^2\big]+2\cdot\E\big[\coc(\pi,\bm{T}_L)\coc(\pi,\bm{T}_R)\big]\\
	&+2\cdot\E\big[\coc(\pi,\bm{T}_L)\mathds{1}_{\{\text{pat}_{\bm{J}}(\bm{T})=\pi\}}\big]+2\cdot\E\big[\coc(\pi,\bm{T}_R)\mathds{1}_{\{\text{pat}_{\bm{J}}(\bm{T})=\pi\}}\big]\\
	&+\P\big(\text{pat}_{\bm{J}}(\bm{T})=\pi\big).
	\end{split}
	\end{equation*}
	For the first two terms of the right-hand side of the above equation we use that $\bm{T}_L$ is an independent copy of $\bm{T}$ with probability $p$ and empty with probability $1-p$ (and similarly for $\bm{T}_R$). Therefore we can rewrite the above expression as
	\begin{equation}
	\label{ueiuwvfow}
	\begin{split}
	\E\big[\coc(\pi,\bm{T})^2\big]=\delta^{-1}\Big(&2\cdot\E\big[\coc(\pi,\bm{T}_L)\coc(\pi,\bm{T}_R)\big]+2\cdot\E\big[\coc(\pi,\bm{T}_L)\mathds{1}_{\{\text{pat}_{\bm{J}}(\bm{T})=\pi\}}\big]\\
	&+2\cdot\E\big[\coc(\pi,\bm{T}_R)\mathds{1}_{\{\text{pat}_{\bm{J}}(\bm{T})=\pi\}}\big]+\P\big(\text{pat}_{\bm{J}}(\bm{T})=\pi\big)\Big),
	\end{split}
	\end{equation}
	where we have also used that $1-2p=\delta.$
	We now analyze the four terms that appear on the right-hand side of the above equation. 
	
	First of all, using that $\bm{T}_L$ is independent of $\bm{T}_R,$ we obtain
	$$\E\big[\coc(\pi,\bm{T}_L)\coc(\pi,\bm{T}_R)\big]=p^2\cdot\E\big[\coc(\pi,\bm{T})\big]^2.$$ 
	Applying Equation (\ref{coccT}) and recalling that $p=\frac{1-\delta}{2},$ we conclude that
	\begin{equation}
	\label{sub1}
	\E\big[\coc(\pi,\bm{T}_L)\coc(\pi,\bm{T}_R)\big]=\frac{P_{231}(\pi)^2}{4}\cdot\delta^{-2}+O(\delta^{-1}).
	\end{equation}
	Obviously, using Proposition \ref{patprob231}, we have
	\begin{equation}
	\label{sub2}
	\P\big(\text{pat}_{\bm{J}}(\bm{T}_{\delta})=\pi\big)=P_{231}(\pi)+O(\delta).
	\end{equation}
	Therefore we have only to study the remaining two similar terms in Equation (\ref{ueiuwvfow}). Note that, if $\pi_L\neq\emptyset$ and $\pi_R\neq\emptyset,$ using the same arguments that we already used in a lot of different steps, we have
	\begin{equation*}
	\begin{split}
	\E\big[\coc(\pi,\bm{T}_L)\mathds{1}_{\{\text{pat}_{\bm{J}}(\bm{T})=\pi\}}\big]&=\E\big[\coc(\pi,\bm{T}_L)\mathds{1}_{\{\text{pat}_{e}(\bm{T}_L)=\pi_L\}}\mathds{1}_{\{\text{pat}_{b}(\bm{T}_R)=\pi_R\}}\big]\\
	&=\E\big[\coc(\pi,\bm{T}_L)\mathds{1}_{\{\text{pat}_{e}(\bm{T}_L)=\pi_L\}}\big]\E\big[\mathds{1}_{\{\text{pat}_{b}(\bm{T}_R)=\pi_R\}}\big]\\
	&=p^2\cdot\E\big[\coc(\pi,\bm{T})\mathds{1}_{\{\text{pat}_{e}(\bm{T})=\pi_L\}}\big]\cdot\P\big(\text{pat}_{b}(\bm{T})=\pi_R\big).
	\end{split}
	\end{equation*}
	Iterating the analysis of the second term (with similar techniques used in the proof of Lemma \ref{twoandthreecross}), we can conclude that  $\E\big[\coc(\pi,\bm{T}_L)\mathds{1}_{\{\text{pat}_{\bm{J}}(\bm{T})=\pi\}}\big]$ is a Laurent polynomial in $\delta$.
	Moreover, noting that $\coc(\pi,\bm{T}_L)\geq 0,$ we can obtain the following bound
	\begin{equation*}
	\begin{split}
	\E\big[\coc(\pi,\bm{T}_L)\mathds{1}_{\{\text{pat}_{\bm{J}}(\bm{T})=\pi\}}\big]\leq\E\big[\coc(\pi,\bm{T}_L)\big]=p\cdot\E\big[\coc(\pi,\bm{T})\big].
	\end{split}
	\end{equation*}
	By Equation (\ref{coccT}), $\E\big[\coc(\pi,\bm{T})\big]= O(\delta^{-1}).$ Therefore we can conclude that
	\begin{equation}
	\label{sub3}
	\E\big[\coc(\pi,\bm{T}_L)\mathds{1}_{\{\text{pat}_{\bm{J}}(\bm{T})=\pi\}}\big]=O(\delta^{-1}).
	\end{equation}
	Similarly, 
	\begin{equation}
	\label{sub4}
	\E\big[\coc(\pi,\bm{T}_R)\mathds{1}_{\{\text{pat}_{\bm{J}}(\bm{T})=\pi\}}\big]=O(\delta^{-1}).
	\end{equation}
	Substituting Equations (\ref{sub1}), (\ref{sub2}), (\ref{sub3}) and (\ref{sub4}) in Equation (\ref{ueiuwvfow}), we finally have
	\begin{equation*}
	\E\big[\coc(\pi,\bm{T})^2\big]=\frac{P_{231}(\pi)^2}{2}\cdot\delta^{-3}+O(\delta^{-2}). \qedhere
	\end{equation*}
\end{proof}

We can finally prove our main theorem.

\begin{proof}[Proof of Theorem \ref{231theorem}]
Applying again Lemma \ref{svantelemma} and using our bijection between $231$-avoiding permutations and binary trees, we conclude that for $n\to\infty,$
\begin{equation}
\label{final1}
\E\big[\coc(\pi,\bm{\sigma}^n)^2\big]\sim P_{231}(\pi)^2\cdot n^2, \quad\text{for all}\quad \pi\in\text{Av}(231).
\end{equation}
This, with Proposition \ref{weakprop}, implies that
\begin{equation}
\label{final2}
\text{Var}\big(\widetilde{\coc}(\pi,\bm{\sigma}^n)\big)\to 0, \quad\text{for all}\quad \pi\in\text{Av}(231).
\end{equation}
We can apply the Second moment method and deduce that
\begin{equation*}
\widetilde{\coc}(\pi,\bm{\sigma}^n)\stackrel{P}{\rightarrow}P_{231}(\pi),\quad\text{for all}\quad\pi\in\text{Av}(231).
\end{equation*}	
Indeed by Chebyschev's inequality, one has, for any fixed $\varepsilon>0,$
\begin{equation*}
\P\Big(\big|	\widetilde{\coc}(\pi,\bm{\sigma}^n)-\E\big[\widetilde{\coc}(\pi,\bm{\sigma}^n)\big]\big|\geq\varepsilon\Big)\leq\frac{1}{\varepsilon^2}\cdot\text{Var}\big(\widetilde{\coc}(\pi,\bm{\sigma}^n)\big),
\end{equation*}
and the right-hand side tends to zero.
\end{proof}

\subsection{The construction of the limiting object}
\label{explcon}

We now exhibit an explicit construction of the limiting object $\bm{\sigma}^{\infty}_{231}$ as a random order $\bm{\preccurlyeq}_{231}$ on $\mathbb{Z}.$ 

\begin{rem}
	The intuition behind the construction that we are going to present comes from the local limit for uniform binary trees (for more details, see \cite{stufler2016local}). When the size of a uniform binary tree tends to infinity, looking around a uniform distinguished vertex, we see an infinite upward spine. Each vertex in this spine is the left or the right child of the previous one with probability $1/2.$ Moreover, attached to this infinite spine there are some independent copies of binary Galton--Watson trees.
	Using this idea and the bijection between binary trees and $231$-avoiding permutations we are going to construct the limiting random total order.
	
	This intuition is not formally needed in the following since we present the construction of the limiting object using the permutations point of view. 
\end{rem}

We have to introduce some notation. We define two operations from $\Sr^{k}\times\mathcal{S}^{\ell}$ to $\Sr^{k+\ell+1},$ for $k>0$ and $\ell\geq0$ ($\mathcal{S}^{0}$ is the set containing the empty permutation): let $(\sigma,i)\in\Sr^{k}$ be a rooted permutation and $\pi\in\mathcal{S}^{\ell}$ be another (unrooted) permutation, we set
\begin{equation*}
\begin{split}
&(\sigma,i)*^R \pi\coloneqq\big(\sigma_1\dots\sigma_k(\ell+k+1)(\pi_1+k)\dots(\pi_{\ell}+k)\,,\,i\big),\\
&(\sigma,i)*^L \pi\coloneqq\big(\pi_1\dots\pi_\ell(\ell+k+1)(\sigma_1+\ell)\dots(\sigma_{k}+\ell)\,,\,\ell+i+1\big).
\end{split}
\end{equation*}

In words, from a graphical point of view, the diagram of $(\sigma,i)*^R \pi$ (resp. $(\sigma,i)*^L \pi$) is obtained starting from the diagram of the rooted permutation $(\sigma,i),$ adding on the top-right (resp. bottom-left) the diagram of $\pi$ and adding a new maximal element between the two diagrams. We give an example below.
\begin{exmp}
	Let $(\sigma,i)=(132,3)$ and $\pi=21$ then
	$$(\sigma,i)*^R \pi=
	\begin{array}{lcr}
	\includegraphics[scale=0.6]{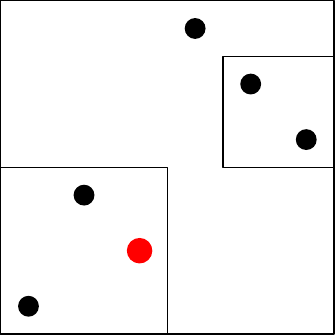}\\
	\end{array}=(132621,3)\quad\text{and}\quad(\sigma,i)*^L \pi=
	\begin{array}{lcr}
	\includegraphics[scale=0.6]{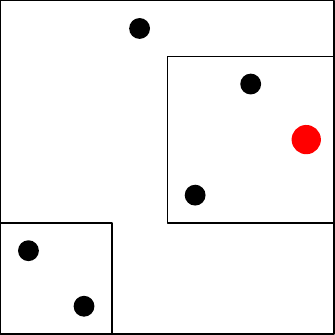}\\
	\end{array}=(216354,6).$$	
\end{exmp}

\begin{obs}
	\label{leftrightrem}
	Note that, given a rooted permutation $(\sigma,i),$ for all permutations $\pi$ (possibly empty), the rooted permutation $(\sigma,i)*^R \pi$ (resp. $(\sigma,i)*^L \pi$) contains at least one more element on the right (resp. left) of the root than $(\sigma,i)$.
\end{obs}

We now consider two families of random variables: let $\{\bm{S}^i\}_{i\geq 0}$ be i.i.d.\ random variables with values in $\text{Av}(231)\cup\{\emptyset\}$ and Boltzman distribution equal to 
$$\P(\bm{S}^i=\pi)=\frac{1}{2}\Big(\frac{1}{4}\Big)^{|\pi|},\quad\text{for all}\quad\pi\in\text{Av}(231)\quad\text{and}\quad\P(\bm{S}^i=\emptyset)=\frac{1}{2}.$$
This is a probability distribution since the generating function for the class $\text{Av}(231)$ is given by $A(z)=\frac{1-\sqrt{1-4z}}{2z}$ (see  \emph{e.g.,} Appendix \ref{combint}).  Moreover, let $\{\bm{X}_i\}_{i\geq 1}$ be i.i.d.\ random variables (also independent of $\{\bm{S}^i\}_{i\geq 0}$) with values in the set $\{R,L\}$ such that
$$\P(\bm{X}_i=\,R)=\frac{1}{2}=\P(\bm{X}_i=\,L).$$

\begin{obs}
	Recalling that $\bm{T}_\delta$ is the binary Galton--Watson tree with offspring distribution $\eta(\delta)$ defined by Equation (\ref{bintree}), we note that, for $\delta=0,$ $\sigma_{\bm{T}_{0}}\stackrel{(d)}{=}(\bm{S}^i|\bm{S}^i\neq\emptyset).$
\end{obs}

We are now ready to construct our random order $\bm{\preccurlyeq}_{231}$ on $\mathbb{Z}.$ Set $\tilde{\bm{S}}^0_{\bullet}\coloneqq(\tilde{\bm{S}}^0,\bm{m}),$ where $\tilde{\bm{S}}^0$ is the random variable $\bm{S}^0$ conditioned to be non empty, and $\bm{m}$ is the (random) index of the maximum of $\tilde{\bm{S}}^0.$ Then we set,
$$\bm{\sigma}^n_{\bullet}\coloneqq((\dots((\tilde{\bm{S}}^0_{\bullet}*^{\bm{X}_1}\bm{S}^1)*^{\bm{X}_2}\bm{S}^2)...\bm{S}^{n-1})*^{\bm{X}_n}\bm{S}^n),\quad\text{for all}\quad n\in\Z_{>0}.$$
Note that, for every fixed $h\in\Z_{>0},$ the sequence $\big\{r_h(\bm{\sigma}^n_{\bullet})\big\}_{n\in\Z_{>0}}$ is a.s.\ stationary (hence a.s.\ convergent) by Observation \ref{leftrightrem} since asymptotically almost surely at least $h$ random variables $\bm{X}_i$ are equal to $R$ and  at least $h$ are equal to $L.$ Let $\bm{\tau}_h$ be the limit of this sequence. The family $\big\{\bm{\tau}_h\big\}_{h\in\Z_{>0}}$ is a.s.\ consistent. Therefore, applying Proposition \ref{consistprop}, the family $\big\{\bm{\tau}_h\big\}_{h\in\Z_{>0}}$ determines a.s.\ a unique random total order $(\Z,\bm{\preccurlyeq}_{231})$ such that, for all $h\in\Z_{>0},$ $$r_h(\Z,\bm{\preccurlyeq}_{231})=\bm{\tau}_h=\lim_{n\to\infty}r_h(\bm{\sigma}^n_{\bullet}),\quad\text{a.s.}\;.$$
\begin{prop}
	\label{explicit_constuction_limit}
	Let $(\Z,\bm{\preccurlyeq}_{231})$ be the random total order defined above and $\bm{\sigma}^\infty_{231}$ be the limiting object defined in Corollary \ref{231corol}. Then
	$$(\Z,\bm{\preccurlyeq}_{231})\stackrel{(d)}{=}\bm{\sigma}^\infty_{231}.$$
\end{prop}

Before proving the proposition we need to introduce some more notation in order to split a permutation in several different parts. In what follows, we suggest to compare the various definitions with the example in Fig.~\ref{splitting_perm}. 

Given a pattern $\pi\in\text{Av}^{2h+1}(231)$ such that indmax$(\pi)> h+1,$ and  $$\text{LRMax}(\pi)\cap[h+2,2h+1]=\{M_1,\dots,M_k\},$$ with $M_1<\dots<M_k,$ then we can decompose it as
\begin{equation}
\label{decomp1}
(\pi,h+1)=\Bigg(\bigg(\dots\Big(\rho^0_{\bullet}(\pi)*^R\rho^1_R(\pi)\Big)*^R\dots\bigg)*^R\rho^{k-1}_R(\pi)\Bigg)*^R\rho^{b}_R(\pi),\quad \text{if}\quad h+1\in\text{LRMax}(\pi)
\end{equation}
\begin{equation}
\label{decomp2}
(\pi,h+1)=\Bigg(\bigg(\dots\Big(\big(\rho^0_{\bullet}(\pi)*^L\rho^e_L(\pi)\big)*^R\rho^1_R(\pi)\Big)*^R\dots\bigg)*^R\rho^{k-1}_R(\pi)\Bigg)*^R\rho^{b}_R(\pi),\quad \text{if}\quad h+1\notin\text{LRMax}(\pi) 
\end{equation} 
where the different terms in the decomposition are precisely described in what follows.

We set, 
\begin{equation*}
\rho^0(\pi)\coloneqq
\begin{cases}
\text{pat}_{[1,RM-1]}(\pi) &\quad\text{if }h+1\in\text{LRMax}(\pi),\\
\text{pat}_{[LM+1,RM-1]}(\pi) &\quad\text{if }h+1\notin\text{LRMax}(\pi). \\ 
\end{cases}
\end{equation*}
where $LM=\max\{0<j<h+1:j\in\text{LRMax}(\pi)\}$ and $RM=\min\{h+1<j<2h+2:j\in\text{LRMax}(\pi)\}.$ We also set
\begin{equation*}
\rho^0_{\bullet}(\pi)\coloneqq
\begin{cases}
(\rho^0(\pi),h+1) &\quad\text{if }h+1\in\text{LRMax}(\pi),\\
\big(\rho^0(\pi),h+1-LM\big) &\quad\text{if }h+1\notin\text{LRMax}(\pi), \\ 
\end{cases}
\end{equation*}
\emph{i.e.}, in both cases, we are rooting $\rho^0(\pi)$ at the index corresponding to the root of $(\pi,h+1).$ 
\begin{rem}
	Note that, in general, if $h+1\notin\text{LRMax}(\pi)$ then $\text{indmax}(\rho^0(\pi))\neq h+1-LM,$ \emph{i.e.,} the index of the maximum of $\rho^0(\pi)$ does not correspond to the root of $\rho^0_{\bullet}(\pi).$
\end{rem}

Then, analyzing the left side of the pattern $\pi,$ if $h+1\notin\text{LRMax}(\pi),$ we set $$\rho_L^e(\pi)\coloneqq\text{pat}_{[1,LM-1]}(\pi).$$

Finally, analyzing the right side of the pattern $\pi$ (regardless of $h+1$ being or not in $\text{LRMax}(\pi)$), if 
$$\text{LRMax}(\pi)\cap[h+2,2h+1]=\{M_1,\dots,M_k\},\quad\text{with}\quad M_1<\dots<M_k,$$
 we set 
 $$\rho_R^i(\pi)\coloneqq\text{pat}_{[M_i+1,M_{i+1}-1]}(\pi),\quad \text{for all}\quad 1\leq i\leq k-1,$$ 
(with the convention that, if $M_{i+1}-M_i=1$ then $\rho_R^i(\pi)=\emptyset$)  and  
$$\rho_R^b(\pi)\coloneqq\text{pat}_{[M_k+1,2h+1]}(\pi),$$
(with the convention that, if $M_{k}=2h+1$ then $\rho_R^b(\pi)=\emptyset$).

\begin{figure}[htbp]
	\begin{center}
		\includegraphics[scale=0.75]{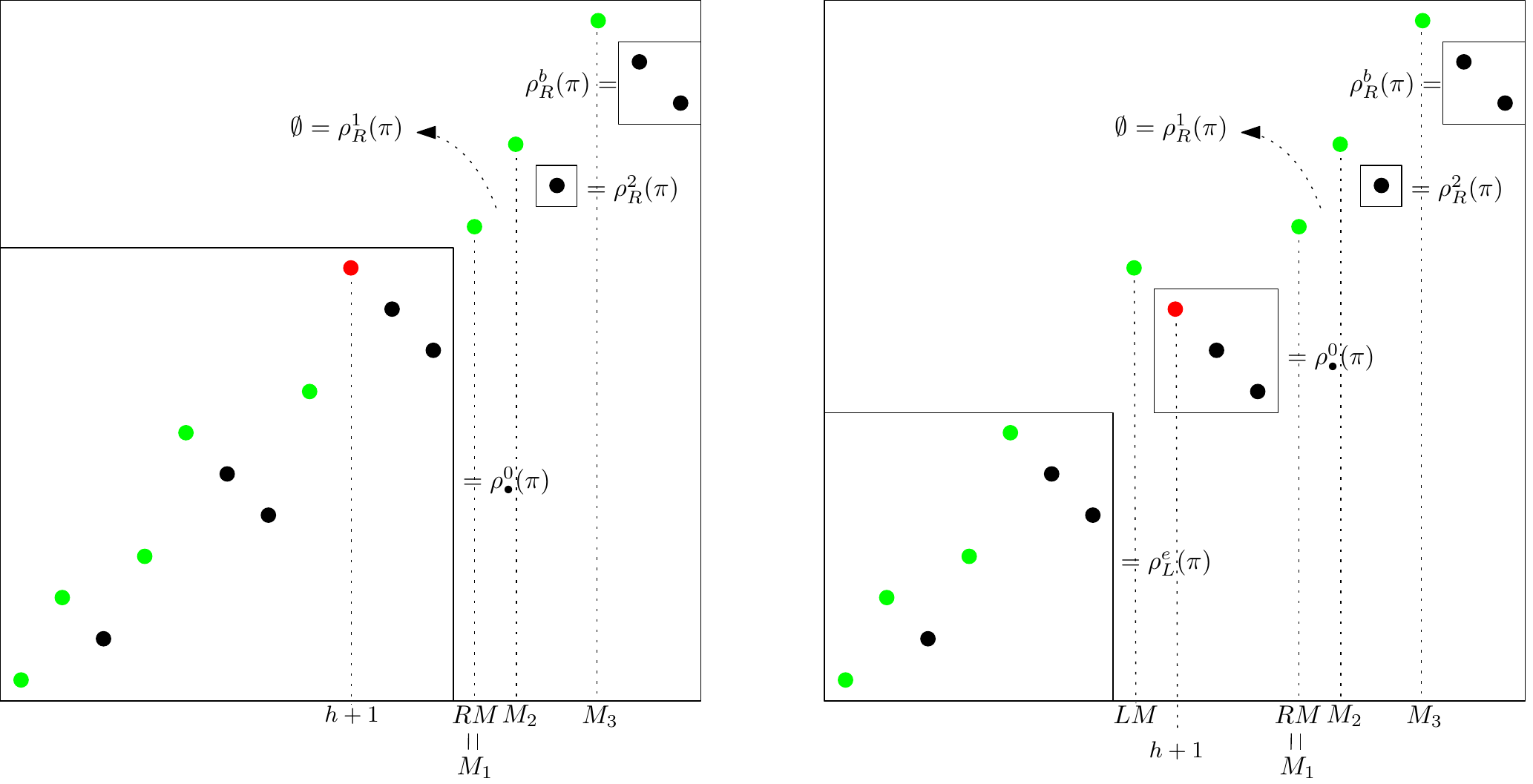}\\
		\caption{Two diagram representations of the splitting of a pattern $\pi\in\text{Av}^{2h+1}(231).$ On the left-hand side the case when $h+1\in\text{LRMax}(\pi),$ on the right-hand side the case when $h+1\notin\text{LRMax}(\pi).$ We highlight in red the value corresponding to the index $h+1.$ Moreover, we highlight in green the left-to-right maxima of the permutations $\pi.$}\label{splitting_perm}
	\end{center}
\end{figure}

\begin{proof}[Proof of Proposition \ref{explicit_constuction_limit}]
	Since by Observation \ref{separating class} the set of clopen balls 
	\begin{equation*}
	\mathcal{A}=\Big\{B\big((A,\preccurlyeq),2^{-h}\big):h\in\Z_{>0},(A,\preccurlyeq)\in\Sri\Big\}
	\end{equation*}
	is a separating class for the space $(\Sri,d),$ it is enough to prove that for all $h\in\Z_{>0},$ for all $\pi\in\text{Av}^{2h+1}(231),$
	\begin{equation}
	\label{whatwehavetoprove}
	\P\big(r_h(\Z,\bm{\preccurlyeq}_{231})=(\pi,h+1)\big)=P_{231}(\pi).
	\end{equation}
	 We fix a pattern $\pi\in\text{Av}^{2h+1}(231).$ We first suppose that:
	 \begin{itemize}
	 	\item indmax$(\pi)> h+1;$
	 	\item $h+1\notin\text{LRMax}(\pi)$ and $\text{LRMax}(\pi)\cap[h+2,2h+1]=\{M_1,\dots,M_k\},$ with $M_1<\dots<M_k;$
	 	\item $\rho^0_{\bullet}(\pi)$ is rooted at its maximum.
	 \end{itemize}    
	 Using the notation explained before, $\pi$ rewrites as 
	 \begin{equation*}
	 \label{decomp}
	 (\pi,h+1)=\Bigg(\bigg(\dots\Big(\big(\rho^0_{\bullet}(\pi)*^L\rho^e_L(\pi)\big)*^R\rho^1_R(\pi)\Big)*^R\dots\bigg)*^R\rho^{k-1}_R(\pi)\Bigg)*^R\rho^{b}_R(\pi).
	 \end{equation*}
	 Recalling that
	 $$\bm{\sigma}^n_{\bullet}\coloneqq((\dots((\tilde{\bm{S}}^0_{\bullet}*^{\bm{X}_1}\bm{S}^1)*^{\bm{X}_2}\bm{S}^2)...\bm{S}^{n-1})*^{\bm{X}_n}\bm{S}^n),\quad\text{for all}\quad n\in\Z_{>0},$$
	 and since by assumption $\rho^0_{\bullet}(\pi)$ is rooted at its maximum, by some easy combinatorial arguments, the event $\big\{r_h(\Z,\bm{\preccurlyeq}_{231})=(\pi,h+1)\big\}$ is equal to
	 \begin{equation*}
	 \big\{\tilde{\bm{S}}^0=\rho^0(\pi),\bm{X}_1=L,\text{pat}_e(\bm{S}^1)=\rho^e_L(\pi),\bm{S}^{\bm{\ell}_j}=\rho^j_R(\pi),1\leq j\leq k-1, \text{pat}_b(\bm{S}^{\bm{\ell}_k})=\rho^b_R(\pi)\big\},
	 \end{equation*}
	 where $\bm{\ell}_j$ denotes the (random) index of the $j$-th $\bm{X}_i$ equal to $R$ (note that $\bm{\ell}_j$ is almost surely well-defined).
	 Using the independence of $\{\bm{S}^i\}_{i\geq 0}$ and $\{\bm{X}_i\}_{i\geq 1}$ we have,
	 \begin{equation*}
	 \begin{split}
	 \P\big(r_h(\Z,\bm{\preccurlyeq}&_{231})=(\pi,h+1)\big)\\
	 =&\P\big(\tilde{\bm{S}}^0=\rho^0(\pi)\big)\cdot\P\big(\bm{X}_1=L\big)\cdot\P\big(\text{pat}_e(\bm{S}^1)=\rho^e_L(\pi)\big)\\
	 &\cdot\prod_{1\leq j\leq k-1}\P\big(\bm{S}^{\bm{\ell}_j}=\rho^j_R(\pi)\big)\cdot\P\big(\text{pat}_b(\bm{S}^{\bm{\ell}_k})=\rho^b_R(\pi)\big)\\
	 =&\Big(\frac{1}{4}\Big)^{|\rho^0(\pi)|}\cdot\frac{1}{2}\cdot\P\big(\text{pat}_e(\bm{S}^1)=\rho^e_L(\pi)\big)\cdot\prod_{1\leq j\leq k-1} \frac{1}{2}\Big(\frac{1}{4}\Big)^{|\rho^j_R(\pi)|}\cdot\P\big(\text{pat}_b(\bm{S}^{\bm{\ell}_k})=\rho^b_R(\pi)\big).
	 \end{split}
	 \end{equation*}
	 Using Observation \ref{future_need} and the fact that, for $\delta=0,$ $\sigma_{\bm{T}_{0}}\stackrel{(d)}{=}(\bm{S}^i|\bm{S}^i\neq\emptyset),$ we have that, 
	 \begin{equation*}
	 \begin{split}
	 &\P\big(\text{pat}_b(\bm{S}^i)=\rho\big)=\frac{2^{|\text{RLMax}(\rho)|}}{2^{2|\rho|}},\quad\text{for all}\quad\rho\in\text{Av}(231),\\
	 &\P\big(\text{pat}_e(\bm{S}^i)=\rho\big)=\frac{2^{|\text{LRMax}(\rho)|}}{2^{2|\rho|}},\quad\text{for all}\quad\rho\in\text{Av}(231).
	 \end{split}
	 \end{equation*}
	 Therefore, we can conclude that
	 \begin{equation*}
	 \begin{split}
	 \P\big(r_h(\Z,\bm{\preccurlyeq}_{231})=(\pi,h+1)\big)
	 =&\Big(\frac{1}{4}\Big)^{|\rho^0(\pi)|}\cdot\frac{1}{2}\cdot\frac{2^{|\text{LRMax}(\rho^e_L(\pi))|}}{2^{2|\rho^e_L(\pi)|}}\cdot\prod_{1\leq j\leq k-1} \frac{1}{2}\Big(\frac{1}{4}\Big)^{|\rho^j_R(\pi)|}\cdot\frac{2^{|\text{RLMax}(\rho^b_R(\pi))|}}{2^{2|\rho^b_R(\pi)|}}\\
	 =& \frac{2^{|\text{LRMax}(\rho^e_L(\pi))|+|\text{RLMax}(\rho^b_R(\pi))|}}{2^{2(|\pi|-(k+1))}2^{k}}=\frac{2^{|\text{LRMax}(\pi)|+|\text{RLMax}(\pi)|}}{2^{2|\pi|}}=P_{231}(\pi).
	 \end{split}
	 \end{equation*}
	 This concludes the proof of Equation (\ref{whatwehavetoprove}) when indmax$(\pi)> h+1,$ $h+1\notin\text{LRMax}(\pi)$ and $\rho^0_{\bullet}(\pi)$ is rooted at its maximum.
	 For the other cases we can proceed as follows:
	 \begin{itemize}
	 	\item if indmax$(\pi)> h+1,$ $h+1\notin\text{LRMax}(\pi)$ but $\rho^0_{\bullet}(\pi)$ is not rooted at its maximum, we have to repeat the decomposition explained in Equations (\ref{decomp1}) and (\ref{decomp2}) for the permutation $\rho^0(\pi)$. We iterate this procedure until $\rho^0_{\bullet}\big(\rho^0(\dots(\rho^0(\pi))\dots)\big)$ is rooted at its maximum (this procedure finishes since at each step $|\rho^0\big(\rho^0(\pi)\big)|<|\rho^0(\pi)|$). In this way we obtain a new decomposition of $(\pi,h+1).$ Finally, with similar computations as before we can verify Equation (\ref{whatwehavetoprove});
	 	
	 	\item if indmax$(\pi)> h+1,$ $h+1\in\text{LRMax}(\pi)$ then the proof of Equation (\ref{whatwehavetoprove}) is similar to the one in the first case substituting the event $\big\{\tilde{\bm{S}}^0=\rho^0(\pi),\bm{X}_1=L,\text{pat}_e(\bm{S}^1)=\rho^e_L(\pi)\big\}$ with the event $\big\{\text{pat}_e(\tilde{\bm{S}}^0_{\bullet})=\rho^0_{\bullet}(\pi)\big\}$ and then arguing exactly in the same way as before;
	 	
	 	\item if indmax$(\pi)<h+1,$ then Equation (\ref{whatwehavetoprove}) follows from the previous cases by a symmetry argument;
	 	
	 	\item finally, if $\text{indmax}(\pi)=h+1,$  $$\P\big(r_h(\Z,\bm{\preccurlyeq}_{231})=(\pi,h+1)\big)=\P\big(r_h(\tilde{\bm{S}}^0)=(\pi,h+1)\big)=P_{231}(\pi),$$
	 	where in the last equality we used Proposition \ref{patprob231} and the fact that, for $\delta=0,$ $\sigma_{\bm{T}_{0}}\stackrel{(d)}{=}\tilde{\bm{S}}^0.$
	 \end{itemize}
 The analysis of all these cases concludes the proof.	 
\end{proof}

\section{The local limit for uniform random 321-avoiding permutations}
\label{321}
The goal of this section is to prove that uniform random 321-avoiding permutations converge in the quenched Benjamini--Schramm sense. 
We start by introducing, for all $k>0,$ the following probability distribution on $\text{Av}^k(321)$,
\begin{equation}
\label{321distr}
P_{321}(\pi)\coloneqq
\begin{cases}
\frac{|\pi|+1}{2^{|\pi|}} &\quad\text{if }\pi=12\dots|\pi|,\\
\frac{1}{2^{|\pi|}} &\quad\text{if }\coc(21,\pi^{-1})=1, \\
0 &\quad\text{otherwise.} \\ 
\end{cases}
\end{equation}

\begin{rem}
	\label{rem_sep_line}
	We note that the second condition $\coc(21,\pi^{-1})=1$ (\emph{i.e.,} $\pi^{-1}$ has exactly one descent, or equivalently, $\pi$ has exactly one \emph{inverse descent}) is equivalent to saying that we can split the permutation $\pi$ in exactly two increasing subsequences such that all the values of one subsequence are smaller than the values of the other one. Formally, if $\pi\in\text{Av}^k(321)$ and $\coc(21,\pi^{-1})=1$ there exists a unique bipartition of $[m]$ as $[m]=L(\pi)\sqcup U(\pi)$ such that for all $(i,j)\in L(\pi)\times U(\pi)$ then $\pi_i<\pi_j$ and $\text{pat}_{L(\pi)}(\pi)$ and $\text{pat}_{U(\pi)}(\pi)$ are increasing.
	
	Given the diagram of a permutation in $\text{Av}^k(321)$ with exactly one inverse descent, we call \emph{separating line} the horizontal line in the diagram between these two increasing subsequences. An example is given in Fig.~\ref{321ex} in the case of the permutation $\pi=14526738,$ where we highlight in orange the separating line between the two increasing subsequences.   
	\begin{figure}[htbp]
		\begin{center}
			\includegraphics[scale=.70]{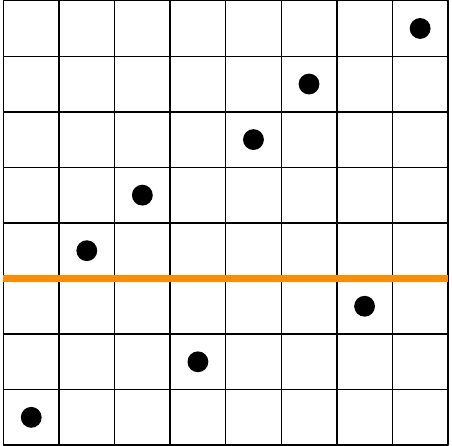}\\
			\caption{Diagram of the permutation $14526738$ with the separating line highlighted in orange.}\label{321ex}
		\end{center}
	\end{figure}
\end{rem}

\begin{rem}
	That Equation (\ref{321distr}) defines a probability distribution on $\text{Av}^k(321)$ is a consequence of Theorem \ref{321theorem} below: in Equation (\ref{shsshgea}), summing over all $\pi\in\text{Av}^k(321)$ the left-hand side, we obtain $\frac{n-k+1}{n}$ which trivially tends to 1. It is however easy to give a direct explanation. Indeed the number of permutations in $\text{Av}^k(321)$ with exactly one inverse descent and separating line at height $\ell\in\{1,\dots,m-1\}$ is $\binom{k}{\ell}-1$ (where the binomial coefficient $\binom{k}{\ell}$ counts the possible choices of points that are below the separating line and the term $-1$ stands for the identity permutation which has no inverse descent). Therefore the number of permutations in $\text{Av}^k(321)$ with exactly one inverse descent is
	$$\sum_{\ell=1}^{k-1}\Bigg(\binom{k}{\ell}-1\Bigg)=2^k-(k+1),$$ 
	which entails $\sum_{\pi\in\text{Av}^k(321)}P_{321}(\pi)=1.$
\end{rem}

We now recall, similarly to the previous section, the following trivial fact.

\begin{obs}
	If $\pi\notin\text{Av}(321)$ and $\sigma\in\text{Av}(321)$ then obviously $\coc(\pi,\sigma)=0.$ Therefore, in what follow, we simply analyze the consecutive occurrences densities $\widetilde{\coc}(\pi,\sigma)$ for $\pi\in\text{Av}(321).$
\end{obs}

\begin{thm}
	\label{321theorem}	
	Let $\bm{\sigma}^n$ be a uniform random permutation in $\emph{Av}^n(321)$ for all $n\in\Z_{>0}.$ We have the following convergence in probability,
	\begin{equation}
		\label{shsshgea}
		\widetilde{\coc}(\pi,\bm{\sigma}^n)\stackrel{P}{\rightarrow}P_{321}(\pi),\quad\text{for all}\quad\pi\in\emph{Av}(321).
	\end{equation}
\end{thm}

Since the limiting frequencies $\big(P_{321}(\pi)\big)_{\pi\in\text{Av}(321)}$ are deterministic, using Corollary \ref{detstrongbsconditions}, we have the following.
\begin{cor}
	\label{321corol}
	Let $\bm{\sigma}^n$ be a uniform random permutation in $\emph{Av}^n(321)$ for all $n\in\Z_{>0}.$ There exists a random infinite rooted permutation $\bm{\sigma}^\infty_{321}$ such that   $$\bm{\sigma}^n\stackrel{qBS}{\longrightarrow}\mathcal{L}(\bm{\sigma}^\infty_{321}).$$
\end{cor}	

\begin{rem}
	We note that, as in the case of Corollary \ref{231corol}, Corollary \ref{321corol} proves the existence of a random infinite rooted permutation $\bm{\sigma}^\infty_{321}$ without constructing explicitly this limiting object. Nevertheless, also in this case, we are able to provide an explicit construction of $\bm{\sigma}^\infty_{321}$ (see Section \ref{explcon2}). 
\end{rem}

The proof of Theorem \ref{321theorem} is structured as follows:
\begin{itemize}
	\item in Section \ref{premres_321} we introduce a well-known bijection between rooted ordered trees and 321-avoiding permutations. Moreover, we explain why it is enough to study a specific family of Galton--Watson trees;
	\item in Section \ref{local_limit_for_trees} we present two results, due to Stufler \cite{stufler2016local} and Janson \cite{janson2012simply}, that characterize the local limit for Galton--Watson trees pointed at a uniform vertex;
	\item in Section \ref{cont_trees} we introduce the notion of contour of a tree and we state some technical results about the characterization of the shape of some specific contours;
	\item in Section \ref{mainres_321} we state two propositions (whose the proofs are postponed to Sections \ref{prop1proof} and \ref{prop2proof}) that are the key results for the proof of Theorem \ref{321theorem}. The proofs of the two propositions are based on the results presented in the previous Sections \ref{local_limit_for_trees} and \ref{cont_trees};
	\item finally, in Section \ref{explcon2} we give an explicit construction of the limiting object $\bm{\sigma}^\infty_{321}.$ 
\end{itemize}

\subsection{From uniform 321-avoiding permutations to Galton--Watson trees}
\label{premres_321}
It is well known that 321-avoiding permutations can be broken into two increasing subsequences, one weakly above the diagonal and one strictly below the diagonal (see \emph{e.g.,} \cite[Section 4.2]{bona2016combinatorics}). Therefore, defining for every permutation $\sigma\in\text{Av}^n(321)$ 
\begin{equation*}
\begin{split}
&E^+=E^+(\sigma)\coloneqq\big\{i\in[n]|\sigma_i\geq i\big\},\\
&E^-=E^-(\sigma)\coloneqq\big\{i\in[n]|\sigma_i< i\big\},
\end{split}
\end{equation*}
we have that $\text{pat}_{E^+}(\sigma)$ and $\text{pat}_{E^-}(\sigma)$ are the increasing patterns.
Note that by convention the indices of fixed points are elements of $E^+.$

We now describe a slightly different, but equivalent (see Remark \ref{rem_equiv}), version of a well-known bijection between 321-avoiding permutations (of size $n$) and rooted ordered trees (with $n+1$ vertices) -- see \emph{e.g.,} \cite{hoffman2016fixed}.

We recall that $\mathbb{T}^n$ denotes the set of rooted ordered trees with $n$ vertices (see Section \ref{tree_notation} for notation about trees and random trees). Suppose $T\in\mathbb{T}^{n+1}$ has $k$ leaves and let $\ell(T)=(\ell_1,\dots,\ell_k)$ be the list of leaves of $T,$ listed in order of appearance in the pre-order traversal of $T,$ \emph{i.e.,} from left to right. Moreover let $s_i$ and $q_i$ be the labels given to the leaf $\ell_i$ respectively by the pre-order labeling of $T$ starting from zero and the post-order labeling of $T$ starting from 1. We set $S(T)\coloneqq\{s_i\}_{1\leq i\leq k}$ and $Q(T)\coloneqq\{q_i\}_{1\leq i\leq k}.$  An example is given in the left-hand side of Fig.~\ref{321bij_1}.

We start by constructing the permutation $\sigma_{T}$ associated to the tree $T.$ We define $\sigma_T$ pointwise on $Q(T)=\{q_i\}_{1\leq i\leq k}$ as
$$\sigma_T(q_i)=s_i,\quad\text{for all}\quad1\leq i\leq k,$$ 
and then we extend $\sigma_T$ on $[n]\setminus Q(T)$ by assigning values of $[n]\setminus S(T)$ in increasing order. Clearly $\sigma_{T}$ is a $321$-avoiding permutation. An example is given in Fig.~\ref{321bij_1}. 

We highlight (without proof, for more details see \cite{hoffman2016fixed}) that this construction implies $E^+(\sigma_T)=Q(T)$ and $\sigma(E^+(\sigma_T))=S(T).$ It can also be proved that this construction is invertible: for a permutation $\sigma\in\text{Av}^n(321)$ there is a unique rooted ordered tree $T_{\sigma}$ with $n+1$ vertices such that $Q(T_{\sigma})=E^+(\sigma)$ and $S(T_{\sigma})=\sigma(E^+(\sigma)).$

\begin{figure}[htbp]
	\begin{center}
		\includegraphics[scale=0.80]{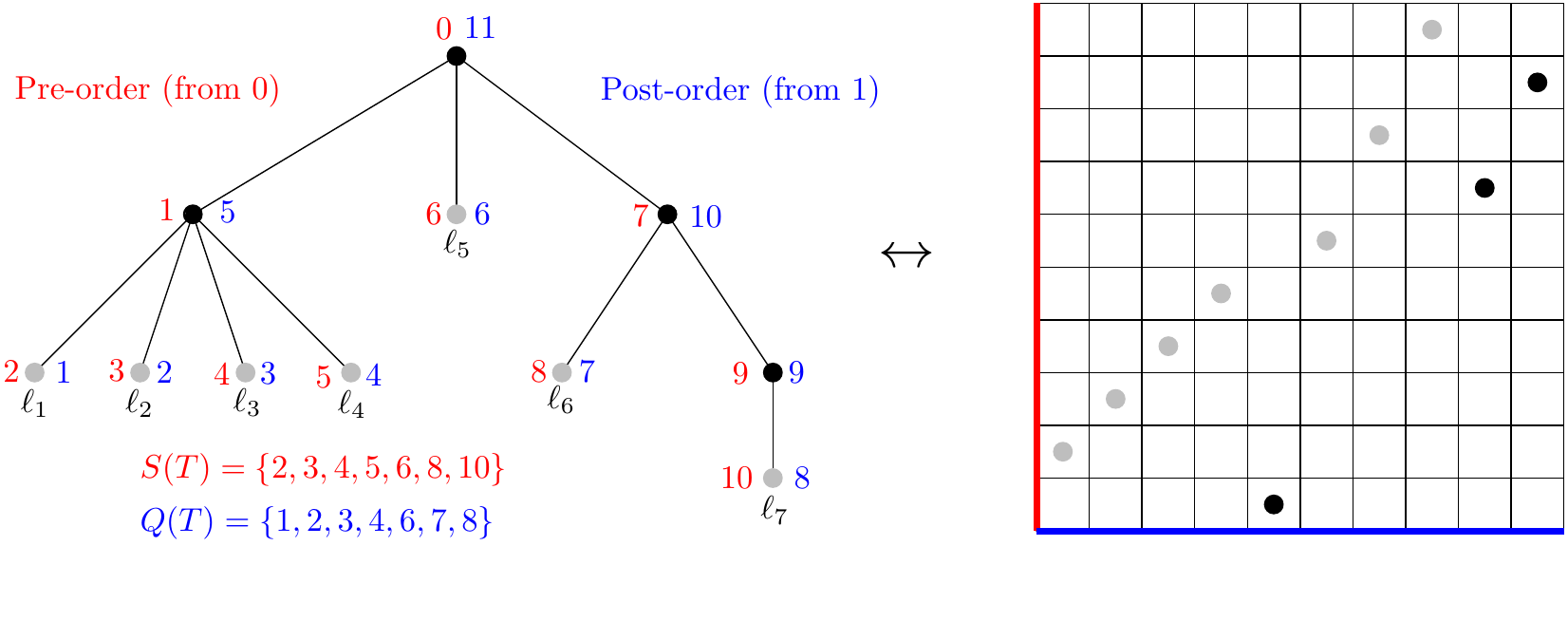}\\
		\caption{On the left of the picture we see a tree $T$ with 11 vertices and 7 leaves: we record in red on the left of each vertex the pre-order label and in blue on the right of each vertex the post-order label. On the right of the picture we see the associated 321-avoiding permutation of size 10: in the diagram, for every leaf in the tree, we draw a gray dot using the post-order label for the $x$-coordinate and the pre-order label for the $y$-coordinate. Then we complete the diagram with the unique increasing subsequence fitting in the empty rows and columns (shown with black dots on the picture).}\label{321bij_1}
	\end{center}
\end{figure}

\begin{obs}
	\label{obs_treeperm}
	Let $\bm{\eta}\sim Geom(1/2)$ \emph{i.e.,} $\eta_k\coloneqq\P(\bm{\eta}=k)=2^{-k-1},$ for all $k\geq 0.$ We recall that a uniform random rooted ordered tree with $n$ vertices $\bm{T}_n$ has the same distribution as a Galton--Watson tree $\bm{T}^{\eta}$ with offspring distribution $\bm{\eta}$ conditioned on having $n$ vertices (see \cite[Example 9.1]{janson2012simply}), namely
	\begin{equation}
	\bm{T}_n\stackrel{(d)}{=}\big(\bm{T}^\eta\;\big|\;|\bm{T}^\eta|=n\big).
	\end{equation}
	Therefore, thanks to the bijection previously introduced, in order to prove Theorem \ref{321theorem}, we can study a conditioned Galton--Watson tree $\bm{T}^\eta_n\coloneqq\big(\bm{T}^\eta\;\big|\;|\bm{T}^\eta|=n\big)$ instead of a uniform 321-avoiding permutation $\bm{\sigma}^n.$ We also point out that $E^+(\bm{\sigma}^n)=Q(\bm{T}^\eta_{n+1})$ a.s.
\end{obs}

\begin{rem}
	\label{rem_equiv}
	At first sight, it seems that the construction presented here is different from that of \cite{hoffman2016fixed}, we argue here that they are in fact the same.
	In \cite{hoffman2016fixed} the authors consider the same numbers $\{s_i\}_{1\leq i\leq k}$ but different numbers $\{q'_i\}_{1\leq i\leq k}$ (denoted by $\{q_i\}_{1\leq i\leq k}$ in \cite{hoffman2016fixed}) which are defined as follow
	\begin{equation*}
	q'_i=s_i-h(\ell_i)+1,\quad\text{for all}\quad 1\leq i\leq k.
	\end{equation*}
	Note that $q'_1=q_1=1.$ Analyzing the difference between two consecutive values in $\{q'_i\}_{1\leq i\leq k},$ we have
	\begin{equation*}
	\begin{split}
	q'_{i+1}-q'_i&=(s_{i+1}-h(\ell_{i+1})+1)-(s_i-h(\ell_i)+1)\\
	&=\big(s_{i+1}-s_i\big)-\big(h(\ell_{i+1})-h(\ell_{i})\big)\\
	&=\big(d(\ell_{i+1},c(\ell_{i+1},\ell_{i}))\big)-\big(d(\ell_{i+1},c(\ell_{i+1},\ell_{i}))-d(\ell_{i},c(\ell_{i+1},\ell_{i}))\big)\\
	&=d(\ell_{i},c(\ell_{i+1},\ell_{i}))=q_{i+1}-q_i,
	\end{split}
	\end{equation*}
	where we recall that $c(a,b)$ is the first common ancestor of $a$ and $b.$
	We conclude that $q_i=q'_i,$ for all $1\leq i\leq k,$ so our notation is consistent with that of \cite{hoffman2016fixed}.
\end{rem}

\subsection{A local limit result for Galton--Watson trees pointed at a uniform vertex}
\label{local_limit_for_trees}
In this section we present a combination of two results due to Stufler \cite{stufler2016local} and Janson \cite{janson2012simply} that characterize the local limit for Galton--Watson trees $\bm{T}^{\eta}_n$ pointed at a uniform vertex.

We begin by recalling the recent construction, due to Stufler, of the local limit tree $\bm{T}^*.$ We provide the construction in the special case where $\eta\sim Geom(1/2)$ (\emph{cf.} Fig.~\ref{Stufler_tree} below).

For each $i\geq 1,$ we let a vertex $u_i$ receive offspring according to an independent copy of $\hat{\bm{\eta}}$, that is a random variable with size-biased distribution $\hat{\eta}_k=k\cdot\eta_k.$ The vertex $u_{i-1}$ gets identified with a child of $u_i$ chosen uniformly at random. All other children of $u_{i}$ become the root of independent copies of the Galton--Watson tree $\bm{T}^{\eta}.$ Finally $u_0$ becomes the root of another independent copy of the Galton--Watson tree $\bm{T}^{\eta}.$ We always think of $\bm{T}^*$ as a tree with distinguished vertex $u_0$.

\begin{figure}[htbp]
	\begin{center}
		\includegraphics[scale=0.85]{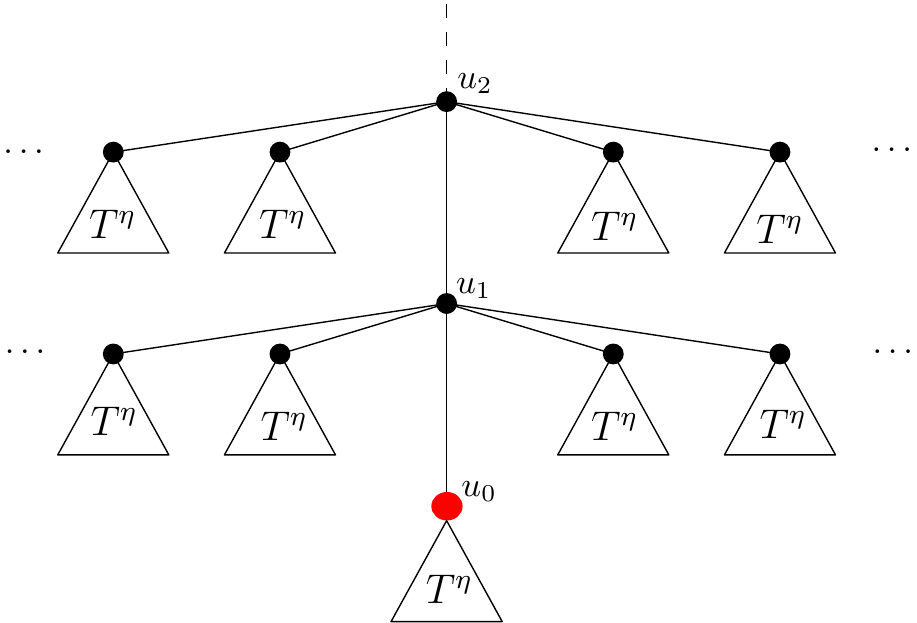}\\
		\caption{A scheme of the ``local limit tree" $\bm{T}^*.$}\label{Stufler_tree}
	\end{center}
\end{figure}

We point out that $\bm{T}^*$ is not a tree in the classical Neveu's formalism. We have to imagine that ``the root" of the tree is at the end of the infinite backward spine composed by the vertices $u_i,$ $i\geq0$. Note also that $\bm{T}^*$ has a.s.\ infinitely many vertices on the left and on the right of the infinite spine. For a formal construction of the set of these trees with infinite spine see \cite[Section 4.1, p.6]{stufler2016local}. In what follows we will refer to this type of trees as \emph{trees with an infinite spine.}

We now state the following easy reformulation of a result due to Janson \cite{janson2012simply}. First we recall that a rooted ordered tree $T$ together with a distinguished vertex $v_0$ is called a \emph{pointed rooted ordered tree} and it is denoted by $T^\bullet=(T,v_0).$ Given $T^\bullet=(T,v_0)$ a pointed rooted ordered tree, we denote with $f_\ell^{\bullet}(T^\bullet)=f_\ell^{\bullet}(T,v_0)$ the pointed fringe subtree rooted at the $\ell$-th ancestor of the distinguished vertex $v_0\in T$ (if this is not well-defined because $v_0$ has height smaller than $\ell$, we set $f_\ell^{\bullet}(T,v_0)=\diamond,$ for some fixed placeholder value $\diamond$). Similarly we denote with $f^\bullet_\ell(\bm{T}^*)$ the fringe subtree of $\bm{T}^*$ rooted at $u_\ell$ with distinguished vertex $u_0.$ Despite $\bm{T}^*$ is not a tree in Neveu's formalism, $f^\bullet_\ell(\bm{T}^*)$ is almost surely a pointed rooted ordered tree.

Finally, we denote with $\mathbb{T}^\bullet_h$ the set of pointed rooted ordered trees whose distinguished vertex has height $h\geq 0$.
\begin{prop}
	\label{keycoroll}
	Let $\bm{T}^\eta_n$ be as before and $\bm{v}_0$ be a uniform vertex of $\bm{T}^\eta_n$ chosen independently of $\bm{T}^\eta_n$. Let also $h\geq0$ and $\mathcal{T}\subseteq \mathbb{T}^\bullet_h.$ Then
	$$\P^{\bm{v}_0}\big(f^\bullet_h(\bm{T}_n^\eta,\bm{v}_0)\in\mathcal{T})\stackrel{P}{\to}\P\big(f^\bullet_h(\bm{T}^*)\in\mathcal{T}\big).$$
\end{prop}

\begin{proof}
	Note that $\P^{\bm{v}_0}\big(f^\bullet_h(\bm{T}_n^\eta,\bm{v}_0)=T^\bullet\big)$, for some $T^\bullet=(T,v_0)\in\mathbb{T}^\bullet_h,$ can be rewritten as 
	\begin{equation*}
	\P^{\bm{v}_0}\big(f^\bullet_h(\bm{T}_n^\eta,\bm{v}_0)=T^\bullet\big)=\P\big(f^\bullet_h(\bm{T}_n^\eta,\bm{v}_0)=T^\bullet\big|\bm{T}^\eta_n\big)\stackrel{(d)}{=}\frac{N_{T^\bullet}(\bm{T}^\eta_n)}{n},
	\end{equation*}
	where $N_{T^\bullet}(\bm{T}^\eta_n)\coloneqq \big|\big\{v\in \bm{T}^\eta_n:f_h^{\bullet}(\bm{T}^\eta_n,v)=T^\bullet\big\}\big|.$
	If we prove that, for all $h\geq 0,$ all $T^\bullet\in\mathbb{T}^\bullet_h,$
	\begin{equation}
	\label{goalofproof}
	\frac{N_{T^\bullet}(\bm{T}^\eta_n)}{n}\stackrel{P}{\to}\P\big(f^\bullet_h(\bm{T}^*)=T^\bullet\big),
	\end{equation}
	then the statement follows from the fact that $\mathcal{T}\subset \mathbb{T}^\bullet_h$ is countable.
	
	Note that
	$$N_{T^\bullet}(\bm{T}^\eta_n)=\big|\big\{u\in \bm{T}^\eta_n:f(\bm{T}^\eta_n,u)=T\big\}\big|.$$
	Indeed, if $f_h^{\bullet}(\bm{T}^\eta_n,v)=T^\bullet$ then $f(\bm{T}^\eta_n,a^h(v))=T.$ Conversely, given $u$ such that $f(\bm{T}^\eta_n,u)=T,$ there exists a unique choice for $v$ such that $f_h^{\bullet}(\bm{T}^\eta_n,v)=T^\bullet:$ using the Neveu's formalism, $v$ must be $v=uw,$ where $w$ is the pointed vertex of $T^\bullet.$
	Therefore by \cite[Theorem 7.12 (ii)]{janson2012simply}, we have,
	$$\frac{N_{T^\bullet}(\bm{T}^\eta_n)}{n}\stackrel{P}{\to}\P\big(\bm{T}^\eta=T\big).$$
	So, proving that $\P\big(\bm{T}^\eta=T\big)=\P\big(f^\bullet_h(\bm{T}^*)=T^\bullet\big)$ will be enough to conclude the proof of Equation \eqref{goalofproof}. This follows easily, since
	$$\P\big(\bm{T}^\eta=T\big)=\prod_{u\in T}\eta_{d^+_u(T)},$$
	and denoting with $\mathcal{P}$ the path from the root to the parent of the distinguished vertex in $T^\bullet,$
	\begin{equation*}
	\P\big(f^\bullet_h(\bm{T}^*)=T^\bullet\big)=\prod_{u\in\mathcal{P}}\Big(\hat{\eta}_{d^+_u(T^\bullet)}\cdot\frac{1}{d^+_u(T^\bullet)}\Big)\prod_{u\in T^{\bullet},u\notin\mathcal{P}}\eta_{d^+_u(T^\bullet)}=\prod_{u\in T^\bullet}\eta_{d^+_u(T^\bullet)}.
	\qedhere
	\end{equation*}	
\end{proof}

\begin{rem}
	\label{otherpossiblestrat1}
	Proposition \ref{keycoroll} implies the local convergence of the rooted tree $(\bm{T}_n^\eta,\bm{v}_0)$ in the quenched sense, \emph{i.e.,} the convergence of the conditional random variable $\big((\bm{T}_n^\eta,\bm{v}_0)|\bm{T}_n^\eta\big).$
\end{rem}

\subsection{Contour function for trees}
\label{cont_trees}
In this section we recall the notion of contour function for trees and we generalize this definition for the local limit tree $\bm{T}^*.$
We start with the following.

\begin{defn}
	A \emph{Dyck path} of length $2n$ is a lattice path from $(0,0)$ to $(2n,0)$, consisting of $n$ up-steps (along the vector $(1,1)$) and $n$ down-steps (along the vector $(1,-1)$), such that the path never goes below the $x$-axis. We encode a Dyck path by a word $w_1\cdots w_{2n}$ containing $n$ letters $U$ and $n$ letters $D$. The condition “the path never goes below the $x$-axis” is equivalent to “every prefix $w_1\cdots w_{k}$ contains at least as many $U$’s as $D$’s”. 
\end{defn}

\begin{defn}
	Given a rooted ordered tree $T$ with $n+1$ vertices, we define the \emph{contour} $C(T)$ of $T$ as the Dyck path of length $2n$ obtained by turning around $T$ from left to right (as for the depth-first traversal), starting from the root, and recording, for each traversed edge, either an up-step if we are visiting the edge on the left from top to bottom or a down-step if we are visiting the edge on the right from bottom to top. 
\end{defn}

An example of contour of a tree is given in Fig.~\ref{cont_tree}. 

\begin{figure}[htbp]
	\begin{center}
		\includegraphics[scale=1]{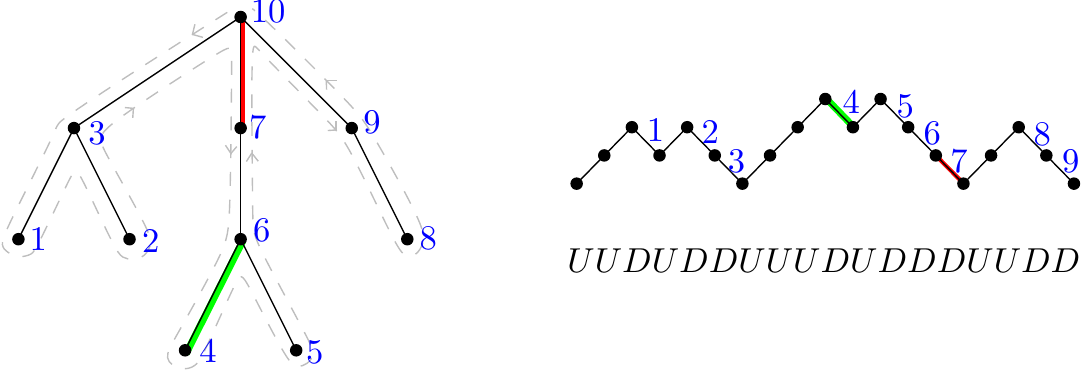}\\
		\caption{A tree with its contour. We highlighted in red and green the right side of two edges in the tree with the corresponding two down-steps in the contour. Moreover, we indicated post-order labels in the tree and enumerated down-steps in the path for future reference in the proof of Lemma \ref{lem_con_lab}.}\label{cont_tree}
	\end{center}
\end{figure}

\begin{rem}
	\label{postorderfprtree}
	From now until the end of Section \ref{321}, when we refer to labels, we always assume that we are speaking about labels given by the post-order labeling starting from 1. Moreover, given a rooted ordered tree and a label $j,$ we use the abuse of notation $(T,j)$ for denoting the pointed tree with distinguished vertex labeled by $j.$ Similarly, we will simply say ``the vertex $j"$ instead of ``the vertex labeled by $j".$
\end{rem}
We now make a connection between contours of trees and the labels of their leaves. We first introduce some notation for words over the alphabet $\{U,D\}.$ We denotes by $U^+$ (resp. $D^+$) the set of finite words of length at least one composed only by letters $U$ (resp. $D$). Given two words $w,u$, we write $wu$ for the concatenation of the two words. More generally, given a family $\{w_i\}_{1\leq i\leq n}$ of words, we write $\prod_{i=1}^n w_i$ for their consecutive concatenation.

\begin{lem}
	\label{lem_con_lab}
	Let $A=\{x_1,x_2,\dots,x_m\}\subset\mathbb{Z}_{>0}$ (with $x_1=1$) be a set of labels. Then
	$$\big\{T\in\mathbb{T}:Q(T)=A\big\}=\bigg\{T\in\mathbb{T}:C(T)\in \Big(\prod_{j=1}^{m-1}U^+D^{x_{j+1}-x_{j}}\Big)U^+D^+\bigg\}.$$
	
	Moreover, the shape of the contour of the tree between two consecutive leaves $\ell_j,\ell_{j+1}$ respectively labeled by $x_j,x_{j+1}$ is of the form $U^+D^{x_{j+1}-x_{j}}.$
\end{lem}

\begin{proof}
	Note that given a tree $T\in\mathbb{T}$ (\emph{cf.} Fig.~\ref{cont_tree}) and its contour $C(T)$ then the leaves of $T$ correspond to letters $D$ in $C(T)$ preceded by a $U$. On the other hand, internal vertices (different from the root) correspond (for the first visit) to letters $U$ preceded by a $U$ and (for the last visit) to letters $D$ preceded by a $D$. 
	Therefore, given a tree $T$ such that $Q(T)=A,$ the only down-steps in its contour preceded by an up-step are the first, $x_2\text{-th,}\dots$ and $x_m\text{-th}$ down-steps. Then it easily follows that $C(T)\in \Big(\prod_{j=1}^{m-1}U^+D^{x_{j+1}-x_{j}}\Big)U^+D^+.$ The same reasoning also proves that given a tree $T$ such that $$C(T)\in \Big(\prod_{j=1}^{m-1}U^+D^{x_{j+1}-x_{j}}\Big)U^+D^+,$$
	then $Q(T)=A,$ concluding the proof.
\end{proof}

We now extend the notion of contour to the local limit tree $\bm{T}^*.$ We introduce some more notation for words over the alphabet $\{U,D\}.$ We denote by $\{U,D\}^{\infty}_R$ the set of right-infinite words over the alphabet $\{U,D\},$ \emph{i.e.,} of infinite words of the form $w_1w_2\cdots w_n\cdots,$ and similarly we denote by $\{U,D\}^{\infty}_L$ the set of left-infinite words over the alphabet $\{U,D\},$ \emph{i.e.,} of infinite words of the form $\cdots w_{-n}\cdots w_{-2}w_{-1.}$ 

\begin{defn}
	A \emph{double-infinite pointed path} is an infinite word in $\{U,D\}^{\infty}_L\hat{w}\{U,D\}^{\infty}_R,$ where $\hat{w}\in\{\hat{U},\hat{D}\}.$ 
\end{defn}

\begin{defn}
	The \emph{contour} of a tree $T$ with an infinite spine is the double-infinite pointed path $$C(T)=\cdots w_{-n}\cdots w_{-1}\hat{w}_0w_1\cdots w_n\cdots,$$
	obtained turning around the tree $T$ (\emph{cf.} Fig.~\ref{inf_prof}) in the following way:
	\begin{itemize}
		\item we start the visit on the right of the vertex $u_0$ in the direction of the vertex $u_1$ (in Fig.~\ref{inf_prof} we start at the red square in the direction of the blue dashed line) and we set $\hat{w}_0=\hat{D}$ (\emph{i.e.,} $\hat{w}_0$ is the down-step corresponding to the right-side of the edge between $u_0$ and $u_1$);
		\item we proceed turning around $T,$ in the same direction (following the blue dashed line) setting, for all $i\geq 1,$ $w_i=U$ if the $i$-th visited edge is visited from top to bottom, otherwise setting $w_i=D;$
		\item we conclude coming back to the starting point (the red square) and turning around $T$ in the other direction (following the green dotted line), setting, for all $i\geq 1,$ $w_{-i}=U$ if the $i$-th visited edge is visited from bottom to top, otherwise setting $w_{-i}=D.$
	\end{itemize}
\end{defn}

In particular, the contour of the local limit tree $\bm{T}^*$ is the random double-infinite pointed path $$C(\bm{T}^*)=\cdots\bm{w}_{-n}\cdots\bm{w}_{-1}\hat{\bm{w}}_0\bm{w}_1\cdots\bm{w}_n\cdots,$$
obtained turning around the tree $\bm{T}^*$ as explained before.

\begin{obs}
	\label{rem_cont_tree}
	Note that the contour of the local limit tree $\bm{T}^*$ is consistent with the definition of contour on finite trees: the contour of a fringe subtree $f^\bullet_h(\bm{T}^*)$ (forgetting the distinguished vertex) is a factor $\bm{w}_{-k}\cdots\bm{w}_{-1}\bm{w}_0\bm{w}_1\cdots\bm{w}_{k'}$ of $C(\bm{T}^*)$ (where $\hat{\bm{w}}_0$ has been replaced by $\bm{w}_0=D.$). Note also that $\bm{w}_{-k}\cdots\bm{w}_{-1}$ contains at least $h$ up-steps and $\bm{w}_0\bm{w}_1\cdots\bm{w}_{k'}$ contains at least $h$ down-steps (corresponding to the edges $(u_i,u_{i+1}),$ for $0\leq i\leq h-1$).
\end{obs}

\begin{figure}[htbp]
	\begin{center}
		\includegraphics[scale=1]{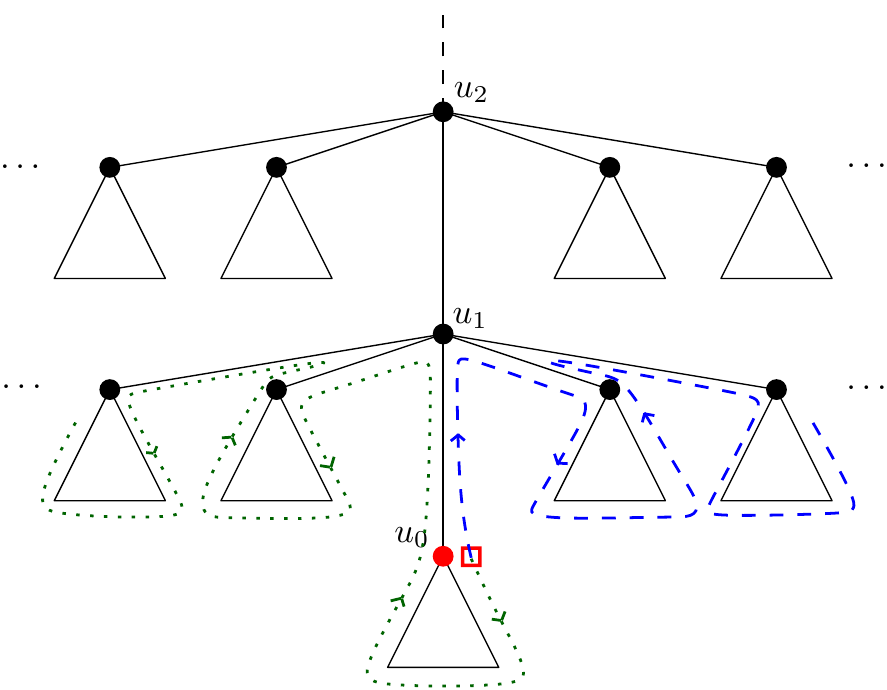}\\
		\caption{ Scheme for the construction of the contour of a tree with an infinite spine.}\label{inf_prof}
	\end{center}
\end{figure}

\begin{lem}
	\label{balanced_contour}
	Let $C(\bm{T}^*)=\cdots\bm{w}_{-n}\cdots\bm{w}_{-1}\hat{\bm{w}}_0\bm{w}_1\cdots\bm{w}_n\cdots$ be the contour of $\bm{T}^*.$ Then $\{\bm{w}_{i}\}_{i\in\mathbb{Z}\setminus\{0\}}$ is a sequence of i.i.d Bernoulli random variables on the set $\{U,D\}.$
\end{lem}
\begin{proof} 
	We fix $i>0.$ Note that if $\bm{w}_{i}=D$ and $\bm{w}_{i}$ is the down-step corresponding to the right side of an edge in $\bm{T}^*$ between the vertices $v$ and $u=a(v)$ then:
	\begin{itemize}
		\item $\bm{w}_{i+1}=D,$ if $v$ is the rightmost child of $u;$ 
		\item $\bm{w}_{i+1}=U,$ if $v$ is not the rightmost child of $u.$ 
	\end{itemize} 
 	We claim that, conditioning on $\bm{w}_{i}=D,$ the event ``$v$ is the rightmost child of $u$" is independent of the shape of the contour of $\bm{T}^*$ before the vertex $v,$ namely is independent of $\{\bm{w}_{j}\}_{j<i}.$ Indeed, the number of children of each vertex of the tree $\bm{T}^*$ is geometrically distributed (with parameter $1/2$) or size-biased geometrically distributed (with parameter $1/2$), and both distribution have the well-known memoryless property. This proves the claim. Moreover, thanks to this observation, 
 	\begin{equation*}
 	\begin{split}
 	&\P(v\text{ is the rightmost child of }u|\bm{w}_{i}=D\text{ and knowing } (\bm{w}_{j})_{j<i})=\frac{1}{2},\\
 	&\P(v\text{ is not the rightmost child of }u|\bm{w}_{i}=D\text{ and knowing } (\bm{w}_{j})_{j<i})=\frac{1}{2}.
 	\end{split}
 	\end{equation*}

 	On the other hand, if $\bm{w}_{i}=U$ and $\bm{w}_{i}$ is the up-step corresponding to the left side of an edge in $\bm{T}^*$ between the vertices $u=a(v)$ and $v$ then:
 	\begin{itemize}
 		\item $\bm{w}_{i+1}=D,$ if $v$ is a leaf;
 		\item $\bm{w}_{i+1}=U,$ if $v$ is not a leaf.
 	\end{itemize}
 	Note that $\bm{w}_{i}=U$ cannot correspond to the left side of an edge in the infinite spine (since all these edges are visited on the right side when $i>0$), hence $v\neq u_j,$ for all $j\geq 0.$
 	Moreover, we claim that, conditioning on $\bm{w}_{i}=U,$ the event ``$v$ is a leaf" is independent of the shape of the contour of $\bm{T}^*$ before the vertex $u,$ namely is independent of $\{\bm{w}_{j}\}_{j<i}.$ Indeed each vertex of the tree $\bm{T}^*$ that is not in the infinite spine is a leaf with probability $1/2,$ independently of all other vertices. This proves the claim.
 	
 	Summing up, for all $i>0,$ $\bm{w}_{i}$ is a Bernoulli random variable independent of $\{\bm{w}_{j}\}_{j<i}.$ Using a symmetric argument for $i<0,$ we can conclude the proof.
\end{proof}

\subsection{The proof of the main result}
\label{mainres_321}
We begin this section by stating two propositions that are the key results for the proof of Theorem \ref{321theorem}. Then we prove Theorem \ref{321theorem} and at the end, in Sections \ref{prop1proof} and \ref{prop2proof}, we give the proofs of the two propositions. 

We need the following.

\begin{defn}
	\label{wiufy0qiu}
	Given a rooted ordered tree $T$ with $n+1$ vertices, we define for $i,k\in[n+1],$
	$$Q_{i,k}(T)=\big\{x\in[-k,k]:x+i\in Q(T)\big\},$$
	\emph{i.e.,} $Q_{i,k}(T)$ is the set of values in $Q(T)\cap[i-k,i+k]$ shifted to the interval $[-k,k].$
	
	Analogously, given a permutation $\sigma\in\text{Av}^n(321),$ we define for $i,k\in[n],$
	$$E^+_{i,k}(\sigma)=\big\{x\in[-k,k]:x+i\in E^+(\sigma)\big\}.$$
\end{defn}

	Note that $E^+_{i,k}(\sigma)=Q_{i,k}(T_{\sigma}).$
	
\begin{prop}\label{prop1}
	For all $n\in\mathbb{Z}_{>0},$ let $\bm{\sigma}^n$ be a uniform random permutation in $\emph{Av}^n(321).$ Fix $k\in\mathbb{Z}_{>0}$ and let $\bm{i}_n$ be uniform in $[n]$ and independent of $\bm{\sigma}^n.$ Then, for every subset $A\subseteq[-k,k],$ as $n\to\infty,$
	$$\P^{\bm{i}_n}\big(E^+_{\bm{i}_n,k}(\bm{\sigma}^n)=A\big)\stackrel{P}{\to}\Big(\frac{1}{2}\Big)^{2k+1}.$$
\end{prop}

Informally, conditionally on $\bm{\sigma}^n,$ every index of $r_k(\bm{\sigma}^n,\bm{i}_n)$ is asymptotically equiprobably in $E^+$ or $E^-.$ Moreover, these events are asymptotically independent.

Before stating our second proposition, we introduce some more notation (\emph{cf.} Fig.~\ref{min_max_perm} below). Given a permutation $\sigma\in\text{Av}(321)$ we define, for $i\in[|\sigma|],$ $k\in\mathbb{Z}_{>0},$
\begin{equation*}
\begin{split}
&m_{i,k}^+(\sigma)=\min\big\{E^+(\sigma)\cap[i-k,i+k]\big\},\\
&M_{i,k}^-(\sigma)=\max\big\{E^-(\sigma)\cap[i-k,i+k]\big\},
\end{split}
\end{equation*}
with the conventions that $\min{\emptyset}=+\infty$ and $\max{\emptyset}=-\infty.$

\begin{figure}[htbp]
	\begin{center}
		\includegraphics[scale=0.40]{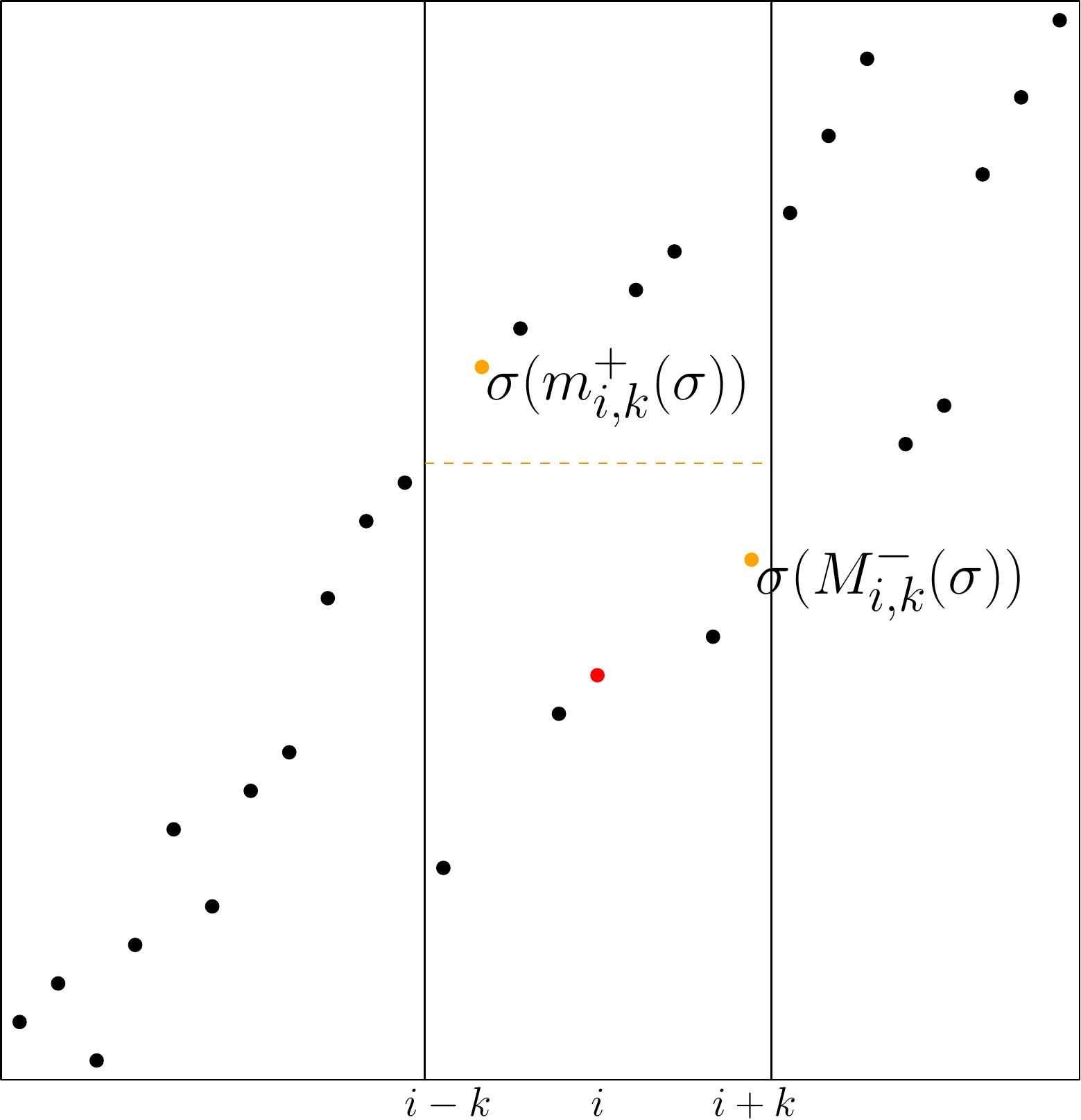}\\
		\caption{A 321-avoiding permutation. The orange dots identify the values $\sigma(m_{i,k}^+(\sigma))$ and $\sigma(M_{i,k}^-(\sigma))$ inside the vertical strip centered in $i$ of width $2k+1.$ Moreover, the dashed orange line identifies the separating line.}\label{min_max_perm}
	\end{center}
\end{figure}

We also define, for a random 321-avoiding permutation $\bm{\sigma}$ and a uniform index $\bm{i}\in[|\sigma|],$ the following events, for all $k\in\mathbb{Z}_{>0},$
\begin{equation*}
\begin{split}
S_k(\bm{\sigma},\bm{i})=&\big\{\bm{\sigma}(m_{\bm{i},k}^+(\bm{\sigma}))>\bm{\bm{\sigma}}(M_{\bm{i},k}^-(\bm{\sigma})),m_{\bm{i},k}^+(\bm{\sigma})\neq +\infty,M_{\bm{i},k}^-(\bm{\sigma})\neq -\infty\big\}\\
&\cup\big\{m_{\bm{i},k}^+(\bm{\sigma})= +\infty\big\}\cup\big\{M_{\bm{i},k}^-(\bm{\sigma})= -\infty\big\}.
\end{split}
\end{equation*}
This is the event that the random rooted permutation $r_k(\bm{\sigma},\bm{i})$ splits into two (possibly empty) increasing subsequences with separated values, namely, the minimum of the upper subsequence is greater than the maximum of the lower subsequence. We will say, if this events holds, that ``a separating line exists" (in Fig.~\ref{min_max_perm} this separating line is dashed in orange).

\begin{prop}\label{prop2}
	For all $n\in\mathbb{Z}_{>0},$ let $\bm{\sigma}^n$ be a uniform random permutation in $\emph{Av}^n(321).$ Fix $k\in\mathbb{Z}_{>0}$ and let $\bm{i}_n$ be uniform in $[n]$ and independent of $\bm{\sigma}^n.$ Then,  as $n\to\infty,$
	$$\P^{\bm{i}_n}\big(S_k(\bm{\sigma}^n,\bm{i}_n)\big)\stackrel{P}{\longrightarrow}1.$$
\end{prop}

Informally, conditionally on $\bm{\sigma}^n,$ asymptotically almost surely there exists a separating line in $r_k(\bm{\sigma}^n,\bm{i}_n)$.

With these two propositions in our hands, we can prove Theorem \ref{321theorem}.

\begin{proof}[Proof of Theorem \ref{321theorem} ]
	In order to prove Equation (\ref{shsshgea}) on page~\pageref{shsshgea}, we can check the equivalent condition (d) of Corollary \ref{detstrongbsconditions}, \emph{i.e.,} we have to prove that, for all $h\in\Z_{>0},$ for all $\pi\in\text{Av}^{2h+1}(321),$
	\begin{equation}
	\label{checkthat}
	\P^{{\bm{i}}_n}\big(r_h(\bm{\sigma}^n,\bm{i}_n)=(\pi,h+1)\big)\stackrel{P}{\to}P_{321}(\pi).
	\end{equation}
	We fix $h\in\Z_{>0}$ and we distinguish three different cases for $\pi\in\text{Av}^{2h+1}(321)$:
	
	\underline{$\coc(21,\pi^{-1})=1:$} We recall (see Remark \ref{rem_sep_line}) that, in this case, $\pi$ splits in a unique way into two increasing subsequences whose indices are denoted by $L(\pi)$ and $U(\pi).$ Note that, in general, these two sets are different from $E^-(\pi)$ and $E^+(\pi)$ (specifically when the permutation has some fixed points at the beginning, for example as in the case of Fig.~\ref{321ex} on page~\pageref{321ex}). We set $U^*(\pi)\coloneqq\{j\in[-h,h]:j+h\in U(\pi)\},$ \emph{i.e.,} $U^*(\pi)$ denotes the shifting of the set $U(\pi)$ inside the interval $[-h,h].$ 
	Conditioning on $S_h(\bm{\sigma}^n,\bm{i}_n),$ namely assuming there is a separating line, $r_h(\bm{\sigma}^n,\bm{i}_n)$ is the union of two increasing subsequences with $U^*(r_h(\bm{\sigma}^n,\bm{i}_n))=E^+_{\bm{i}_n,h}(\bm{\sigma}^n).$ In particular, using Proposition \ref{prop2}, asymptotically we have that,
	$$\P^{{\bm{i}}_n}\big(r_h(\bm{\sigma}^n,\bm{i}_n)=(\pi,h+1)\big)=\P^{{\bm{i}}_n}\big(S_h(\bm{\sigma}^n,\bm{i}_n),E^+_{\bm{i}_n,h}(\bm{\sigma}^n)=U^*(\pi)\big)+o(1).$$
	By Proposition \ref{prop1} and \ref{prop2}, we have that
	$$\P^{{\bm{i}}_n}\big(S_h(\bm{\sigma}^n,\bm{i}_n),E^+_{\bm{i}_n,h}(\bm{\sigma}^n)=U^*(\pi)\big)\stackrel{P}{\to}\frac{1}{2^{2h+1}}=P_{321}(\pi),$$
	and so Equation (\ref{checkthat}) holds.
	
	\underline{$\pi=12\dots|\pi|$:} In this case $\pi$ does not split in a unique way into two increasing subsequences. We have exacly $2h+2$ different choices for this splitting, namely setting either $U(\pi)=\emptyset$ or $U(\pi)=[k,h],$ for some $k\in[-h,h].$ Therefore, using Proposition \ref{prop2}, asymptotically we obtain that
	$$\P^{{\bm{i}}_n}\big(r_h(\bm{\sigma}^n,\bm{i}_n)=(\pi,h+1)\big)=\P^{{\bm{i}}_n}\Big(S_h(\bm{\sigma}^n,\bm{i}_n),E^+_{\bm{i}_n,h}(\bm{\sigma}^n)\in\{\emptyset\}\cup\bigcup_{k\in[-h,h]}\{[k,h]\}\Big)+o(1).$$
	Again by Proposition \ref{prop1} and \ref{prop2}, we have that
	$$\P^{{\bm{i}}_n}\Big(S_h(\bm{\sigma}^n,\bm{i}_n),E^+_{\bm{i}_n,k}(\bm{\sigma}^n)\in\{\emptyset\}\cup\bigcup_{k\in[-h,h]}\{[k,h]\}\Big)\stackrel{P}{\to}\frac{2h+2}{2^{2h+1}}=P_{321}(\pi),$$
	and so Equation (\ref{checkthat}) holds.
		
	\underline{$\coc(21,\pi^{-1})>1:$} In this case $\pi$ does not present a separating line. Since by Proposition \ref{prop2} we know that asymptotically almost surely there exists a separating line in $r_h(\bm{\sigma}^n,\bm{i}_n)$ we can immediately conclude that,
	\begin{equation*}
	\P^{{\bm{i}}_n}\big(r_h(\bm{\sigma}^n,\bm{i}_n)=(\pi,h+1)\big)\leq\P^{{\bm{i}}_n}\big(S_h(\bm{\sigma}^n,\bm{i}_n)^c\big)\stackrel{P}{\to}0=P_{321}(\pi).\qedhere
	\end{equation*}
\end{proof}

\subsubsection{The proof of Proposition \ref{prop1}}
\label{prop1proof}
We recall that $\bm{T}^{\eta}_n$ denotes a Galton--Watson tree  with offspring distribution $\eta\sim Geom(1/2)$ conditioned on having $n$ vertices and $\bm{i_n}$ is a uniform index in $[n]$ independent of the other random objects. We label the tree $\bm{T}^{\eta}_n$ with the post-order labeling starting from 1 and, as said in Remark \ref{postorderfprtree}, we denote with $(\bm{T}^{\eta}_n,\bm{i_n})$ the random pointed rooted ordered tree obtained from $\bm{T}^{\eta}_n$ by distinguishing the vertex $\bm{i_n}.$ We also recall that  $A\subseteq[-k,k]$.

\begin{obs}
	\label{elvfdowcv}
	Thanks to Observation \ref{obs_treeperm} and Definition \ref{wiufy0qiu},
	$$\P^{\bm{i}_n}\big(E^+_{\bm{i}_n,k}(\bm{\sigma}^n)=A\big)=\P^{\bm{i}_n}\big(Q_{\bm{i}_n,k}(\bm{T}^{\eta}_{n+1})=A\big)=\P^{\bm{i}_{n+1}}\big(Q_{\bm{i}_{n+1},k}(\bm{T}^{\eta}_{n+1})=A\big)+O(1/n),$$
	where in the last equality we shifted the index $\bm{i}_n$ to $\bm{i}_{n+1}$ for later convenience.
	Therefore instead of studying the event $\big\{E^+_{\bm{i}_n,k}(\bm{\sigma}^n)=A\big\}$ that appears in the statement of Proposition \ref{prop1}, we can study the event $\big\{Q_{\bm{i}_n,k}(\bm{T}^{\eta}_n)=A\big\}.$ 
\end{obs}

We need to introduce some technical results in order to study the event $\big\{Q_{\bm{i}_n,k}(\bm{T}^{\eta}_n)=A\big\}$ using Proposition \ref{keycoroll}. Informally, this event is not ``local" in full generality, but it is ``local" asymptotically almost surely.
\begin{defn}
	Given a tree $T$ we say that a vertex $v$ is \emph{before} (resp. \emph{after}) a vertex $u$ if the label of $v$ is smaller (resp. greater) than the label of $u.$ 
\end{defn}

For all $H>0,$ we consider the following events,
\begin{equation*}
\begin{split}
E_H(\bm{T}^{\eta}_n,\bm{i}_n)=\big\{f^\bullet_H(\bm{T}^{\eta}_n,\bm{i}_n)&\text{ has at least }k\text{ vertices before and}\\
&\qquad k\text{ vertices after the vertex }\bm{i}_n\big\}
\end{split}
\end{equation*}
and 
\begin{equation*}
\begin{split}
E_H(\bm{T}^*)=\big\{f^\bullet_H(\bm{T}^*)&\text{ has at least }k\text{ vertices before and}\\
&\qquad k\text{ vertices after the vertex }u_0\big\}.
\end{split}
\end{equation*}

\begin{lem}\label{lemma_A}
	We have the following:
	\begin{enumerate} [(a)] 
		\item for all $H>0,$
		$\P^{\bm{i}_n}\big(E_H(\bm{T}^{\eta}_n,\bm{i}_n)\big)\stackrel{P}{\to}\P\big(E_H(\bm{T}^*)\big),$ as $n\to\infty;$
		\item $\P\big(E_H(\bm{T}^*)\big)\stackrel{H\to\infty}{\longrightarrow}1.$
	\end{enumerate}	
\end{lem}

\begin{proof}
	(a) This is a trivial application of Proposition  \ref{keycoroll};
	
	(b) The result it is an immediate consequence of the fact that the limiting tree $\bm{T}^*$ has a.s.\ infinitely many vertices on the left and on the right of the infinite spine.
\end{proof}

We note that, for all $H,n>0,$ conditioning on the event $E_H(\bm{T}^{\eta}_n,\bm{i}_n),$ we can rewrite the event $\big\{Q_{\bm{i}_n,k}(\bm{T}^{\eta}_n)=A\big\}$ as
$$\big\{f^\bullet_H(\bm{T}^{\eta}_n,\bm{i}_n)\in\mathcal{T}_{k}(A)\big\},$$
where $\mathcal{T}_{k}(A)$ is the set of all pointed rooted ordered trees $T^\bullet=(T,j)$ with at least $k$ vertices before and $k$ vertices after the vertex $j$ such that $Q_{j,k}(T)=A.$ 
Therefore, using Proposition \ref{keycoroll}, for all $H>0,$
\begin{equation}
\label{eq_B}
\P^{\bm{i}_n}\big(\big\{Q_{\bm{i}_n,k}(\bm{T}^{\eta}_n)=A\big\}\cap E_H(\bm{T}^{\eta}_n,\bm{i}_n)\big)\stackrel{P}{\to}\P\big(\big\{f^\bullet_H(\bm{T}^*)\in\mathcal{T}_{k}(A)\big\}\cap E_H(\bm{T}^*)\big).
\end{equation}

\begin{lem}\label{lemma_C} With the same notation as above,
	$$\P\big(\big\{f^\bullet_H(\bm{T}^*)\in\mathcal{T}_{k}(A)\big\}\cap E_H(\bm{T}^*)\big)\stackrel{H\to\infty}{\longrightarrow}\Big(\frac{1}{2}\Big)^{2k+1}.$$
\end{lem}

\begin{proof} 
	Suppose that $A=\big\{x_{m},\dots,x_{1},y_{1},\dots,y_{\ell}\big\},$ where $m,\ell\leq k,$ $x_{i+1}<x_{i}<0$ for all $1\leq i\leq m-1$ and $0\leq y_i < y_{i+1}$ for all $1\leq i\leq \ell-1.$ For all $H>0,$ using the same ideas as in Lemma \ref{lem_con_lab} on the right and left part of the contour of $\bm{T}^*$, the event $\big\{f^\bullet_H(\bm{T}^*)\in\mathcal{T}_{k}(A)\big\}$ rewrites as 
	$$\Big\{C(\bm{T}^*)\in\{U,D\}^{\infty}_LD^{|-k-x_{m}|+1}\prod_{p=m-1}^{1} w_pU^{+}D^{|x_{1}|-1}\hat{D}D^{y_1}U^{+}\prod_{q=1}^{\ell-1} w_qD^{k-y_{\ell}+1}\{U,D\}^{\infty}_R\Big\},$$ 
	where $w_p=U^{+}D^{|x_{p+1}-x_{p}|},$ for all $1\leq p \leq m-1$ and $w_q=D^{y_{q+1}-y_{q}}U^{+},$ for all $1\leq q \leq \ell-1$ (see Fig.~\ref{proof_cont} for a representation of this event). 
	
	\begin{figure}[htbp]
		\begin{center}
			\includegraphics[scale=0.8]{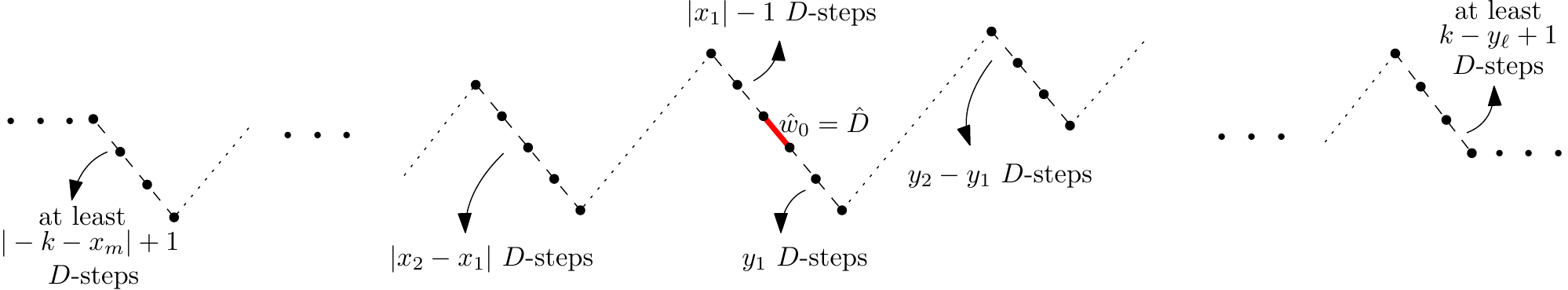}\\
			\caption{A scheme of the contour of $C(\bm{T}^*).$ We used dashed lines for sequences of down-steps with a fixed size (explicitly reported) and dotted lines for sequences of up-steps with arbitrary size at least 1.}\label{proof_cont}
		\end{center}
	\end{figure}
	
	Using Lemma \ref{balanced_contour}, and noting that for all $k>0,$ $\sum_{j\geq 1}\P(\bm{w}_{k+1}\cdots\bm{w}_{k+j}=U^{j})=\sum_{j\geq 1}(1/2)^j=1$ (and a similar results holds also for $k<0$), we have that 
	\begin{equation*}
	\begin{split}
	\P\Big(&C(\bm{T}^*)\in\{U,D\}_L^{\infty}D^{|-k-x_{m}|+1}\prod_{p=m-1}^{1} w_pU^{+}D^{|x_1|-1}\hat{D}D^{y_1}U^{+}\prod_{q=1}^{\ell-1} w_qD^{k-y_{\ell}+1}\{U,D\}^{\infty}_R\Big)\\
	&=\Big(\frac{1}{2}\Big)^{k+x_{m}+1}\prod_{p=m-1}^{1}\Big(\frac{1}{2}\Big)^{x_{p}-x_{p+1}}\Big(\frac{1}{2}\Big)^{-x_{1}-1+y_{1}}\prod_{q=1}^{\ell-1}\Big(\frac{1}{2}\Big)^{y_{q+1}-y_{q}}\Big(\frac{1}{2}\Big)^{k-y_{\ell}+1}=\Big(\frac{1}{2}\Big)^{2k+1}.
	\end{split}
	\end{equation*}
	Since by Lemma \ref{lemma_A} (b), $\P\big(E_H(\bm{T}^*)\big)\stackrel{H\to\infty}{\longrightarrow}1,$ the statement is proved.
\end{proof}

With these results in our hands we can finish the proof of Proposition \ref{prop1}.

\begin{proof}[Proof of Proposition \ref{prop1}]
	We start studying the conditional probability $\P^{\bm{i}_n}\big(Q_{\bm{i}_n,k}(\bm{T}^{\eta}_n)=A\big),$ which, we recall, is a random variable.
	For all $n,H>0,$ almost surely,
	\begin{equation*}
	\begin{split}
	\P^{\bm{i}_n}\big(Q_{\bm{i}_n,k}(\bm{T}^{\eta}_n)=A\big)= &\P^{\bm{i}_n}\big(\big\{Q_{\bm{i}_n,k}(\bm{T}^{\eta}_n)=A\big\}\cap E_H(\bm{T}^{\eta}_n,\bm{i}_n)\big)\\
	&+\P^{\bm{i}_n}\big(\big\{Q_{\bm{i}_n,k}(\bm{T}^{\eta}_n)=A\big\}\cap E_H(\bm{T}^{\eta}_n,\bm{i}_n)^c\big).
	\end{split}
	\end{equation*}
	Therefore, almost surely,
	\begin{equation}
	\label{keybound}
	\begin{split}
	&\P^{\bm{i}_n}\big(Q_{\bm{i}_n,k}(\bm{T}^{\eta}_n)=A\big)\geq\P^{\bm{i}_n}\big(\big\{Q_{\bm{i}_n,k}(\bm{T}^{\eta}_n)=A\big\}\cap E_H(\bm{T}^{\eta}_n,\bm{i}_n)\big),\\
	&\P^{\bm{i}_n}\big(Q_{\bm{i}_n,k}(\bm{T}^{\eta}_n)=A\big)\leq\P^{\bm{i}_n}\big(\big\{Q_{\bm{i}_n,k}(\bm{T}^{\eta}_n)=A\big\}\cap E_H(\bm{T}^{\eta}_n,\bm{i}_n)\big)+\P^{\bm{i}_n}\big(E_H(\bm{T}^{\eta}_n,\bm{i}_n)^c\big).
	\end{split}
	\end{equation}
	By Equation (\ref{eq_B}), for all $H>0,$ the following limit for the lower bound in Equation (\ref{keybound}) holds
	\begin{equation}
	\label{firstlim}
	\P^{\bm{i}_n}\big(\big\{Q_{\bm{i}_n,k}(\bm{T}^{\eta}_n)=A\big\}\cap E_H(\bm{T}^{\eta}_n,\bm{i}_n)\big)\stackrel{P}{\to}\P\big(\big\{f^\bullet_H(\bm{T}^*)\in\mathcal{T}_{k}(A)\big\}\cap E_H(\bm{T}^*)\big).
	\end{equation}
	Combining Equation (\ref{eq_B}) and Lemma \ref{lemma_A} (a), for all $H>0,$ the following limit for the upper bound in Equation (\ref{keybound}) holds
	\begin{equation}
	\label{secondlim}
	\begin{split}
	\P^{\bm{i}_n}\big(\big\{Q_{\bm{i}_n,k}(\bm{T}^{\eta}_n)=A\big\}&\cap E_H(\bm{T}^{\eta}_n,\bm{i}_n)\big)+\P^{\bm{i}_n}\big(E_H(\bm{T}^{\eta}_n,\bm{i}_n)^c\big)\\
	&\stackrel{P}{\to}\P\big(\big\{f^\bullet_H(\bm{T}^*)\in\mathcal{T}_{k}(A)\big\}\cap E_H(\bm{T}^*)\big)+\P\big(E_H(\bm{T}^*)^c\big).
	\end{split}
	\end{equation}
	Since the limiting objects are deterministic, the limits in Equations (\ref{firstlim}) and (\ref{secondlim}) hold jointly for all $H>0$. Moreover,
	applying Skorokhod's representation theorem in the product space over all $H>0,$ we can also assume that there exists a coupling such that these limits hold also almost surely jointly for all $H>0$.  Therefore, almost surely, for all $H>0,$
	\begin{equation}
	\label{seckeybound}
	\begin{split}
	&\liminf_{n\to\infty}\P^{\bm{i}_n}\big(Q_{\bm{i}_n,k}(\bm{T}^{\eta}_n)=A\big)\geq\P\big(\big\{f^\bullet_H(\bm{T}^*)\in\mathcal{T}_{k}(A)\big\}\cap E_H(\bm{T}^*)\big),\\
	&\limsup_{n\to\infty}\P^{\bm{i}_n}\big(Q_{\bm{i}_n,k}(\bm{T}^{\eta}_n)=A\big)
	\leq\P\big(\big\{f^\bullet_H(\bm{T}^*)\in\mathcal{T}_{k}(A)\big\}\cap E_H(\bm{T}^*)\big)+\P\big(E_H(\bm{T}^*)^c\big).
	\end{split}
	\end{equation}
	Finally, since by Lemma \ref{lemma_C},
	$$\P\big(\big\{f^\bullet_H(\bm{T}^*)\in\mathcal{T}_{k}(A)\big\}\cap E_H(\bm{T}^*)\big)\stackrel{H\to\infty}{\longrightarrow}\Big(\frac{1}{2}\Big)^{2k+1}$$ 
	and by Lemma \ref{lemma_A} (a),
	$$\P\big(E_H(\bm{T}^*)^c\big)\stackrel{H\to\infty}{\longrightarrow}0,$$
	we can conclude, from Equation (\ref{seckeybound}), that the limit in probability of $\P^{\bm{i}_n}\big(Q_{\bm{i}_n,k}(\bm{T}^{\eta}_n)=A\big)$ exists and is equal to $\big(\frac{1}{2}\big)^{2k+1}.$
	Finally, thanks to Observation \ref{elvfdowcv}, 
	\begin{equation*}
	\lim_{n\to\infty}\P^{\bm{i}_n}\big(E^+_{\bm{i}_n,k}(\bm{\sigma}^n)=A\big)=\lim_{n\to\infty}\P^{\bm{i}_n}\big(Q_{\bm{i}_n,k}(\bm{T}^{\eta}_n)=A\big)=\Big(\frac{1}{2}\Big)^{2k+1}.\qedhere
	\end{equation*}
\end{proof}

\subsubsection{The proof of Proposition \ref{prop2}}
\label{prop2proof}
It remains to prove the following:
$$\P^{\bm{i}_n}\big(S_k(\bm{\sigma}^n,\bm{i}_n)\big)\stackrel{P}{\longrightarrow}1.$$

Before entering into the details, we first give a quick idea of the proof: by a result from \cite{hoffman2017pattern}, the points of a large uniform 321-avoiding permutation of size $n$ are at distance of order $\sqrt n$ from the diagonal $x=y.$ Therefore, in a window of constant width, the points above the diagonal are all higher than the points below the diagonal. 

\begin{proof}[Proof of Proposition \ref{prop2}]
	In order to simplify notation we fix $\bm{m}^+=m_{\bm{i}_n,k}^+(\bm{\sigma}^n)$ and $\bm{M}^-=M_{\bm{i}_n,k}^-(\bm{\sigma}^n).$ We split the probability $\P\big(S_k(\bm{\sigma}^n,\bm{i}_n)\big)$ into two different terms as follow,
	\begin{equation*}
	\begin{split}
	\P\big(S_k(\bm{\sigma}^n,\bm{i}_n)\big)=&\P\big(\bm{\sigma}^n(\bm{m}^+)>\bm{\bm{\sigma}^n}(\bm{M}^-),\bm{m}^+\neq +\infty,\bm{M}^-\neq -\infty\big)\\
	&+\P\Big(\big\{\bm{m}^+\neq+\infty,\bm{M}^-\neq-\infty\big\}^c\Big).
	\end{split}
	\end{equation*}
	We analyze the first term,
	\begin{equation*}
	\begin{split}
	&\P\big(\bm{\sigma}^n(\bm{m}^+)>\bm{\bm{\sigma}^n}(\bm{M}^-),\bm{m}^+\neq +\infty,\bm{M}^-\neq -\infty\big)\\
	&=\P\big(\bm{\sigma}^n(\bm{m}^+)-\bm{m}^+-(\bm{\bm{\sigma}^n}(\bm{M}^-)-\bm{M}^-)>\bm{M}^--\bm{m}^+,\bm{m}^+\neq +\infty,\bm{M}^-\neq -\infty\big).
	\end{split}
	\end{equation*}
	Setting for all $i\in[n],$ $\Delta_i(\bm{\sigma}^n)=\bm{\sigma}^n(i)-i,$ we have
	\begin{equation}
	\label{blawfuighfpuw}
	\begin{split}
	&\P\big(\bm{\sigma}^n(\bm{m}^+)>\bm{\bm{\sigma}^n}(\bm{M}^-),\bm{m}^+\neq +\infty,\bm{M}^-\neq -\infty\big)\\
	&=\P\big(\Delta_{\bm{m}^+}(\bm{\sigma}^n)-\Delta_{\bm{M}^-}(\bm{\sigma}^n)>\bm{M}^--\bm{m}^+,\bm{m}^+\neq +\infty,\bm{M}^-\neq -\infty\big).
	\end{split}
	\end{equation}
	We now analyze the behavior of $\Delta_{\bm{m}^+}(\bm{\sigma}^n)$ and $\Delta_{\bm{M}^-}(\bm{\sigma}^n).$
	By an easy application (stated in \cite[p.10]{janson2017patterns321}) of a result due to Hoffman, Rizzolo and Slivken (see \cite{hoffman2017pattern}) we can assume that there exist a coupling for $(\{\bm{\sigma}^n\}_{n\in\Z_{>0}},\bm{e}),$ where $\bm{e}$ is a Brownian excursion, such that, almost surely,
	\begin{equation*}
	\begin{split}
	&\Delta_{j}(\bm{\sigma}^n)=\sqrt{2n}\cdot\bm{e}(j/n)+\varepsilon_n^+(j,\bm{e},\bm{\sigma}^n),\quad\text{for all}\quad j\in E^+(\bm{\sigma}^n),\\
	&\Delta_{j}(\bm{\sigma}^n)=-\sqrt{2n}\cdot\bm{e}(j/n)+\varepsilon_n^-(j,\bm{e},\bm{\sigma}^n),\quad\text{for all}\quad j\in E^-(\bm{\sigma}^n),
	\end{split}
	\end{equation*}
	where $\varepsilon_n^*(j,\bm{e},\bm{\sigma}^n)$ are error terms in probability of order $o(n^{1/2})$ uniformly in $j,$ \emph{i.e.,} $\frac{\sup_j\{\varepsilon_n^*(j,\bm{e},\bm{\sigma}^n)\}}{n^{1/2}}\stackrel{P}{\to}0$ (w.r.t.\ the probability given by the coupling of $(\{\bm{\sigma}^n\}_{n\in\Z_{>0}},\bm{e})$).
	
	Therefore, setting $j=\bm{m}^+$ or $j=\bm{M}^-,$ we obtain, almost surely,
	\begin{equation*}
	\begin{split}
	&\Delta_{\bm{m}^+}(\bm{\sigma}^n)=\sqrt{2n}\cdot\bm{e}(\bm{m}^+/n)+\varepsilon_n^+(\bm{m}^+,\bm{e},\bm{\sigma}^n),\\
	&\Delta_{\bm{M}^-}(\bm{\sigma}^n)=-\sqrt{2n}\cdot\bm{e}(\bm{M}^-/n)+\varepsilon_n^-(\bm{M}^-,\bm{e},\bm{\sigma}^n),
	\end{split}
	\end{equation*}
	where $\frac{\varepsilon_n^+(\bm{m}^+,\bm{e},\bm{\sigma}^n)}{n^{1/2}}\stackrel{P}{\to}0$ (w.r.t.\ the probability given by the coupling of $(\{\bm{\sigma}^n\}_{n\in\Z_{>0}},\bm{e},\{\bm{i}_n\}_{n\in\Z_{>0}})$) since, almost surely, $\varepsilon_n(\bm{m}^+,\bm{e},\bm{\sigma}^n)\leq\sup_j\{\varepsilon_n(j,\bm{e},\bm{\sigma}^n)\}.$ The same holds also for $\varepsilon_n^-(\bm{M}^-,\bm{e},\bm{\sigma}^n)$. In what follows all the probabilities and the error terms have to be interpreted in the same way, \emph{i.e}. w.r.t.\ the probability given by the coupling of $(\{\bm{\sigma}^n\}_{n\in\Z_{>0}},\bm{e},\{\bm{i}_n\}_{n\in\Z_{>0}}).$ Moreover, we will simply denote the errors terms of order less than $n^\alpha$ with $\bm{o}_P(n^\alpha).$
	
	We can rewrite the last term in Equation (\ref{blawfuighfpuw}) as
	\begin{equation}
	\label{piegfpogfpfh}
	\begin{split}
	&\P\big(\bm{\sigma}^n(\bm{m}^+)>\bm{\bm{\sigma}^n}(\bm{M}^-),\bm{m}^+\neq +\infty,\bm{M}^-\neq -\infty\big)\\
	&=\P\Big(\bm{e}(\bm{m}^+/n)+\bm{e}(\bm{M}^-/n)>\tfrac{\bm{M}^--\bm{m}^+}{\sqrt{2n}}+\bm{o}_p(1),\bm{m}^+\neq +\infty,\bm{M}^-\neq -\infty\Big).
	\end{split}
	\end{equation}
	
	Noting that, if $\bm{m}^+\neq +\infty$ and $\bm{M}^-\neq -\infty$, then $-k\leq\bm{m}^+-\bm{i}_n,\bm{M}^--\bm{i}_n\leq k,$ a.s., we have that
	$$\frac{1}{n}|\bm{m}^+-\bm{i}_n|=\bm{o}_p(1)\quad\text{and}\quad\frac{1}{n}|\bm{M}^--\bm{i}_n|=\bm{o}_p(1).$$
	
	Since $\bm{e}$ is a continuous process, applying Lemma \ref{tecn_lemma} in the appendix, we can rewrite Equation (\ref{piegfpogfpfh}) as
	\begin{equation}
	\begin{split}
	&\P\big(\bm{\sigma}^n(\bm{m}^+)>\bm{\bm{\sigma}^n}(\bm{M}^-),\bm{m}^+\neq +\infty,\bm{M}^-\neq -\infty\big)\\
	&=\P\Big(\bm{e}(\bm{i}_n/n)>\tfrac{\bm{M}^--\bm{m}^+}{2\sqrt{2n}}+\bm{o}_p(1),\bm{m}^+\neq +\infty,\bm{M}^-\neq -\infty\Big).
	\end{split}
	\end{equation}
	Noting that, if $\bm{m}^+\neq +\infty,\bm{M}^-\neq -\infty$, than $\bm{M}^--\bm{m}^+\leq 2k,$ a.s., we can obtain the following bound,
	\begin{equation}
	\begin{split}
	&\P\big(\bm{\sigma}^n(\bm{m}^+)>\bm{\bm{\sigma}^n}(\bm{M}^-),\bm{m}^+\neq +\infty,\bm{M}^-\neq -\infty\big)\\
	&\geq\P\Big(\bm{e}(\bm{i}_n/n)+\bm{o}_p(1)>\tfrac{k}{\sqrt{2n}},\bm{m}^+\neq +\infty,\bm{M}^-\neq -\infty\Big).
	\end{split}
	\end{equation}
	From now until the end of the proof we denote with $\{\bm{\varepsilon}_n\}_{n\in\Z_{>0}}$ the family of random error terms that appears in the previous equation. Since $\bm{i}_n/n\stackrel{(d)}{\to}\bm{U},$ where $\bm{U}$ is a uniform random variable in $(0,1),$ by Skorokhod's representation theorem, we can assume that the convergence holds almost surely. Since $\bm{e}$ is a.s.\ continuous, we can deduce that
	\begin{equation}
	\label{reqi0qugv1g1}
	\bm{e}(\bm{i}_n/n)+\bm{\varepsilon}_n\stackrel{(d)}{\to}\bm{e}(\bm{U}).
	\end{equation} 
	In particular, since $\bm{e}$ and $\bm{U}$ are independent ($\bm{i}_n$ is independent of $\bm{\sigma}^n,$ hence of $\bm{e}$), $\bm{e}(\bm{U})$ admits a density (called Rayleigh distribution, see \cite[p.5]{bertoin1994path}) and so the distribution function $F_{\bm{e}(\bm{U})}(t)$ is continuous. From Equation (\ref{reqi0qugv1g1}), we can conclude that the convergence 
	\begin{equation*}
	F_{\bm{e}(\bm{i}_n/n)+\bm{\varepsilon}_n}(t)\to F_{\bm{e}(\bm{U})}(t),
	\end{equation*}
	is uniform in $\mathbb{R}$ as stated in \cite[exercise 14.8, p.198]{billingsley2008probability}. Therefore,
	\begin{equation*}
	\P\Big(\bm{e}(\bm{i}_n/n)+\bm{\varepsilon}_n>\frac{k}{\sqrt{2n}}\Big)=1-F_{\bm{e}(\bm{i}_n/n)+\bm{\varepsilon}_n}\Big(\frac{k}{\sqrt{2n}}\Big)\to 1-F_{\bm{e}(\bm{U})}(0)=1.
	\end{equation*}
	Finally, summing up, we obtain that
	\begin{equation*}
	\lim_{n\to\infty}\P\big(S_k(\bm{\sigma}^n,\bm{i}_n)\big)\geq\lim_{n\to\infty}\P\big(\bm{m}^+\neq +\infty,\bm{M}^-\neq -\infty\big)+\P\Big(\big\{\bm{m}^+\neq+\infty,\bm{M}^-\neq-\infty\big\}^c\Big)=1.
	\end{equation*}
	Noting that $\P\big(S_k(\bm{\sigma}^n,\bm{i}_n)\big)=\E^{\bm{\sigma}^n}\Big[\P^{\bm{i}_n}\big(S_k(\bm{\sigma}^n,\bm{i}_n)\big)\Big],$ and using Markov's inequality for the probability $1-\P^{\bm{i}_n}\big(S_k(\bm{\sigma}^n,\bm{i}_n)\big)$, we can conclude that,
	\begin{equation*}
	\P^{\bm{i}_n}\big(S_k(\bm{\sigma}^n,\bm{i}_n)\big)\stackrel{P}{\longrightarrow}1.\qedhere
	\end{equation*}
\end{proof}

\subsection{The construction of the limiting object}
\label{explcon2}

We now exhibit an explicit construction of the limiting object $\bm{\sigma}_{321}^{\infty}$ as a random order $\bm{\preccurlyeq}_{321}$ on $\Z.$ 

We consider the set of integer numbers $\Z,$ and we label, uniformly and independently, each integer with a ``$+$" or a ``$-$": namely, for all $x\in\Z,$ 
$$\P(x \text{ has label }``+")=\frac{1}{2}=\P(x \text{ has label }``-").$$
We set $A^+\coloneqq\{x\in\Z:x\text{ has label }``+"\}$ and $A^-\coloneqq\{x\in\Z:x\text{ has label }``-"\}.$
Then we define a random total order $\bm{\preccurlyeq}_{321}$ on $\Z$ saying that, for all $x,y\in\Z,$ $x\bm{\preccurlyeq}_{321}y$ if either
$x<y$ and $x,y\in A^-,$ or $x<y$ and $x,y\in A^+,$ or $x\in A^-$ and $y\in A^+.$ An example is given in Fig.~\ref{constr_order321}.

\begin{figure}[htbp]
	\begin{center}
		\includegraphics[scale=0.8]{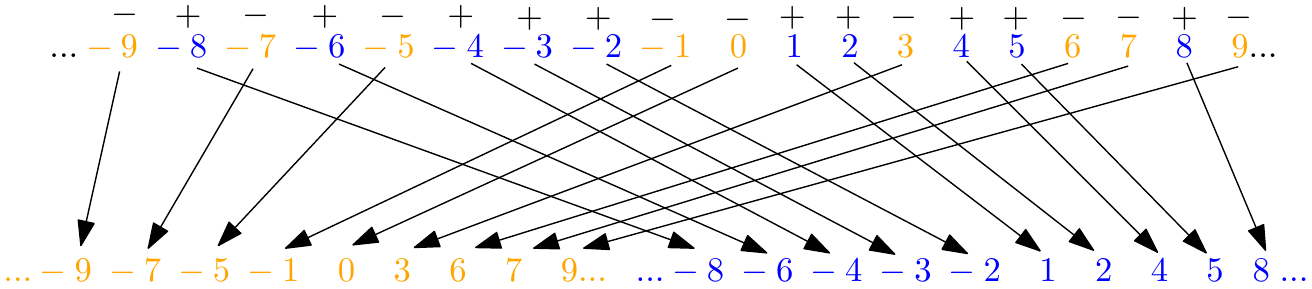}\\
		\caption{A sample of the random total order $(\Z,\bm{\preccurlyeq}_{321}).$ On the top, the standard total order on $\Z$ with the integers labeled by random ``-" and ``+" signs. We paint in orange the integers labeled by ``-" and in blue the integers labeled by ``+." Then, in the bottom part of the picture, we move the orange numbers at the beginning of the new random total order and the blue numbers at the end. Reading the bottom line from left to right gives the random total order $\bm{\preccurlyeq}_{321}$ on $\Z.$ }\label{constr_order321}
	\end{center}
\end{figure}

\begin{prop}
	Let $(\Z,\bm{\preccurlyeq}_{321})$ be the random total order defined above and $\bm{\sigma}^\infty_{321}$ be the limiting object defined in Corollary \ref{321corol}. Then
	$$(\Z,\bm{\preccurlyeq}_{321})\stackrel{(d)}{=}\bm{\sigma}^\infty_{321}.$$
\end{prop}

\begin{proof}
	Since by Observation \ref{separating class} the set of clopen balls 
	\begin{equation*}
	\mathcal{A}=\Big\{B\big((A,\preccurlyeq),2^{-h}\big):h\in\Z_{>0},(A,\preccurlyeq)\in\Sri\Big\}
	\end{equation*}
	is a separating class for the space $(\Sri,d),$ it is enough to prove that for all $h\in\Z_{>0},$ for all $\pi\in\text{Av}^{2h+1}(321),$
	$$\P\big(r_h(\Z,\bm{\preccurlyeq}_{321})=(\pi,h+1)\big)=P_{321}(\pi).$$

	We fix $h\in\Z_{>0}$ and, as in the proof of Theorem \ref{321theorem}, we distinguish three different cases for $\pi\in\text{Av}^{2h+1}(321)$:
	
	\underline{$\coc(21,\pi^{-1})=1:$} We recall (see Remark \ref{rem_sep_line}) that, in this case, $\pi$ splits in a unique way into two increasing subsequences whose indices are denoted by $L(\pi)$ and $U(\pi).$ Recalling that $U^*(\pi)\coloneqq\{j\in[-h,h]:j+h\in U(\pi)\},$ we have
	\begin{equation*}
	\begin{split}
	\P\big(r_h(\Z,\bm{\preccurlyeq}_{321})=(\pi,h+1)\big)=\P\big(A^+\cap[-h,h]=U^*(\pi)\big)=\frac{1}{2^{2h+1}}=P_{321}(\pi).
	\end{split}
	\end{equation*}
	
	\underline{$\pi=12\dots|\pi|$:} In this case $\pi$ does not split in a unique way into two increasing subsequences. We have exactly $2h+2$ different choices for this splitting, namely setting either $U^*(\pi)=\emptyset$ or $U^*(\pi)=[k,h],$ for some $k\in[-h,h].$ Therefore, we obtain that
	\begin{equation*}
	\begin{split}
	\P\big(r_h(\Z,\bm{\preccurlyeq}_{321})=(\pi,h+1)\big)=\P\Big(A^+\cap[-h,h]\in\{\emptyset\}\cup\bigcup_{k\in[-h,h]}\{[k,h]\}\Big)=\frac{2h+2}{2^{2h+1}}=P_{321}(\pi).
	\end{split}
	\end{equation*}
	
	\underline{$\coc(21,\pi^{-1})>1:$} Since, by construction, $r_h(\Z,\bm{\preccurlyeq}_{321})$ contains a.s.\ at most one inverse descent, then
	$$\P\big(r_h(\Z,\bm{\preccurlyeq}_{321})=(\pi,h+1)\big)=0=P_{321}(\pi).$$
	The analysis of this three cases concludes the proof.
\end{proof}

\appendix

\section{A generating function proof that $P_{231}(\pi)$ is a probability distribution} 
\label{combint}
We give here a generating function proof of the fact that  
\begin{equation}
\label{probdens}
P_{231}(\pi)\coloneqq\frac{2^{|\text{LRMax}(\pi)|+|\text{RLMax}(\pi)|}}{2^{2|\pi|}},\quad\text{for all}\quad \pi\in\text{Av}(231),
\end{equation}
defines a probability distribution on $\text{Av}^k(231),$ for all $k\geq 1.$ 

Using our bijection between 231-avoiding permutations and binary trees, we can consider the generating function 
$$C(z,x,y)=\sum_{T\in\mathcal{T}_b}z^{|T|}x^{|\text{LRMax}(\sigma_T)|}y^{|\text{RLMax}(\sigma_T)|},$$
where $\mathcal{T}_b\coloneqq\mathbb{T}_b\cup \emptyset$ is the set of (possibly empty) binary trees. 
Using the classical decomposition of the binary trees in root, left subtree and right subtree, together with Observation \ref{maxnode}, we obtain the following equation for our generating function,
\begin{equation}
\label{genfctrel}
C(z,x,y)=1+zxy\cdot C(z,x,1)\cdot C(z,1,y),
\end{equation}
where the term $1$ corresponds to the empty tree.

Now setting $x=y=1$ in Equation (\ref{genfctrel}) and $A(z)\coloneqq C(z,1,1),$ we obtain the classical relation for the generating function of Catalan numbers:
$$A(z)=1+z\cdot A(z)^2,$$
that can be algebraically solved to yield
$A(z)=\frac{1-\sqrt{1-4z}}{2z}.$

Substituting $y=1$ in Equation (\ref{genfctrel}) and setting $B(z,x)\coloneqq C(z,x,1)$, we have
$$B(z,x)= 1+ zx\cdot B(z,x)\cdot A(z)$$
and so
$$B(z,x)=\frac{1}{1-zx\cdot A(z)}=\frac{2}{x\sqrt{1-4z}-x+2}.$$
Finally, noting the symmetry $C(z,x,1)=C(z,1,x)$, we obtain
\begin{equation*}
\begin{split}
C(z,x,y)&=1+zxy\bigg(\frac{2}{x\sqrt{1-4z}-x+2}\bigg)\bigg(\frac{2}{y\sqrt{1-4z}-y+2}\bigg)\\
&=1+\frac{4zxy}{(2-x+x\sqrt{1-4z})(2-y+y\sqrt{1-4z})}.
\end{split}
\end{equation*}
Once we have the explicit expression for $C(z,x,y),$ in order to prove that Equation (\ref{probdens}) defines a probability distribution on $\text{Av}^k(231)$, it suffices to notice that
$$C(z,2,2)=\frac{1}{1-4z},$$
and so $[z^n]C(z,2,2)=2^{2n}.$

\section{A small lemma}
\label{small_lemma}
\begin{lem}
	\label{tecn_lemma}
	Let $(\bm{X}_n)_{n\in\Z_{>0}},$ $(\tilde{\bm{X}}_n)_{n\in\Z_{>0}},$ be two sequences of real random variables with values in $[0,1]$ such that, for all $n\in\Z_{>0},$
	$$\bm{X}_n-\tilde{\bm{X}}_n=\bm{o}_P(1),$$
	\emph{i.e.,} $\bm{X}_n-\tilde{\bm{X}}_n\stackrel{P}{\to}0.$ Then, for every real continuous random process $\bm{Y}=(\bm{Y}_t)_{t\in [0,1]}$ (\emph{i.e.,} a random variable with values in $\mathcal{C}([0,1],\mathbb{R})$) we have
	$$\bm{Y}_{\bm{X}_n}-\bm{Y}_{\tilde{\bm{X}}_n}=\bm{o}_P(1).$$
\end{lem}
\begin{proof}
	 Since $\bm{Y}=(\bm{Y}_t)_{t\in [0,1]}$ is a.s.\ continuous on a compact set, it is also a.s.\ uniformly continuous. Therefore the continuity modulus $\omega(\bm{Y},\delta)\coloneqq\sup_{s,t\in[0,1],|s-t|\leq\delta}|\bm{Y}_t-\bm{Y}_s|$ tends a.s.\ to zero as $\delta$ tends to zero.
	 By Skorokhod's representation theorem we can assume that there exist a coupling such that $\bm{X}_n-\tilde{\bm{X}}_n\stackrel{a.s.}{\to}0,$ therefore
	 $$|\bm{Y}_{\bm{X}_n}-\bm{Y}_{\tilde{\bm{X}}_n}|\leq\omega(\bm{Y},\bm{X}_n-\tilde{\bm{X}}_n)\stackrel{a.s.}{\to}0.$$
	 Hence we can conclude that $\bm{Y}_{\bm{X}_n}-\bm{Y}_{\tilde{\bm{X}}_n}=\bm{o}_P(1)$ in the original probability space.
\end{proof}

\section*{Acknowledgements}
The author is very grateful to Mathilde Bouvel and Valentin F\'eray for introducing him to the fantastic world of random permutations. He also thanks them for the constant and stimulating discussions and suggestions.  

The author also warmly thanks Benedikt Stufler for precious suggestions about the terminology for Benjamini--Schramm convergence and for some explanations about local results for random trees.

Finally, he thanks Tommaso Padovan for precious help in the realization of various simulations and the anonymous referee for all his/her precious and useful comments.

This work was completed with the support of the SNF grant number $200021\_172536$, ``Several aspects of the study of non-uniform random permutations".

\bibliographystyle{abbrv}
\bibliography{mybib}

\end{document}